\documentclass[a4paper,11pt]{article}

\usepackage{macros}
\usepackage[normalem]{ulem}
\usepackage{nameref}
\usepackage{zref-xr}
\zxrsetup{toltxlabel}
\title{Finite length for unramified $\GL_2$}
\definecolor{olive}{rgb}{0.5, 0.5, 0.0}

\newcounter{mar-counter}
\renewcommand{\mar}[1]{\stepcounter{mar-counter}\marginpar{\#\arabic{mar-counter}\ \raggedright\tiny #1}}

\usepackage{bigfoot}

\DeclareNewFootnote{AAffil}[alph]

\usepackage{etoolbox}
\makeatletter
\patchcmd\maketitle{\def\@makefnmark{\rlap{\@textsuperscript{\normalfont\@thefnmark}}}}{}{}{}
\makeatother

\makeatletter
\def\thanksAAffil#1{%
  \footnotemarkAAffil\protected@xdef\@thanks{\@thanks%
        \protect\footnotetextAAffil[\the \c@footnoteAAffil]{#1}}%
}
\def\thanksANote#1{%
  \footnotemarkANote%
  \protected@xdef\@thanks{\@thanks%
        \protect\footnotetextANote[\the \c@footnoteANote]{#1}}%
}
\makeatother

\author{Christophe Breuil\thanksAAffil{CNRS, B\^atiment 307, Facult\'e d'Orsay, Universit\'e Paris-Saclay, 91405 Orsay Cedex, France}\\
\and
Florian Herzig\thanksAAffil{Dept.\ of Math., Univ.\ of Toronto, 40 St.\ George St., BA6290, Toronto, ON M5S 2E4, Canada}\\
\and
Yongquan Hu\thanksAAffil{Morningside Center of Mathematics, Academy of Mathematics and Systems Science, Chinese Academy  of \newline
\indent\hspace{1.5mm} Sciences, Beijing 100190, China; University of the Chinese Academy of Sciences, Beijing 100049, China}
\\
\and
Stefano Morra\thanksAAffil{Lab.\ d'Analyse, G\'eom\'etrie, Alg\`ebre, 99 Av.\ Jean Baptiste Cl\'ement, 93430 Villetaneuse, France}\\
\and
Benjamin Schraen\thanksAAffil{Universit\'e Claude-Bernard-Lyon-I, Institut Camille Jordan, 69622 Villeurbanne, France}}

\date{ }

\begin{document}

\maketitle

\begin{abstract}
Let $p$ be a prime number and $K$ a finite unramified extension of $\Qp$. If $p$ is large enough with respect to $[K:\Qp]$ and under mild genericity assumptions, we prove that the admissible smooth representations of $\GL_2(K)$ that occur in Hecke eigenspaces of the mod $p$ cohomology are of finite length. We also prove many new structural results about these representations of $\GL_2(K)$ and their subquotients.
\end{abstract}

\tableofcontents

\section{Introduction}

\subsection{The main results}
\label{sec:results}

Let $p$ be a prime number, $F$ a totally real number field and $D$ a quaternion algebra of center $F$ which is split at all $p$-adic places and at exactly one infinite place. In order to simplify this introduction we assume that $p$ is inert in $F$ (in the text we only need $p$ unramified in $F$) and denote by $v$ the unique $p$-adic place of $F$. To an absolutely irreducible continuous representation $\rbar : {\rm Gal}(\overline F/F)\rightarrow \GL_2(\F)$ (here $\F$ is a sufficiently large finite extension of $\Fp$) and $V^v$ a compact open subgroup of $(D\otimes_F{\mathbb A}_F^{\infty,v})^\times$ (here ${\mathbb A}_F^{\infty,v}$ is the ring of finite prime-to-$v$ ad\`eles of $F$), we associate the admissible smooth representation of $\GL_2(F_v)$ over $\F$: 
\begin{equation}\label{eq:goal}
\pi\defeq \varinjlim_{V_v}\Hom_{{\rm Gal}(\overline F/F)}\!\big(\rbar,H^1_{{\rm \acute et}}(X_{V^vV_v} \times_F \overline F, \F)\big),
\end{equation}
where the inductive limit runs over compact open subgroups $V_v$ of $(D\otimes_FF_v)^\times\cong \GL_2(F_v)$ and $X_{V^vV_v}$ is the smooth projective Shimura curve over $F$ associated to $D$ and $V^vV_v$. Throughout this introduction we fix $\pi$ as in (\ref{eq:goal}) such that $\pi\ne 0$. Recall that, when $F=\Q$ (and $X_{V^vV_v}$ is the compactified modular curve) and under very weak assumptions on $\rbar\vert_{{\rm Gal}(\Qpbar/\Qp)}$, the $\GL_2(\Qp)$-representation $\pi$ has been completely understood for quite some time (see \cite{emerton-local-global}, \cite{CDP}). Unfortunately, this is no longer the case when $F_v\ne \Qp$ despite recent progress (\cite{HuWang2}, \cite{BHHMS1}, \cite{BHHMS2}, \cite{BHHMS3}, \cite{YitongWangGKD}, \cite{YitongWang}). The main aim of the present work is to take a new step in the (long) journey towards the comprehension of the $\GL_2(F_v)$-representation $\pi$ when $F_v\ne \Qp$ by proving that, for $\rbar$ sufficiently generic and under a standard multiplicity one assumption (commonly referred to as ``the minimal case''), $\pi$ is of finite length. 
(The multiplicity one assumption makes the weight cycling much easier and simplifies many computations. Also, assuming $r>1$ would change the length of the $\rbar$-eigenspace $\pi$: see \cite{Lucrezia} in case $\rbar\vert_{\Gal(\o F_v/F_v)}$ is semisimple.)

Under similar assumptions, it was already known that $\pi$ is absolutely irreducible if and only if $\rbar\vert_{\Gal(\o F_v/F_v)}$ is (\cite[Thm.~3.105(i)]{BHHMS2}), and that $\pi$ is of length $3$ when $\rbar\vert_{\Gal(\o F_v/F_v)}$ is reducible and $[F_v:\Qp]=2$ (\cite{HuWang2} for $\rbar\vert_{\Gal(\o F_v/F_v)}$ nonsplit, \cite[Thm.~3.105(ii)]{BHHMS2} for $\rbar\vert_{\Gal(\o F_v/F_v)}$ split\footnote{\cite[Thm.~3.105]{BHHMS2} is stated in the global setting of compact unitary groups but the proof is the same.}). Hence the main contribution of this work is to prove that $\pi$ is of finite length when $\rbar\vert_{\Gal(\o F_v/F_v)}$ is reducible and $[F_v:\Qp]\geq 3$. We also obtain many intermediate and aside results on (the irreducible constituents of) $\pi$.

Let us describe our most important results in more detail. 

We set $K\defeq F_v$, $f\defeq [K:\Qp]$ and $q\defeq p^f$. We denote by $\omega$ the mod $p$ cyclotomic character of $\Gal(\o K/K)$ (that we consider as a character of $K^\times$ via local class field theory, where uniformizers correspond to geometric Frobenius elements), and by $\omega_{f}$, $\omega_{2f}$ Serre's fundamental characters of the inertia subgroup $I_K$ of $\Gal(\o K/K)$ of level $f$, $2f$ respectively. In this introduction, we say that $\rbar$ is \emph{generic} if the following conditions are satisfied:
\begin{enumerate}
\item $\rbar\vert_{\Gal(\overline F/F(\sqrt[p]{1}))}$ is absolutely irreducible;
\item\label{it:wramifies}for $w\!\nmid\! p$ such that either $D$ or $\rbar$ ramifies at $w$, the framed deformation ring of $\rbar\vert_{{\rm Gal}(\overline F_w/F_w)}$ over the Witt vectors $W(\F)$ is formally smooth;
\item\label{it:wisv}$\rbar\vert_{I_{K}}$ is up to twist of form
\begin{equation*}
\begin{pmatrix}\omega_f^{\sum_{j=0}^{f-1} (r_j+1) p^j}&*\\0&1\end{pmatrix}\text{\ with\ }\max\{12,2f+1\} \leq r_j \leq p-\max\{15,2f+4\}
\end{equation*}
or 
\begin{equation*}
\left(\begin{matrix}\omega_{2f}^{\sum_{j=0}^{f-1} (r_j+1) p^j} & \\ & \omega_{2f}^{q(\text{same})}\end{matrix}\right) \text{\ with\ }
  \begin{cases}
    \max\{12,2f+1\} \leq r_j \leq p-\max\{15,2f+4\} \!\!&\!\! j>0 \\ \max\{13,2f+2\} \leq r_0 \leq p-\max\{14,2f+3\}.&
  \end{cases}
\end{equation*}
\end{enumerate}
Note that \ref{it:wisv} implies $p \geq \max\{27,4f+5\}$ and that \ref{it:wramifies} can be made explicit (\cite{Shotton}, \cite[Rk.\ 8.1.1]{BHHMS1}). The bounds on $r_j$ in \ref{it:wisv} are such that all the results mentioned in this introduction hold 
(in the paper many results actually require  weaker bounds, and a few results require stronger bounds). By \cite[Thm.\ 1.9]{BHHMS1} (for $\rbar\vert_{{\rm Gal}(\o K/K)}$ semisimple) and \cite[Thm.\ 6.3(ii)]{YitongWangGKD} (for $\rbar\vert_{{\rm Gal}(\o K/K)}$ non-semisimple) for $\rbar$ generic there is a unique integer $r\geq 1$ (the ``multiplicity'') such that, for any (absolutely) irreducible representation $\sigma$ of $\GL_2(\cO_K)$ over $\F$, we have $\dim_{\F}\Hom_{\GL_2(\cO_K)}(\sigma, \pi)\in \{0, r\}$ (the notation $\o r$ and $r$ is somewhat unfortunate but is consistent with \cite[\S~8]{BHHMS1}).

In the sequel we let $\rhobar\defeq \rbar^\vee\vert_{{\rm Gal}(\o K/K)}$, where $\rbar^\vee$ is the dual of $\rbar$. 

If $\pi_1$ and $\pi_2$ are representations of a group, we denote by $\!\begin{xy} (0,0)*+{\pi_1}="a"; (10,0)*+{\pi_2}="b"; {\ar@{-}"a";"b"}\end{xy}\!$ an arbitrary \emph{nonsplit} extension of $\pi_2$ by $\pi_1$ (so $\pi_1$ is a subrepresentation and $\pi_2$ is a quotient). We say a finite length representation is \emph{uniserial} if it has a unique composition series, in which case we write $\!\begin{xy} (0,0)*+{\pi_1}="a"; (10,0)*+{\pi_2}="b"; (20,0)*+{\pi_3}="c"; (30,0)*+{\cdots}="d"; {\ar@{-}"a";"b"}; {\ar@{-}"b";"c"}; {\ar@{-}"c";"d"}\end{xy}\!$, where $\pi_i$ are the (irreducible) graded pieces. Finally we let $B(K)$ be the subgroup of upper triangular matrices in $\GL_2(K)$. 

\begin{thm}\label{thm:mainintro}
Assume that $\rbar$ is generic and that $r=1$.
\begin{enumerate}
\item\label{it:irrintro} If $\rhobar$ is irreducible then $\pi$ is irreducible supersingular.
\item\label{it:splitintro} If $\rhobar$ is split, i.e.~$\rhobar\cong \begin{pmatrix}\chi_{1} &0\\0 &\chi_2\end{pmatrix}$, then
\[\pi\cong \Ind_{B(K)}^{\GL_2(K)}(\chi_2\otimes \chi_1\omega^{-1})\oplus \pi' \oplus \Ind_{B(K)}^{\GL_2(K)}(\chi_1\otimes \chi_2\omega^{-1}),\]
where $\pi'=0$ if $K=\Qp$ and $\pi'$ has length $\in \{1,\dots,f-1\}$ with {pairwise distinct  constituents all of which are supersingular} if $K\ne \Qp$.
\item\label{it:nonsplitintro} If $\rhobar$ is nonsplit, i.e.~$\rhobar\cong \begin{pmatrix}\chi_{1} &*\\0 &\chi_2\end{pmatrix}$ with $*\ne 0$, then
\[\pi \cong \Big(\begin{xy} (0,0)*+{\Ind_{B(K)}^{\GL_2(K)}(\chi_2\otimes \chi_1\omega^{-1})}="a"; (30,0)*+{\pi'}="b"; (60,0)*+{\Ind_{B(K)}^{\GL_2(K)}(\chi_1\otimes \chi_2\omega^{-1})}="c"; {\ar@{-}"a";"b"}; {\ar@{-}"b";"c"}\end{xy}\Big),\]
where $\pi'=0$ if $K=\Qp$ and $\pi'$ is uniserial of length $\in \{1,\dots,f-1\}$ with  {pairwise distinct  constituents all of which are supersingular} if $K\ne \Qp$.
\end{enumerate}
\end{thm}

Part \ref{it:irrintro} was known (\cite[Thm.~3.105(i)]{BHHMS2}, as already mentioned), \ref{it:splitintro} easily follows from Theorem \ref{thm:length:f+1}(i) with the first statement of \cite[Thm.\ 1.17]{BHHMS2} and from Corollary \ref{cor:finite-length}(iv), and \ref{it:nonsplitintro} follows from Theorem \ref{thm:fin-length-nonsplit}(ii) and Corollary \ref{cor:pi-mult-free1}.

Theorem \ref{thm:mainintro} implies that $\pi$ is of finite length and multiplicity free. It is expected that $\pi'$ in Theorem \ref{thm:mainintro}\ref{it:splitintro}, \ref{it:nonsplitintro} always has length $f-1$ (see \cite[p.\ 107]{BP}) but we only know this when $f=2$ (in fact we do not have an example of a $\pi'$ of length $\geq 2$). Note also that, although one can optimistically hope that $\pi'$ only depends on $\rhobar$ and that $\pi'$ in Theorem \ref{thm:mainintro}\ref{it:splitintro} is the semisimplification of $\pi'$ in Theorem \ref{thm:mainintro}\ref{it:nonsplitintro}, at present we know none of these statements when $f>1$, even for $f=2$.

Nevertheless we can prove several results on the irreducible constituents of $\pi$. Let $I$ (resp.~$I_1$) be the subgroup of $\GL_2(\cO_K)$ of matrices which are upper triangular modulo $p$ (resp.~upper unipotent modulo $p$) and $K_1\cong 1+p{\rm M}_2(\cO_K) \subset I_1$ be the subgroup of matrices which are trivial modulo $p$. Let $Z_1\cong 1+p\cO_K$ be the center of $I_1$ (or $K_1$). We will extensively use the Iwasawa algebra $\Lambda\defeq \F\bbra{I_1/Z_1}$ which is a (noncommutative) noetherian local ring of Krull dimension $3f$. We denote by $\m$ its maximal ideal. Since $\pi$ has a central character, $\pi$ and any of its subquotients are $\Lambda$-modules, and likewise for their linear duals. Since $\pi$ is admissible, the latter are moreover finitely generated $\Lambda$-modules. Recall that a nonzero finitely generated $\Lambda$-module $M$ is Cohen--Macaulay of grade $c\geq 0$ if $\Ext^i_{\Lambda}(M, \Lambda)$ is nonzero if and only if $i=c$.

\begin{thm}\label{thm:mainbisintro}
Assume that $\rbar$ is generic, that $r=1$ and that $\rhobar$ is semisimple.
\begin{enumerate}
\item\label{it:cohenintro}
The linear dual $\Hom_{\F}(\pi',\F)$ of any nonzero subquotient $\pi'$ of $\pi$ is a Cohen--Macaulay $\Lambda$-module of grade $2f$.
\item\label{it:socleintro}
Any subquotient of $\pi$ is generated by its $\GL_2(\cO_K)$-socle.
\item\label{it:phigammaintro}
For any subquotient $\pi'$ of $\pi$ we have
\[\dim_{\F\ppar{X}}D_\xi^\vee({\pi'})=|\JH(\soc_{\GL_2(\cO_K)}(\pi'))|,\]
where $D_\xi^\vee({\pi'})$ is the cyclotomic $(\varphi,\Gamma)$-module associated to ${\pi'}$ in \cite[\S~2.1.1]{BHHMS2} and $\JH$ means the set of Jordan--H\"older (or irreducible) constituents.
\item\label{it:exactintro}
For any subrepresentations $\pi_1\subseteq \pi_2$ of $\pi$ we have a split exact sequence of $\GL_2(\cO_K)$-repre\-sen\-tations 
\[0\ra \soc_{\GL_2(\cO_K)}(\pi_1)\ra \soc_{\GL_2(\cO_K)}(\pi_2)\ra \soc_{\GL_2(\cO_K)}(\pi_2/\pi_1)\ra0.\]
\item\label{it:exactintroI}
For any subrepresentations $\pi_1\subseteq \pi_2$ of $\pi$ and any $n\geq 1$ we have an exact sequence of $I$-representations
\[0\ra \pi_1[\m^n]\rightarrow \pi_2[\m^n]\rightarrow(\pi_2/\pi_1)[\m^n]\rightarrow0,\]
which is split for $n\leq \max\{6,f+1\}$.
\end{enumerate}
\end{thm}

Note first that for $\pi$ itself part \ref{it:cohenintro} was known using \cite[Prop.\ A.8]{HuWang2} (without assuming $\rhobar$ semisimple) and part \ref{it:socleintro} was known by \cite[Thm.~1.14]{BHHMS2}. Moreover \ref{it:phigammaintro} was known for subrepresentations $\pi_1$ of $\pi$ by \cite[Prop.~3.87(ii)]{BHHMS2}. In particular Theorem \ref{thm:mainbisintro} was already known for $\rhobar$ irreducible (as $\pi$ is then also irreducible), and thus the main novelty in Theorem \ref{thm:mainbisintro} is that we obtain nontrivial results for \emph{subquotients} of $\pi$ (when $\rhobar$ is reducible).

When $\rhobar$ is split reducible, \ref{it:cohenintro} is contained in Corollary \ref{cor:finite-length}(ii), \ref{it:socleintro} is Corollary \ref{cor:finite-length}(iii), \ref{it:phigammaintro} is contained in Corollary \ref{cor:finite-length}(i) and  \ref{it:exactintro} is Lemma \ref{lem:soc-exact}. Finally \ref{it:exactintroI} is Corollary \ref{cor:split-In}. 
The splitness of the exact sequences in \ref{it:exactintro} and in \ref{it:exactintroI} for $n\leq \max\{6,f+1\}$ can be seen as (very weak) evidence for the hope that $\pi$ is semisimple when $\rhobar$ is.

When $\rhobar$ is non-semisimple, we have the following version of Theorem \ref{thm:mainbisintro}:

\begin{thm}\label{thm:mainterintro}
Assume that $\rbar$ is generic, that $r=1$ and that $\rhobar$ is non-semisimple (reducible).
\begin{enumerate}
\item\label{it:cohenbisintro}
The linear dual of any nonzero subquotient of $\pi$ is a Cohen--Macaulay $\Lambda$-module of grade $2f$.
\item\label{it:K1intro}
Any subquotient of $\pi$ is generated by its $K_1$-invariants.
\end{enumerate}
\end{thm}

The proofs in the non-semisimple case are significantly harder and usually much more technical than in the split case. Part \ref{it:cohenbisintro} is contained in Corollary \ref{cor:subquot} and part \ref{it:K1intro} is Theorem \ref{thm:fin-length-nonsplit}(i).

Theorem \ref{thm:mainterintro} is shorter than Theorem \ref{thm:mainbisintro} because, in the nonsplit case, if $\pi_1\subseteq \pi_2$ are nonzero subrepresentations of $\pi$ the maps $\pi_2^{I_1}\rightarrow(\pi_2/\pi_1)^{I_1}$ and $\pi_2^{K_1}\rightarrow(\pi_2/\pi_1)^{K_1}$ are not surjective in general (even for $f=1$). 
{Nonetheless, in \cite{BHHMS5} we will completely determine the (semisimple) $I$-representation $(\pi_2/\pi_1)^{I_1}$   and the $\GL_2(\Fq)$-representation $(\pi_2/\pi_1)^{K_1}$.}  We will also determine $\dim_{\F\ppar{X}}D_\xi^\vee(\pi_2/\pi_1)$.

Under the same assumptions ($\rbar$ generic, $r=1$) we prove several other results that are not stated above. 
For instance, \emph{just assuming $\rbar$ generic}, we completely determine $\gr_\m(\pi^{\vee})$ \ as \ a \ graded \ $\gr_\m(\Lambda)$-module, \ where \ $\pi^{\vee}\defeq \Hom_{\F}(\pi,\F)$ denotes the linear dual of $\pi$ which is a finitely generated $\Lambda$-module,  $\gr_\m(\Lambda)\defeq \bigoplus_{n\geq 0}\m^n/\m^{n+1}$ \ and \ $\gr_\m(\pi^{\vee})\defeq \bigoplus_{n\geq 0}\m^n\pi^\vee/\m^{n+1}\pi^\vee$ (see Theorem \ref{thm:CMC} below). This is a key result. Indeed, on the one hand it makes it possible to determine $\gr_\m((\pi_2/\pi_1)^{\vee})$ for any subrepresentations $\pi_1\subseteq \pi_2$ of $\pi$ (Corollary \ref{cor:finite-length}(ii) for $\rhobar$ split, \cite{BHHMS5} for $\rhobar$ nonsplit with suitable genericity).
On the other hand, and most crucially, knowing $\gr_\m(\pi^{\vee})$ is the starting point of \emph{all} the important proofs of this work as we explain now.

\subsection{Some sketches of proofs}\label{sec:sketchintro}

One important question left open in \cite[\S~3.3.2]{BHHMS2} was the precise structure of the graded $\gr_\m(\Lambda)$-module $\gr_\m(\pi^{\vee})$ (see \cite[Rk.~3.72(i)]{BHHMS2}). We answer this question in the next theorem. We need more notation. Recall from \cite[\S~3.1]{BHHMS2} that $\gr_{\m}(\Lambda)\cong \bigotimes_{j\in\{0,\dots,f-1\}}\F\langle y_j,z_j,h_j\rangle$ with relations $[y_j,z_j]=h_j$, $[h_j,z_i]=[y_i,h_j]=0$ for all $i,j\in\{0,\dots,f-1\}$. We let \[R \defeq \gr_{\m}(\Lambda)/(h_j : 0 \le j \le f-1) \cong \F[y_j,z_j : 0 \le j \le f-1]\]
which is a (graded) commutative polynomial ring. We let $H\defeq \begin{pmatrix}\F_q^\times&0\\0&\F_q^\times\end{pmatrix}\cong I/I_1$, which naturally acts on $\Lambda$, $\gr_\m(\Lambda)$ and $R$. Recall that the irreducible continuous representations of $I$ over $\F$ factor as characters $\chi:H\rightarrow \F^\times$. In \cite[Def.~3.57]{BHHMS2} to each $\chi\in \JH(\pi^{I_1})$ we associated an ideal $\mathfrak{a}(\chi)$ of $R$ (containing $y_jz_j$ for all $j\in\{0,\dots,f-1\}$) which is denoted by $\fa(\lambda)$ in the text and recalled in (\ref{eq:id:al}) below.

\begin{thm}[Theorem \ref{thm:CMC}]\label{thm:CMCintro}
Assume that $\rbar$ is generic. 
\begin{enumerate}
\item\label{it:grisointro} We have an isomorphism of graded $\gr_{\m}(\Lambda)$-modules with compatible $H$-action
\[\gr_\m(\pi^{\vee})\cong \Bigg(\bigoplus_{\chi\in\JH(\pi^{I_1})}\chi^{-1}\otimes_{\F} \frac{R}{\mathfrak{a}(\chi)}\Bigg)^{\oplus r}.\]
\item\label{it:dualintro} The $\gr_{\m}(\Lambda)$-module $\gr_\m(\pi^{\vee})$ is Cohen--Macaulay of grade $2f$.
\end{enumerate}
\end{thm}

In particular the graded $\gr_{\m}(\Lambda)$-module $\gr_\m(\pi^{\vee})$ {together with its compatible $H$-action} is \emph{local}, i.e.~depends only on $\rhobar$, and even just on $\rhobar\vert_{I_K}$. We remark that Theorem \ref{thm:CMCintro} allows us to compute the entire Hilbert polynomial of $\gr_{\m}(\pi^{\vee})$ (cf.~\cite{BHHMS5}).
Note that, although we know the $\gr_{\m}(\Lambda)$-module $\gr_\m(\pi^{\vee})$ thanks to Theorem \ref{thm:CMCintro}\ref{it:grisointro}, we still do not understand the $\Lambda$-module $(\pi\vert_I)^\vee$.  

We sketch the proof of Theorem \ref{thm:CMCintro} (which is given in \S~\ref{sec:cohen-macaulay}, especially \S~\ref{sec:proof-theorem}). Denote by $N$ the $\gr_{\m}(\Lambda)$-module on the right-hand side of \ref{it:grisointro}. First \ref{it:dualintro} follows from \ref{it:grisointro}, since $N$ is Cohen--Macaulay by a direct computation, hence the main issue is \ref{it:grisointro}. If $M$ is any finitely generated $R$-module which is killed by the ideal $(y_jz_j : 0\leq j\leq f-1)$ of $R$ (for instance $N$), we define its characteristic cycle (\cite[Def.~3.79]{BHHMS2})
\begin{equation}\label{eq:cycleintro}
\mathcal{Z}(M)\defeq \sum_{\q}{\rm length}(M_{\q})[\q]\ \in \bigoplus_{\q}\Z[\q],
\end{equation}
where $\q$ runs through the minimal prime ideals of $R/(y_jz_j : 0\leq j\leq f-1)$.
As $N$ is Cohen--Macaulay, any nonzero $\gr_{\m}(\Lambda)$-submodule of $N$ has a nonzero cycle (i.e.~$N$ is pure). Since by \cite[Thm.~3.67]{BHHMS2} we already have a surjection of graded $\gr_{\m}(\Lambda)$-modules $N\twoheadrightarrow \gr_\m(\pi^{\vee})$ (which implies $\mathcal{Z}(N)\geq \mathcal{Z}(\gr_\m(\pi^{\vee}))$ in $\bigoplus_{\q}\Z[\q]$), to prove \ref{it:grisointro} it is enough to prove $\mathcal{Z}(N)= \mathcal{Z}(\gr_\m(\pi^{\vee}))$, as $\mathcal{Z}(-)$ is additive on short exact sequences (\cite[Lemma 3.80]{BHHMS2}). To show this, we construct a resolution of the $\Lambda$-module $(\pi\vert_I)^\vee$ by a complex of filtered $\Lambda$-modules $P_\bullet$ with compatible $H$-action such that the associated complex $\gr(P_\bullet)$ of $\gr_{\m}(\Lambda)$-modules satisfies $H_0(\gr(P_\bullet))\cong N$ and $H_1(\gr(P_\bullet))=0$. Such a filtered complex gives rise to a spectral sequence $E_i^{s}\Longrightarrow H_i(P_{\bullet})$ for $i,s\geq 0$ (\cite[\S~III.1]{LiOy}) and using $H_1(\gr(P_\bullet))=0$ we prove that $E_0^\infty=E_0^1$. Since $E_0^1=H_0(\gr(P_\bullet))\cong N$ and $E_0^\infty\cong \gr(\pi^\vee)$, where $\gr(\pi^\vee)$ is here computed for the quotient filtration on $\pi^\vee$ induced by the surjection $P_0\twoheadrightarrow \pi^\vee$, we deduce $N\cong \gr(\pi^\vee)$, which implies $\mathcal{Z}(N)= \mathcal{Z}(\gr(\pi^{\vee}))$. But we have $\mathcal{Z}(\gr(\pi^{\vee}))= \mathcal{Z}(\gr_\m(\pi^{\vee}))$ by \cite[Lemma 3.81]{BHHMS2}, and thus $\mathcal{Z}(N)=\mathcal{Z}(\gr_\m(\pi^{\vee}))$. The construction of $P_\bullet$ with its properties is quite involved and in particular crucially uses the following result (where the $\Ext^i_{I/Z_1}$ are computed in the category of smooth representations of $I/Z_1$ over $\F$).

\begin{prop}[\S~\ref{sec:verify-assumpt-iv}]
For any smooth character $\chi:I\ra \F^{\times}$ and any $i \ge 0$, $\Ext^i_{I/Z_1}(\chi,\pi)\neq0 $ only if $\chi\in \JH(\pi^{I_1})$, in which case $\dim_{\F}\Ext^i_{I/Z_1}(\chi,\pi)=\binom{2f}{i}r$.
\end{prop}

Theorem \ref{thm:CMCintro} turns out to be a crucial ingredient in the proof that $\pi$ is of finite length when $r=1$ and $\rhobar$ is reducible. We assume these two hypothesis from now on, and we present below a unified sketch of proof in the two cases $\rhobar$ split and $\rhobar$ nonsplit, though in the text we found it preferable to separate the two cases (mainly because the nonsplit case is much more technical).

We fix a nonzero subrepresentation $\pi_1\subseteq \pi$ and let $\pi_2\defeq \pi/\pi_1$. Hence we have an exact sequence of $\Lambda$-modules with $H$-action $0\rightarrow \pi_2^\vee \rightarrow \pi^\vee \rightarrow \pi_1^\vee\rightarrow 0$. The $\m$-adic filtration on $\pi^\vee$ induces a filtration on $\pi_2^\vee$ and we denote by $\gr(\pi_2^\vee)$ the associated $\gr_{\m}(\Lambda)$-module. Just like the definition of the $\gr_{\m}(\Lambda)$-module $N$ in Theorem \ref{thm:CMCintro}\ref{it:grisointro} only uses the $H$-representation $\pi^{I_1}$ (and \emph{a fortiori} only the $\GL_2(\Fq)$-representation $\pi^{K_1}$), we define an explicit quotient $N_1$ of $N$ which only depends on the $\GL_2(\Fq)$-representation $\pi_1^{K_1}$. In the split case one has
\begin{equation}\label{eq:splitN1}
N_1=\bigoplus_{\chi\in\JH(\pi_1^{I_1})}\chi^{-1}\otimes_{\F} \frac{R}{\mathfrak{a}(\chi)},
\end{equation}
in particular $N_1$ is then a direct summand of $N$ and only depends on the $H$-representation $\pi_1^{I_1}$, but this is no longer true in the nonsplit case if $\pi_1\ne \pi$ (see Step $2$ in the proof of Proposition \ref{prop:nonsplit-I1} together with (\ref{eq:fa-1}) and Definition \ref{def:IJideals}). Defining $N_2\defeq \ker(N\twoheadrightarrow N_1)$, we prove that there  is a commutative diagram with exact rows of graded $\gr_{\m}(\Lambda)$-modules (see Step $1$ in the proof of Proposition \ref{prop:split-I1} for $\rhobar$ split, Step $2$ in the proof of Proposition \ref{prop:nonsplit-I1} for $\rhobar$ nonsplit):
\begin{equation}\label{eq:diagramintro}
\begin{gathered}
  \xymatrix{0\ar[r]& \gr(\pi_2^{\vee})\ar[r]& \gr_{\m}(\pi^{\vee})\ar[r]& \gr_{\m}(\pi_1^{\vee})\ar[r]&0 \\
  0\ar[r]& N_2\ar[r]\ar@{^{(}->}[u]& N\ar[r]\ar[u]^\cong & N_1\ar[r]\ar@{->>}[u]& 0}
  \end{gathered}
\end{equation}
with injective (resp.\ surjective) vertical map on the left (resp.\ right) and where the middle isomorphism is Theorem \ref{thm:CMCintro}\ref{it:grisointro}.

The next step is the following theorem:

\begin{thm}[Proposition \ref{prop:split-I1}, Proposition \ref{prop:nonsplit-I1}]\label{thm:isointro}
The left vertical injection in (\ref{eq:diagramintro}), hence also the right vertical surjection, are isomorphisms. In particular $\gr_{\m}(\pi_1^{\vee})$, $\gr(\pi_2^{\vee})$ are Cohen--Macaulay $\gr_{\m}(\Lambda)$-modules of grade $2f$, and $\pi_1^{\vee}$, $\pi_2^{\vee}$ are Cohen--Macaulay $\Lambda$-modules of grade $2f$.
\end{thm}

We sketch the proof of Theorem \ref{thm:isointro}.

The \ Cohen--Macaulayness \ of \ $\pi_1^{\vee}$, \ $\pi_2^{\vee}$ \ follows \ from \ the \ one \ of \ $\gr_{\m}(\pi_1^{\vee})$, \ $\gr(\pi_2^{\vee})$ (\cite[Prop.~III.2.2.4]{LiOy}), which itself follows from the first statement of Theorem \ref{thm:isointro} as $N_1$, $N_2$ can be checked to be Cohen--Macaulay $\gr_{\m}(\Lambda)$-modules. Note that, by d\'evissage and since $\Lambda$ is Auslander regular, one then deduces from \cite[Cor.~III.2.1.6]{LiOy} that the linear dual of \emph{any} subquotient of $\pi^\vee$ is Cohen--Macaulay of grade $2f$. In particular this proves Theorem \ref{thm:mainbisintro}\ref{it:cohenintro} and Theorem \ref{thm:mainterintro}\ref{it:cohenbisintro}.

Hence it is enough to prove $N_1 \congto \gr_{\m}(\pi_1^{\vee})$. Since, just like $N$, the $\gr_{\m}(\Lambda)$-module $N_1$ is pure, by the same argument as for $N$ (see the sentences below (\ref{eq:cycleintro})) it is enough to prove that $\mathcal{Z}(N_1)=\mathcal{Z}(\gr_\m(\pi_1^{\vee}))$, or equivalently by diagram (\ref{eq:diagramintro}) that $\mathcal{Z}(N_2)=\mathcal{Z}(\gr(\pi_2^{\vee}))$.

We then use the essential self-duality of $\pi$ (\cite[Thm.~8.2]{HuWang2} with \cite[Thm.~8.4.1]{BHHMS1} and \cite[Thm.~6.3(i)]{YitongWangGKD}): there is a $\GL_2(K)$-equivariant isomorphism ${\rm Ext}^{2f}_{\Lambda}(\pi^{\vee},\Lambda)\cong \pi^{\vee}\otimes_{\F} (\det(\rhobar)\omega^{-1})$, \ where \ ${\rm Ext}^{2f}_{\Lambda}(\pi^{\vee},\Lambda)$ \ is \ endowed \ with \ the \ action \ of \ $\GL_2(K)$ \ defined \ in \ \cite[Prop.~3.2]{Ko}. Then we can define $\widetilde \pi_2\subseteq \pi$ as the unique $\GL_2(K)$-subrepresentation such that
\[\widetilde \pi_2^{\vee}= \im\Big\{{\rm Ext}^{2f}_{\Lambda}(\pi^{\vee},\Lambda)\ra {\rm Ext}^{2f}_{\Lambda}(\pi_2^{\vee},\Lambda)\Big\}\otimes_{\F}(\det(\rhobar)^{-1}\omega).\]
Since $\widetilde \pi_2$ is a subrepresentation of $\pi$, we can define a surjection of $\gr_{\m}(\Lambda)$-modules
\[\widetilde N_2\twoheadrightarrow \gr_{\m}(\widetilde \pi_2^{\vee})\]
analogous to $N_1\twoheadrightarrow \gr_{\m}(\pi_1^{\vee})$, where $\widetilde N_2$ again only depends on the $\GL_2(\Fq)$-representation $\widetilde \pi_2^{K_1}$. In particular $\mathcal{Z}(\widetilde N_2)\geq \mathcal{Z}(\gr_\m(\widetilde\pi_2^{\vee}))$. Note that by the same argument as in the proof of \cite[Prop.~3.87(iii)]{BHHMS2} we have $\mathcal{Z}(\gr_\m(\widetilde\pi_2^{\vee}))= \mathcal{Z}(\gr(\pi_2^{\vee}))$. Since $\mathcal{Z}(\gr(\pi_2^{\vee}))\geq \mathcal{Z}(N_2)$ by the left injection in (\ref{eq:diagramintro}), we deduce
\[\mathcal{Z}(\widetilde N_2)\geq \mathcal{Z}(\gr_\m(\widetilde\pi_2^{\vee}))= \mathcal{Z}(\gr(\pi_2^{\vee}))\geq \mathcal{Z}(N_2)\]
and hence it is enough to prove $\mathcal{Z}(\widetilde N_2)=\mathcal{Z}(N_2)$. 

The equality $\mathcal{Z}(\widetilde N_2)= \mathcal{Z}(N_2)$ is the heart of the proof of Theorem \ref{thm:isointro} and is particularly subtle in the nonsplit case. In both cases (split or nonsplit) it boils down to determining the $\GL_2(\Fq)$-representation $\widetilde \pi_2^{K_1}$ from the $\GL_2(\Fq)$-representation $\pi_1^{K_1}$. For that, we do not know any proof that avoids $(\vp,\Gamma)$-modules. We have the formula
\begin{equation}\label{eq:dimintro}
\dim_{\F\ppar{X}}D_{\xi}^{\vee}(\widetilde{\pi}_2)
= \dim_{\F\ppar{X}}D_{\xi}^{\vee}(\pi_2) = \dim_{\F\ppar{X}}D_{\xi}^{\vee}(\pi) - \dim_{\F\ppar{X}}D_{\xi}^{\vee}(\pi_1),
\end{equation}
where the first equality follows from $\mathcal{Z}(\gr_\m(\widetilde\pi_2^{\vee}))= \mathcal{Z}(\gr(\pi_2^{\vee}))$ with \cite[Prop.~3.87(i)]{BHHMS2} and the second from the exactness of the functor $D_{\xi}^{\vee}$ (\cite[Thm.~3.29]{BHHMS2}). In the split case, using the equalities
\begin{eqnarray*}
\dim_{\F\ppar{X}}D_{\xi}^{\vee}(\pi) &=&2^f,\\
\dim_{\F\ppar{X}}D_{\xi}^{\vee}(\widetilde{\pi}_2)&=&|\JH(\soc_{\GL_2(\cO_K)}(\widetilde \pi_2))|,\\
\dim_{\F\ppar{X}}D_\xi^\vee({\pi_1})&=&|\JH(\soc_{\GL_2(\cO_K)}(\pi_1))|
\end{eqnarray*}
(where \ the \ first \ follows \ from \ \cite[Thm.~1.7]{BHHMS2} \ and \ where \ the \ other \ two \ are \cite[Prop.~3.87(ii)]{BHHMS2}), we manage starting from (\ref{eq:dimintro}) to determine $\soc_{\GL_2(\cO_K)}(\widetilde \pi_2)$, hence $\widetilde \pi_2^{K_1}$ (using the proof of \cite[Thm.~19.10]{BP}), hence $\widetilde N_2$, and finally check that $\mathcal{Z}(\widetilde N_2)=\mathcal{Z}(N_2)$. In the nonsplit case using the (much harder) equalities
\begin{eqnarray*}
\dim_{\F\ppar{X}}D_{\xi}^{\vee}(\pi) &=&2^f,\\
\dim_{\F\ppar{X}}D_\xi^\vee(\widetilde\pi_2)&=&|\JH(\widetilde\pi_2^{K_1}) \cap W(\rhobar^\ss)|,\\
\dim_{\F\ppar{X}}D_\xi^\vee({\pi_1})&=&|\JH(\pi_1^{K_1}) \cap W(\rhobar^\ss)|
\end{eqnarray*}
(which all follow from \cite[Thm.~1.2]{YitongWang}) with (\ref{eq:dimintro}) (and Theorem \ref{thm:conj2} in the text applied to both $\pi_1$, $\widetilde \pi_2$), we can again determine $\widetilde \pi_2^{K_1}$ and once more check $\mathcal{Z}(\widetilde N_2)=\mathcal{Z}(N_2)$.

We now sketch the proof that $\pi$ is of finite length (for $\rhobar$ reducible) using Theorem \ref{thm:isointro}. 

Let $\pi_1\subseteq \pi$ be a nonzero subrepresentation, and let $\pi'_1\subseteq \pi_1$ be the $\GL_2(K)$-subrepresentation generated by $\soc_{\GL_2(\cO_K)}(\pi_1)$ if $\rhobar$ is split, by $\pi_1^{K_1}$ if $\rhobar$ is nonsplit. We then have ${\pi'_1}^{K_1}\congto \pi_1^{K_1}$ in both cases (using the proof of \cite[Thm.~19.10]{BP} in the split case). The $\gr_{\m}(\Lambda)$-module $N_1$ in (\ref{eq:diagramintro}) is the same for both $\pi_1$ and $\pi'_1$ since it only depends on the $\GL_2(\Fq)$-representation ${\pi'_1}^{K_1}\cong \pi_1^{K_1}$. By Theorem \ref{thm:isointro} we deduce that the natural surjection $N_1 \congto \gr_{\m}(\pi_1^{\vee}) \onto \gr_{\m}(\pi_1'^{\vee})$ is an isomorphism, in particular $\m^n\pi_1^\vee/\m^{n+1}\pi_1^\vee\congto \m^n\pi_1'^\vee/\m^{n+1}\pi_1'^\vee$ for all $n\geq 0$, hence by d\'evissage $\pi_1^\vee/\m^{n+1}\pi_1^\vee\congto \pi_1'^\vee/\m^{n+1}\pi_1'^\vee$ for $n\geq 0$, hence $\pi_1^\vee\congto \pi_1'^\vee$ or equivalently $\pi'_1\congto \pi_1$. 

This first implies that $\pi_1$ is generated by its $\GL_2(\cO_K)$-socle if $\rhobar$ is split, by its $K_1$-invariants if $\rhobar$ is nonsplit (since $\pi'_1$ is). As the quotient of a $\GL_2(K)$-representation generated by its $\GL_2(\cO_K)$-socle (resp.~its $K_1$-invariants) is \emph{a fortiori} also generated by its $\GL_2(\cO_K)$-socle (resp.~its $K_1$-invariants), we have proven Theorem \ref{thm:mainbisintro}\ref{it:socleintro} and Theorem \ref{thm:mainterintro}\ref{it:K1intro}. 

We then obtain that $\pi$ is of finite length, as there are only finitely many $\GL_2(\Fq)$-subrepresenta\-tions $\pi_1^{K_1}$ inside the $\GL_2(\Fq)$-representation $\pi^{K_1}$ (recall the latter is explicitly known and only depends on $\rhobar\vert_{I_K}$, see \cite{HuWang, LMS} for $\rhobar$ split, \cite{DanWild} for $\rhobar$ nonsplit). A more precise calculation inside $\pi^{K_1}$ gives the more precise statements in Theorem \ref{thm:mainintro}\ref{it:splitintro}, \ref{it:nonsplitintro}, though the multiplicity freeness in the nonsplit case is more involved, see Corollary \ref{cor:pi-mult-free1}.

So far we have briefly gone over the proofs of Theorem \ref{thm:mainintro}, of Theorem \ref{thm:mainbisintro}\ref{it:cohenintro}, \ref{it:socleintro} and of Theorem \ref{thm:mainterintro}\ref{it:cohenbisintro}, \ref{it:K1intro}. We now sketch the proofs of Theorem \ref{thm:mainbisintro}\ref{it:phigammaintro}, \ref{it:exactintro}, \ref{it:exactintroI}. 

Since in the split case $N_1$ in (\ref{eq:splitN1}) is a direct summand of $N$, Theorem \ref{thm:isointro} implies that the exact sequence of graded $\gr_{\m}(\Lambda)$-modules $0\rightarrow \gr(\pi_2^{\vee})\rightarrow \gr_{\m}(\pi^{\vee})\rightarrow \gr_{\m}(\pi_1^{\vee})\rightarrow 0$ in (\ref{eq:diagramintro}) is split. Then by a dimension count we deduce that the map $\pi[\m^n]\rightarrow(\pi/\pi_1)[\m^n]$ is surjective for all $n\geq 0$. It is then not difficult to deduce the exactness in Theorem \ref{thm:mainbisintro}\ref{it:exactintroI}. The splitness in \emph{loc.~cit.}~for $n\leq \max\{6,f+2\}$ comes from the following description of the $I$-representation $\pi[\m^n]$ for such $n$ (see Lemma \ref{lem:isom-modcI}):
\begin{equation}\label{eq:Iisointro}
\pi[\m^n]\cong \bigoplus_{\chi\in \JH(\pi^{I_1})}\tau_{\chi}^{(n)},
\end{equation}
where the $I$-representations $\tau_{\chi}^{(n)}$ (denoted $\tau_\lambda^{(n)}$ in the text) are defined in Lemma \ref{lem:tau-embed}. From (\ref{eq:Iisointro}) one deduces $\pi_1[\m^n]\cong \bigoplus_{\chi\in \JH(\pi_1^{I_1})}\tau_{\chi}^{(n)}$ -- whence the splitting -- using the isomorphism $N_1\congto \gr_\m(\pi_1^{\vee})$ in Theorem \ref{thm:isointro} together with (\ref{eq:splitN1}) (see the end of the proof of Corollary \ref{cor:split-In}). 

Then the first exact sequence in Theorem \ref{thm:mainbisintro}\ref{it:exactintro} easily follows from the exact sequence in Theorem \ref{thm:mainbisintro}\ref{it:exactintroI} applied with $n=1$ (see Lemma \ref{lem:soc-exact}). Note that this first exact sequence implies Theorem \ref{thm:mainbisintro}\ref{it:phigammaintro} by the exactness of $D_{\xi}^{\vee}$ (\cite[Thm.~3.29]{BHHMS2}) and the case of subrepresentations (\cite[Prop.~3.87(ii)]{BHHMS2}). The second exact sequence in Theorem \ref{thm:mainbisintro}\ref{it:exactintro} and its splitness both follow from the first using, as we have seen with $\widetilde\pi_2$ above, that if we know $\soc_{\GL_2(\cO_K)}(\pi_1)$ for a subrepresentation $\pi_1\subseteq \pi$ when $\rhobar$ is split we also know $\pi_1^{K_1}$, and moreover that $\pi_1^{K_1}$ is a direct summand of $\pi^{K_1}$.

\textbf{Acknowledgements}: The results of this work in the non-semisimple case use as a key input results of Yitong Wang \cite{YitongWang}.
 We are thankful to the referee for helpful comments. 
Y.\;H.\ is \ partially \ supported by \ National \ Key \ R$\&$D \ Program \ of \ China \ 2020YFA0712600, National Natural Science Foundation of China Grants 12288201 and 12425103, National Center for Mathematics and Interdisciplinary Sciences and Hua Loo-Keng Key Laboratory of Mathematics, Chinese Academy of Sciences. F.\;H.\ is partially supported by an NSERC grant. S.\;M.\ and B.\;S.\ are partially supported by the Institut Universitaire de France.

\subsection{Notation and preliminaries}
\label{sec:notation}
\label{sec:preliminaries}

We normalize local class field theory so that uniformizers correspond to geometric Frobenius elements.
We fix an embedding $\kappa_0:\F_q\into \F$ and let $\kappa_j\defeq \kappa_0\circ\varphi^j$, where $\varphi$ is the arithmetic Frobenius on $\F_q$. Given $J\subset\{0,\dots,f-1\}$ we define $J^c\defeq \{0,1,\dots,f-1\} \setminus J$. We let $I\defeq \begin{pmatrix}\cO_K^\times&\cO_K\\p\cO_K&\cO_K^\times\end{pmatrix} \subset \GL_2(\cO_K)$ denote the (upper) Iwahori subgroup of $\GL_2(K)$, $I_1$ the pro-$p$ radical of $I$, $Z_1$ the center of $I_1$, and $K_1\defeq 1+p{\rm M}_2(\cO_K)\subset I_1$. We let $\Gamma\defeq \GL_2(\F_q)\cong \GL_2(\cO_K)/K_1$.

Let $\brho:\Gal(\o K/K)\ra \GL_2(\F)$ be a continuous representation. We will say that $\rhobar$ is \emph{$n$-generic} for some integer $n \ge 0$ if, up to twist, $\rhobar|_{I_K}^{\ss}\not\cong \omega\oplus 1$ and either (using the notation of \S~\ref{sec:results})
\begin{equation}\label{eq:rhobar-generic-red}
  \rhobar|_{I_K} \cong \left(\begin{matrix}\omega_f^{\sum_{j=0}^{f-1} (r_j+1) p^j} & * \\ & 1\end{matrix}\right) \qquad \text{with $n\leq  r_j \leq p-3-n$ for all $0 \le j \le f-1$}
\end{equation}
or 
\begin{equation}\label{eq:rhobar-generic-irred}
  \rhobar|_{I_K} \cong \left(\begin{matrix}\omega_{2f}^{\sum_{j=0}^{f-1} (r_j+1) p^j} & \\ & \omega_{2f}^{p^f(\text{same})}\end{matrix}\right) \qquad \text{with\ }
  \begin{cases}
    n \leq r_j \leq p-3-n &\text{for $0 < j \le f-1$,} \\ n+1 \leq r_0 \leq p-2-n & \text{for $j = 0$.}
  \end{cases}
\end{equation}
In particular, if $\rhobar$ is $n$-generic then it is  $n$-generic in the sense of \cite[Def.\ 2.3.4]{BHHMS1} {(see also the beginning of \cite[\S~4.1]{BHHMS1})}, and $\rhobar$ is $0$-generic precisely when $\rhobar$ is generic in the sense of \cite[Def.~11.7]{BP} (note that the condition $\rhobar|_{I_K}^{\ss}\not\cong \omega\oplus 1$, up to twist, precisely rules out that $(r_0,\dots,r_{f-1})\in\{(0,\dots,0),(p-3,\dots,p-3)\}$ when $\rhobar$ is reducible).

Attached to a $0$-generic $\rhobar$ we have a set $W(\rhobar)$ of Serre weights, i.e.~irreducible representations of $\Gamma$ over $\F$, defined in \cite[\S~3]{BDJ}, and a finite length $\Gamma$-representation $D_0(\rhobar)$ over $\F$, defined in \cite[\S~13]{BP}, which is of the form $D_0(\rhobar)=\bigoplus_{\tau\in W(\rhobar)}D_{0,\tau}(\rhobar)$, where each $D_{0,\tau}(\rhobar)$ is indecomposable and multiplicity free with socle the Serre weight $\tau$ (\cite[\S~15]{BP}).

Suppose that $\rhobar$ is {0}-generic. Recall the set $\P$ parametrizing $D_0(\rhobar)^{I_1}$, see \cite[\S~4]{breuil-buzzati} (and denoted there by $\mathscr{PD}$, resp.\ $\mathscr{PID}$, if $\rhobar$ is reducible, resp.\ irreducible). Recall also the subset $\D \subset \P$ parametrizing (the $I_1$-invariants of) the set of Serre weights in $W(\brho)$ (denoted in \emph{loc.~cit.} by $\D$ or $\mathscr{ID}$ if $\rhobar$ is reducible or irreducible respectively).
We let $\D^\ss\subset \P^\ss$ denote the corresponding sets for the semisimplification $\rhobar^\ss$ of $\rhobar$, so $\P \subset \P^\ss$ and $\D \subset \D^\ss$.
Note that $\chi \in \JH(D_0(\rhobar)^{I_1})$ implies $\chi \ne \chi^s$ by \cite[Cor.\ 13.6]{BP}.

Since we will be using this many times, we recall more precisely that if $\brho$ is \emph{reducible}, $\P^\ss$ denotes the set of $f$-tuples $(\lambda_0(x_0),\dots,\lambda_{f-1}(x_{f-1}))$ such that:
\begin{enumerate}
\item $\lambda_j(x_j) \in \{x_j,x_j+1,x_j+2,p-3-x_j,p-2-x_j,p-1-x_j\}$;
\item if $\lambda_j(x_j) \in \{x_j,x_j+1,x_j+2\}$, then $\lambda_{j+1}(x_{j+1}) \in \{x_{j+1},x_{j+1}+2,p-2-x_{j+1}\}$;
\item if $\lambda_j(x_j) \in \{p-3-x_j,p-2-x_j,p-1-x_j\}$, then $\lambda_{j+1}(x_{j+1}) \in \{x_{j+1}+1,p-3-x_{j+1},p-1-x_{j+1}\}$
\end{enumerate}
and $\D^\ss$ is the subset such that $\lambda_j(x_j) \in \{x_j,x_j+1,p-3-x_j,p-2-x_j\}$.
Moreover, there exists a unique subset $J_{\rhobar} \subset \{0,\dots,f-1\}$ such that
\begin{align}
  \D &= \Big\{\lambda \in \D^\ss : \lambda_j(x_j) \in \{x_j+1,p-3-x_j\} \Rightarrow j \in J_{\rhobar}\Big\},\notag \\
  \P &= \Big\{\lambda \in \P^\ss : \lambda_j(x_j) \in \{x_j+2,p-3-x_j\} \Rightarrow j \in J_{\rhobar}\Big\}.\label{eq:P}
\end{align}
In particular, $|W(\brho)| = 2^{|J_{\brho}|}$.

For $\lambda\in\mathscr{P}$ we denote by $\chi_\lambda$ the character of $H$ corresponding to $\lambda$.
(More precisely, in \cite[\S~4]{breuil-buzzati} a Serre weight $\sigma_\lambda$ is associated to $\lambda \in \P$ and $\chi_\lambda$ is the action of $H=I/I_1$ on the $1$-dimensional subspace $\sigma_\lambda^{I_1}$.)
Set
\begin{equation}\label{eq:J-lambda}J_{\lambda}\defeq \{j\in\{0,\dots,f-1\}: \lambda_j(x_j)\in \{x_j+1,x_j+2,p-3-x_j\}\}\end{equation}
and let $\ell(\lambda) \defeq  |J_\lambda|.$
By \cite[\S~11]{BP} the map $\lambda \mapsto J_\lambda$ induces a bijection between $\mathscr{D}^\ss$ and the set of subsets of $\{0,\dots,f-1\}$.
Sometimes we will abuse notation and write $J_\tau \defeq  J_\lambda$ and $\ell(\tau) \defeq  \ell(\lambda)$ if $\tau \in W(\brho^\ss)$ is parametrized by $\lambda \in \D^\ss$.
Given $\lambda\in \D^\ss$ with corresponding subset $J=J_\lambda\subset\{0,\dots,f-1\}$ we write $\delta(\lambda)\in \D^\ss$ for the $f$-tuple defined by $\delta(\lambda)_j\defeq \lambda_{j+1}$ for all $j\in\{0,\dots,f-1\}$, and $\delta(J)\subset\{0,\dots,f-1\}$ for the subset corresponding to $\delta(\lambda)$.

As in \cite[\S~1]{BP}, given $f$ integers $r_0,\dots,r_{f-1}\in\{0,\dots,p-1\}$ we denote by $(r_0,\dots,r_{f-1})$ the Serre weight
\begin{equation*}
\mathrm{Sym}^{r_0}\F^2\otimes_{\F}(\mathrm{Sym}^{r_1}\F^2)^{\rm Fr}\otimes\cdots\otimes_{\F}(\mathrm{Sym}^{r_{f-1}}\F^2)^{{\rm Fr}^{f-1}},
\end{equation*}
where $\GL_2(\Fq)$ acts on $(\mathrm{Sym}^{r_j}\F^2)^{{\rm Fr}^j}$ via $\kappa_j:\Fq\hookrightarrow \F$.
Following \cite[\S~2]{HuWang2}, we say that a Serre weight is \emph{$m$-generic} for some integer $m\geq 0$ if, up to twist, $\sigma\cong (r_0,\dots,r_{f-1})$, where $m\leq r_j\leq p-2-m$ for all $j\in \{0,\dots,f-1\}$.
We say that an $\F$-valued character $\chi$ of $I$ is \emph{$m$-generic} if $\chi=\sigma^{I_1}$ for some $m$-generic Serre weight $\sigma$. 
For any smooth character $\chi : I \to \F\s$ we define $\chi^s\defeq \chi(\Pi (\cdot) \Pi^{-1})$ with $\Pi\defeq \begin{pmatrix}0&1\\ p&0\end{pmatrix}$. 
If $\sigma$ is a Serre weight, we write $\chi_\sigma$ for the character of $I/I_1$ on $\sigma^{I_1}$ and $\sigma^{[s]}$ for the unique Serre weight distinct from $\sigma$ such that $\chi_{\sigma^{[s]}}=\chi_{\sigma}^s$. We remark that if $\brho$ is $n$-generic, then  any  $\sigma\in W(\brho^{\rm ss})$ is $n$-generic, and $\chi_{\lambda}$ is $(n-1)$-generic for any $\lambda\in\mathscr{P}^{\rm ss}$ (if $n\geq 1$).

Let $\Lambda \defeq  \F\bbra{I_1/Z_1}$, a complete noetherian local ring with maximal ideal $\m \defeq \m_{I_1/Z_1}$, and let $\gr(\Lambda) \defeq  \gr_\m(\Lambda)$ be the  graded ring associated to $\Lambda$ with respect to the $\fm$-adic filtration on $\Lambda$. The rings $\Lambda$ and $\gr(\Lambda)$ are Auslander regular (see \cite[Thm.~5.3.4]{BHHMS1} with \cite[Thm.~III.2.2.5]{LiOy}). Recall (\cite[\S~3.1]{BHHMS2}) that we have an isomorphism of (noncommutative) algebras
\begin{equation}\label{eq:grlambda}
\gr(\Lambda)\cong \bigotimes_{j\in\{0,\dots,f-1\}}\F\langle y_j,z_j,h_j\rangle
\end{equation}
with relations $[y_j,z_j]=h_j$, $[h_j,z_i]=[y_i,h_j]=0$ for all $i,j\in\{0,\dots,f-1\}$.
We use increasing filtrations throughout, i.e.\ $F_n \Lambda = \m^{-n}$ for $n \le 0$, and the degrees of $y_j$ and $z_j$ (resp.\ $h_j$) are $-1$ (resp.\ $-2$).
Define the graded ideal $J \defeq  (h_j,y_jz_j : 0 \le j \le f-1)$ of $\gr(\Lambda)$.
As in \cite[\S~3]{BHHMS2} we define
\[R \defeq  \gr(\Lambda)/(h_j : 0 \le j \le f-1) \cong \F[y_j,z_j : 0 \le j \le f-1]\]
which is the largest commutative quotient of $\gr(\Lambda)$. We also define the following quotient of $R$:
\[\o R \defeq  \gr(\Lambda)/J \cong R/(y_jz_j : 0 \le j \le f-1).\]

We recall from \cite[Def.~3.57]{BHHMS2} that given $\lambda\in\mathscr{P}$ we have an associated {homogeneous} ideal $\mathfrak{a}(\lambda)=(t_0,\dots,t_{f-1})$ of $R$, where the $t_j = t_j(\lambda)$ are defined as follows:
\begin{equation}
\label{eq:id:al}
t_j\defeq \left\{\begin{array}{llll}z_j& {\rm if} &\lambda_j(x_j)\in\{x_j,p-3-x_j\} \ \mathrm{and}\ j\in J_{\brho}\\
y_j& {\rm if} &\lambda_j(x_j)\in\{x_j+2,p-1-x_j\} \ \mathrm{and}\ j\in J_{\brho}\\
y_jz_j & {\rm if} &\lambda_j(x_j)\in \{x_j,p-1-x_j\} \ \mathrm{and}\ j\notin J_{\brho}\\
y_jz_j& {\rm if} &\lambda_j(x_j)\in \{x_j+1,p-2-x_j\}.\end{array}\right.
\end{equation}
Note that $(y_jz_j : 0 \le j \le f-1) \subset \fa(\lambda)$, so we often think of $\fa(\lambda)$ as ideal of $\o R$.

Let $H\defeq \begin{pmatrix}\F_q^\times&0\\0&\F_q^\times\end{pmatrix}\cong I/I_1$.
We write $\alpha_j:H\ra\F^\times$ for the character defined by $\begin{pmatrix}a&0\\0&d\end{pmatrix}\mapsto \kappa_j(ad^{-1})$.
We recall that for any $j\in\{0,\dots,f-1\}$ the element $y_j$ (resp.~$z_j$, resp.~$h_j$) is an $H$-eigenvector with associated eigencharacter $\alpha_j$ (resp.~$\alpha_j^{-1}$, resp.~the trivial character). {Note that $H$ acts on $I_1/Z_1$ by conjugation and hence on $\Lambda$ (resp.\ $\gr(\Lambda)$), preserving the filtration (resp.\ the grading). This induces $H$-actions also on $R$, $\overline{R}$, and $R/\fa(\lambda)$ for any $\lambda \in \P$. We say that a filtered $\Lambda$-module $M$ has a \emph{compatible $H$-action} if it has an $H$-action that preserves the filtration and such that $h(rm) = h(r)h(m)$ for all $h \in H$, $r \in \Lambda$, and $m\in M$. Similarly we define the notion of a graded $\gr(\Lambda)$-module with compatible $H$-action.}

Suppose that $H'$ is a compact $p$-adic analytic group and that $\pi_1$, $\pi_2$ are smooth representations of $H'$ over $\F$.
We write $\Ext^i_{H'}(\pi_1,\pi_2)$ for the $i$-th Ext group computed in the category of smooth representations of $H'$ over $\F$.
Dually, the functors $\Tor_i^{\F\bbra{H'}}(\pi_1^\vee,\pi_2^\vee)$ and $\Ext^i_{\F\bbra{H'}}(\pi_1^\vee,\pi_2^\vee)$ are computed in the abelian category of pseudocompact $\F\bbra{H'}$-modules.
(See for example \cite[\S~2]{emerton-ordII}.)
If $\sigma$ is a smooth representation of $H'$ over $\F$ we write $\Inj_{H'}\sigma$ for the injective envelope of $\sigma$ in the category of smooth $H'$-representations over $\F$. If $\sigma$ has finite length, we write $\JH(\sigma)$ for its set of irreducible constituents up to isomorphism.

Throughout this paper, if $R$ is a filtered (resp.\ graded) ring, a morphism of filtered (resp.\ graded) $R$-modules $f:M\ra N$ will always be a \emph{filtered (resp.\ graded) morphism of degree zero}, i.e.~satisfying $f(M_i)\subset N_i$ for all $i\in \Z$. {For $k\in\Z$, $M(k)$ denotes the filtered (resp.\ graded) $R$-module obtained by filtering (resp.\ grading) $M$ by $F_n(M(k))\defeq M(n+k)$ (resp.\ $M(k)_n\defeq M_{n+k}$) for all $n\in \Z$.}
  
If $R$ is any ring and $M$ any left $R$-module, we recall that $\Ext^{i}_R(M,R)$ for $i\in \Z_{\ge 0}$ is a right $R$-module (for $i=0$ the right $R$-action is given by $(f r)(m)\defeq f(m)r$ for $r\in R$, $f\in \Hom_R(M,R)$ and $m\in M$) and we use the notation $\EE^i_R(M)\defeq \Ext^{i}_R(M,R)$. 
{If $R = \Lambda$ or $R = \gr(\Lambda)$, we can and will use the anti-involution $g \mapsto g^{-1}$ on $I/Z_1$ to consider any right $R$-module (with compatible $H$-action or not) as a left $R$-module.}

\section{Cohen--Macaulayness of \texorpdfstring{$\gr_{\m}(\pi^{\vee})$}{gr\_m(pi\^{})}}
\label{sec:cohen-macaulay}

We completely describe $\gr_\m(\pi^{\vee})$ for a smooth mod $p$ representation $\pi$ of $\GL_2(K)$ satisfying assumptions (i), (ii) in \cite[\S~3.3.2]{BHHMS2} and an extra assumption \ref{it:assum-iv} (defined below). When $\pi$ is a suitable Hecke eigenspace in the mod $p$ cohomology, we prove that $\pi$ satisfies \ref{it:assum-iv} (in addition to (i) and (ii)).

\subsection{The theorem}
\label{sec:theorem}

We state the main theorem (Theorem \ref{thm:CMC}).

Let $\rhobar:\Gal(\o K/K)\ra\GL_2(\F)$ be a continuous {$0$-generic} representation as in \S~\ref{sec:notation}.
Let $\pi$ be a 
smooth representation of $\GL_2(K)$ over $\F$ {admitting a central character and} satisfying assumptions (i), (ii) in \cite[\S~3.3.2]{BHHMS2}, i.e.~
\begin{enumerate}
\item\label{it:assum-i} there exists an integer $r\ge 1$ such that $\pi^{K_1}\cong D_0(\rhobar)^{\oplus r}$ as $\GL_2(\cO_K)K^\times$-representations, where \ $K^\times$ \ acts \ by \ $\det(\rhobar)\omega^{-1}$ \ (in \ particular \ $\pi$ \ is \ admissible \ and \ has \ central \ character $\det(\rhobar)\omega^{-1}$); 
\item\label{it:assum-ii} for any $\lambda\in\P$ we have $[\pi[\fm^3]:\chi_\lambda]=[\pi[\fm]:\chi_\lambda]$.
\end{enumerate}

For later reference we also recall assumption (iii) of \cite[\S~3.3.5]{BHHMS2}, \emph{though we will not assume it until section \ref{sec:finite-length-ss}}:
\begin{enumerate}\setcounter{enumi}{2}
\item\label{it:assum-iii} there is a $\GL_2(K)$-equivariant isomorphism of $\Lambda$-modules
\[
\EE^{2f}_\Lambda(\pi^\vee) \cong \pi^\vee\otimes (\det(\rhobar)\omega^{-1}),
\]
where $\EE^{2f}_\Lambda(\pi^\vee)$ is endowed with the $\GL_2(K)$-action defined in \cite[Prop.~3.2]{Ko}.
\end{enumerate}

Additional to assumptions \ref{it:assum-i}, \ref{it:assum-ii} above, we make the following assumption on $\pi$:

\begin{enumerate}\setcounter{enumi}{3}
\item\label{it:assum-iv} for any smooth character $\chi:I\ra \F^{\times}$ and any $i \ge 0$, $\Ext^i_{I/Z_1}(\chi,\pi)\neq0 $ only if
  $[\pi[\m]:\chi]\neq0$, in which case
  \[m_i\defeq \dim_{\F}\Ext^i_{I/Z_1}(\chi,\pi)=\binom{2f}{i}r,\]
where $r\geq 1$ is the multiplicity in assumption~\ref{it:assum-i}.
\end{enumerate}

\emph{Note that we do not assume that $r=1$ or that $\rhobar$ is semisimple}.

\begin{rem}\label{rem:hyp4}
    By picking a minimal free resolution of $\pi^\vee$ with compatible $H$-action over the local ring $\Lambda$ (cf.~Remark \ref{rem:facts}(v)), we see that $\Tor_i^\Lambda(\F,\pi^\vee)$ is dual to
  \begin{equation*}
    \Ext^i_{\Lambda}(\pi^\vee,\F) \cong \Ext^i_{I_1/Z_1}(\F,\pi) \cong \bigoplus_{\chi} \Ext^i_{I/Z_1}(\chi,\pi),
  \end{equation*}
  where $\chi$ runs over all smooth $\F$-characters of $I$.  From assumption~\ref{it:assum-iv} we deduce that
  \begin{equation}\label{eq:dimension0}
  \dim_\F \Tor_i^\Lambda(\F,\pi^\vee) = (\dim_\F \pi^{I_1}) \binom{2f}{i}
  \end{equation}
  (as $\pi[\m] = \pi^{I_1}$).
    Decomposing for the action of $H$, we see moreover that $\Ext^i_{I/Z_1}(\chi,\pi)$ is dual to the $\chi^{-1}$-isotypic piece of $\Tor_i^\Lambda(\F,\pi^\vee)$, hence
  \[\Tor_i^{\Lambda}(\F,\pi^{\vee})\cong \bigoplus_{\lambda\in\mathscr{P}}(\chi_{\lambda}^{-1})^{\oplus m_i}.\]
\end{rem}

Our aim in this subsection is to prove the following theorem which strengthens \cite[Thm.\ 3.67]{BHHMS2}.

\begin{thm}\label{thm:CMC}
Assume that $\brho$ is $9$-generic and that $\pi$ satisfies assumptions \ref{it:assum-i}, \ref{it:assum-ii} and \ref{it:assum-iv} above.
Then we have an isomorphism of graded $\gr(\Lambda)$-modules with compatible $H$-action
\begin{equation}\label{eq:gr(pi)-bis}\bigg(\bigoplus_{\lambda\in\mathscr{P}}\chi_{\lambda}^{-1}\otimes \frac{R}{\mathfrak{a}(\lambda)}\bigg)^{\oplus r}\cong \gr_\m(\pi^{\vee}).\end{equation}
In particular, $\gr_\m(\pi^{\vee})$ is a Cohen--Macaulay $\gr(\Lambda)$-module of grade $2f$. Moreover, $\gr_\m(\pi^{\vee})$ is essentially self-dual in the sense that
\begin{equation}\label{eq:autodual-gr}\EE^{2f}_{\gr(\Lambda)}(\gr_\m(\pi^{\vee}))\cong \gr_\m(\pi^{\vee})\otimes (\det(\brho)\omega^{-1}) \end{equation}
as $\gr(\Lambda)$-modules (without grading) with compatible $H$-action.
\end{thm}

\begin{rem}\label{rem:CMC}
  The fact that $\gr_\m(\pi^{\vee})$ is Cohen--Macaulay as $\gr(\Lambda)$-module implies that $\pi^\vee$ is Cohen--Macaulay as $\Lambda$-module \cite[Prop.~III.2.2.4]{LiOy}.
But this was already known by (the proof of) \cite[Prop.~A.8]{HuWang2} when $r=1$.
\end{rem}

\begin{rem}\label{rem:intro}
  The isomorphism~(\ref{eq:gr(pi)-bis}) together with the proof of Corollary~\ref{cor:dual-of-N} show that the isomorphism~(\ref{eq:autodual-gr}) cannot respect the grading, even up to shift.
Namely, $\F\otimes_{\gr(\Lambda)} \EE^{2f}_{\gr(\Lambda)}(\gr_\m(\pi^\vee)) $ is not supported in just one degree.
\end{rem}

The proof of Theorem~\ref{thm:CMC} will be given in \S~\ref{sec:proof-theorem}.
In Proposition~\ref{prop:dim-Ext} we verify that a globally defined $\pi = \pi(\brho)$ satisfies assumption~\ref{it:assum-iv} (see \S~\ref{sec:verify-assumpt-iv} below for details).
 We note that some cases of assumption~\ref{it:assum-iv} were established in \cite[Prop.\ 10.10, Cor.\ 10.11]{HuWang2} when $\brho$ is nonsplit reducible.

\subsection{Preliminaries on filtered and graded modules}
\label{sec:prelim}

Following \cite[\S~I.6]{LiOy}, a finitely generated filtered $\Lambda$-module $L$ is called \emph{filt-free}  if it is free as a $\Lambda$-module with basis $(e_i)_{1\leq i\leq n}$ 
having the property that there exists a family $(k_i)_{1\leq i\leq n}$ of integers such that 
\begin{equation*}
  F_kL=\bigoplus_{1\leq i\leq n}(F_{k-k_i}\Lambda)e_i,\ \ \forall\ k\in\Z.
\end{equation*}
For convenience, we call $(e_i)_{1\leq i\leq n}$ a \emph{filt-basis} of $L$. Equivalently, $L$ is filt-free if and only if $L\cong \bigoplus_{i=1}^n\Lambda(-k_i)$ for some integers $k_i$.
(We remark that \cite{LiOy} add the condition $e_i\notin F_{k_i-1}L$, but this is automatic over a separated ring, and should not be demanded otherwise because of \cite[Lemma I.6.2(1)]{LiOy}.)

 If $L$ is a filt-free module and $L'$ is a submodule which is itself a free $\Lambda$-module, then $L'$, equipped with the induced filtration, need not be filt-free in general, even if $L'$ is a direct summand of $L$ as $\Lambda$-modules (see Remark~\ref{rem:not-filt-free}). However, we will see that this is true in some special cases (see Lemma \ref{lem:crit-free-bis}).

\begin{rem}\label{rem:filt-exam}
 Consider the filt-free module $L=\Lambda(0)\oplus \Lambda(-2)$, with filt-basis $(e_1,e_2)$. Let $e'=xe_1+e_2$, with $x\in\Lambda$ and $L'\defeq \Lambda e'$. Then $L'$ is a direct summand of $L$ as a $\Lambda$-module. One checks that, equipped with the induced filtration $L'$ is isomorphic to $\Lambda(-2)$, and $\gr(L')$ is a direct summand of $\gr(L)$.\\
 \indent However, if we take $e''= e_1+xe_2$ with $x\in\m\backslash \m^2$ and $L''\defeq \Lambda e''$ equipped with the induced filtration, then the morphism $\F\otimes_{\gr(\Lambda)}\gr(L'')\ra\F\otimes_{\gr(\Lambda)}\gr(L)$ is zero. Note that $L''$ is still filt-free (isomorphic to $\Lambda(-1)$).
\end{rem}

\begin{rem}\label{rem:not-filt-free}
  Suppose $L = \Lambda(0)\oplus \Lambda(0)\oplus \Lambda(-2)$, with a filt-basis $(e_1,e_2,e_3)$.
  Let $L'$ be the submodule generated by $f_1 \defeq  e_1+Y_0e_3$ and $f_2 \defeq  e_2+Z_0e_3$, with induced filtration, where $Y_0, Z_0 \in \m\setminus \m^2$ with $\gr(Y_0) = y_0$, $\gr(Z_0) = z_0$.
  Then it is easy to check that $L'$ is a direct summand as $\Lambda$-module, which is not filt-free because $F_1L' = L'$, $F_0 L' = \m L'$ but $F_{-1}L'$ is strictly bigger
  than $\m^2 L'$ (it contains $Z_0 f_1 - Y_0 f_2$).
\end{rem}

Recall that, if $A$ is a noetherian domain, then the nonzero elements form an Ore set and we can talk about its skew field of fractions (\cite[Thm.\ 6.8]{goodearl-warfield}). Therefore, any finitely generated $A$-module has a generic rank. 
In particular, this applies  to the case $A=\gr(\Lambda)$ {or $A=\Lambda$}. 
Moreover, if $L$ is a filtered $\Lambda$-module with a good filtration, then $\gr(L)$ has a generic rank that is independent of the choice of good filtration.
(This can be proved just as in the proof of \cite[Prop.\ 3.3]{Bj89}, cf.\ the proof of \cite[Lemma 3.81]{BHHMS2}.)

The next criterion reflects some features of Remark \ref{rem:filt-exam}.

\begin{lem}\label{lem:crit-free-bis}
Let $L$ be a filt-free $\Lambda$-module with compatible $H$-action. Assume that $L$ admits a direct sum decomposition of  filtered $\Lambda$-modules $L=L'\oplus L''$ compatible with $H$-action, with the following properties:
\begin{enumerate}
\item As filtered $\Lambda$-modules we have\[L' \cong \bigoplus_{i=1}^m\Lambda(-k_i),\ \ L''\cong \bigoplus_{j=m+1}^n\Lambda(-\ell_j)\]
 with $k_i\geq \ell_j$ for any pair $(i,j)$.
\item As $H$-modules, $\JH(\F\otimes_{\Lambda} L')\cap \JH(\F\otimes_{\Lambda}L'')=\emptyset$.
\end{enumerate}
Assume that $P$ is an $H$-stable direct summand of $L$ such that the composition
\begin{equation}\label{eq:cond-free-bis}\F\otimes_{\Lambda} P\into \F\otimes_{\Lambda}L\onto \F\otimes_{\Lambda}L'\end{equation}
is an isomorphism, where the second morphism is induced by the projection $L=L'\oplus L''\twoheadrightarrow L'$. Then $P$, equipped with the induced filtration, is filt-free and we have an equality $\gr(P)=\gr(L')$ inside $\gr(L)$.
\end{lem}

\begin{rem}\label{rem:cond-free}
Keep the notation of Lemma \ref{lem:crit-free-bis}. Under hypothesis (ii), the composition \eqref{eq:cond-free-bis} is {automatically} an isomorphism provided that $\F\otimes_{\Lambda}P \cong \F\otimes_{\Lambda}L'$ as $H$-modules.  
\end{rem}
\begin{proof}
Let $(e_1,\dots,e_m)$ be a filt-basis of $L'$ with $e_i$ of degree $k_i$, and similarly $(e_{m+1},\dots,e_n)$ a filt-basis of $L''$ with $e_j$ of degree $l_j$. We may require that each $e_i$ is an eigenvector of $H$ ($1\leq i\leq n$), {as $H$ preserves degrees}. By Nakayama's lemma, the surjectivity of \eqref{eq:cond-free-bis} implies that the composition
$\widetilde{\phi}:P\into L\onto L'$
is also surjective. Since $L'$ is free, $P$ splits as $L'\oplus N'$ for some submodule $N'$ of $P$, but the injectivity of \eqref{eq:cond-free-bis} implies that $\F\otimes_{\Lambda}N'=0$, hence $N'=0$ by Nakayama's lemma again.  We deduce that $\widetilde{\phi}$ is an isomorphism and that $P$ is free of rank $m$. Hence, $L = P \oplus L''$ and we may write uniquely
\[e_i=f_i+g_i,\ \ 1\leq i\leq m,\]
where $f_i\in P$ and $g_i\in L''$.
Since $P$ is $H$-stable, it follows that $f_i,g_i$ are eigenvectors of $H$ {with the same eigencharacter as $e_i$}.
Condition (ii) then forces that $g_i\in\m L''$ for $1\leq i\leq m$.

We claim that $f_i\in F_{k_i}L$ but $f_i\notin F_{k_i-1}L$. Indeed, we have
\[F_{k_i}L= F_{k_i}L'\oplus F_{k_i}L''=F_{k_i}L'\oplus \big(\bigoplus_{j=m+1}^n(F_{k_i-l_j}\Lambda)e_j\big)\supset F_{k_i}L'\oplus L'' \]
as $k_i\geq l_j$ for any pair $(i,j)$ by hypothesis (i), hence $f_i\in F_{k_i}L$. On the other hand,
\begin{equation}\label{eq:F(ki-1)}F_{k_i-1}L=F_{k_i-1}L'\oplus \big(\bigoplus_{j=m+1}^n(F_{k_i-l_j-1}\Lambda)e_j\big)\supset F_{k_i-1}L'\oplus \m L''\end{equation}
thus $f_i\notin F_{k_i-1}L$ because $e_i\notin F_{k_i-1}L'$ by choice. This proves the claim.

Now, since $P$ is equipped with the induced filtration from $L$, the claim implies that $f_i\in F_{k_i}P$ but $f_i\notin F_{k_i-1}P$.
On the other hand, since $\bigoplus_{j=m+1}^n \m e_j\subset F_{k_i-1}L$ by \eqref{eq:F(ki-1)}, we have $g_i\in F_{k_i-1}L$ and the associated principal part of $f_i$ equals that of $e_i$. Since $\gr(L')$ is generated by the principal parts of $(e_i)_{1\leq i\leq m}$, we obtain an inclusion $\gr(L')\subset \gr(P)$. However, since $P$ has rank $m$, the generic rank of $\gr(P)$ is also equal to $m$ as observed above, hence by Lemma \ref{lem:G'} below (applied with $A=\gr(\Lambda)$ and $M=\gr(L)$) we deduce an equality $\gr(P)=\gr(L')$. In particular, $\gr(P)$ is gr-free (see \cite[\S~I.4.1]{LiOy}), and consequently $P$ is filt-free by \cite[Lemma~I.6.4(3)]{LiOy}.
\end{proof}

\begin{lem}\label{lem:G'}
Let $A$ be a noetherian domain and $M$ be a finite free $A$-module. Assume that there exist $A$-submodules $M'\subset M''$ of $M$ such that 
\begin{enumerate}
\item $M'$ is a direct summand of $M$;
\item $M'$ and $M''$ have the same generic rank.
\end{enumerate}
Then $M'=M''$.
\end{lem}
\begin{proof}
By (i) we have $M=M'\oplus C$ for some $A$-submodule $C$ of $M$. Since $M'\subset M''$, it is easy to check that
\begin{equation}\label{eq:directsumm}
	M''=M'\oplus (M''\cap C).
	\end{equation}
We need to prove that $M''\cap C=0$. If this were not the case, then $M''\cap C$ would have a nonzero generic rank (as $M$ is free, hence torsion-free), and the generic rank of $M''$ would be strictly greater than that of $M'$, which contradicts (ii).
\end{proof}

The following lemma will be useful later.

\begin{lem}\label{lem:free-CA}
Let $\phi: P\ra L$ be a morphism between two free $\Lambda$-modules of finite rank. Assume that $\overline{\phi}: \F\otimes_{\Lambda} P\ra \F\otimes_{\Lambda} L$ is injective. Then $\phi$ is also injective and identifies $P$ with a direct summand of $L$.

The same statement holds if $P$ and $L$ are two gr-free $\gr(\Lambda)$-modules of finite rank and $\phi$ is a graded morphism.
\end{lem}
\begin{proof}
The first statement is a special case of \cite[Lemma 1.3.4(b)]{BH93} {whose proof extends to the noncommutative noetherian local ring $\Lambda$}. 

The proof in the graded case is similar, noting that $\gr(\Lambda)$ is a graded local ring (supported in degrees $\le 0$).
\end{proof}

Suppose that $R = \bigoplus_{d \le 0} R_d$ is a negatively graded ring and that $M$ is a graded $R$-module (here $R$ is not necessarily the ring of \S~\ref{sec:notation}). Working in the category of graded $R$-modules (with graded morphisms of degree 0), for any $n \in \Z$ we can form the quotient object $M_{\ge n} \defeq  M/\bigoplus_{d < n} M_d$, and moreover the functor $M \mapsto M_{\ge n}$ is exact.
This construction applies in particular to graded abelian groups (i.e.~$R = \Z$ supported in degree 0).
If $N$ is any graded right $R$-module, then $N \otimes_R M$ is naturally a graded abelian group, where $(N \otimes_R M)_d$ is generated by all $n \otimes m$ with $n \in N_i$, $m \in M_j$, $i+j = d$ \cite[\S~I.4.1]{LiOy}.
As the functor that forgets the grading is exact, we see (for example by \cite[Ex.\ 2.4.2]{weibel}) that the usual Tor functors $\Tor_i^R(N,M)$ are naturally graded abelian groups.

\begin{lem}\label{lem:graded-Tor}
  Suppose that $n, i \ge 0$ and that $N$ is supported in degree $0$.
  \begin{enumerate}
  \item We have a canonical isomorphism $(N \otimes_R M)_{\ge n} \cong N \otimes_R (M_{\ge n})$ of graded abelian groups. 
      \item If $M \to M'$ is a morphism in the category of graded $R$-modules inducing an isomorphism $M_{\ge n} \congto M'_{\ge n}$, then the natural map $\Tor_i^R(N,M)_{\ge n} \to \Tor_i^R(N,M')_{\ge n}$ of graded abelian groups is an isomorphism.
  \end{enumerate}
\end{lem}

\begin{proof}
  (i) By assumption, $N \otimes_R (\bigoplus_{d < n} M_d)$ is supported in degrees $< n$ and $N \otimes_R (M_{\ge n})$ is supported in degrees $\ge n$.
  By exactness of the functor $M \mapsto M_{\ge n}$, the natural map $N \otimes_R M \onto N \otimes_R (M_{\ge n})$ induces an isomorphism $(N \otimes_R M)_{\ge n} \congto N \otimes_R (M_{\ge n})$, as desired.

  (ii) We first show that if $M_{\ge n} = 0$, then $\Tor_i^R(N,M)_{\ge n} = 0$ for all $i$.
  As $N$ is supported in degree 0 and $R$ is negatively graded, we can pick a graded free resolution $\cdots \to F_1 \to F_0 \to N \to 0$ that is supported in degrees $\le 0$.
  By exactness of the functor $(\cdot)_{\ge n}$, the group $\Tor_i^R(N,M)_{\ge n}$ is computed as the $i$-th homology of the complex $(F_\bullet \otimes_R M)_{\ge n}$, which vanishes because $F_\bullet \otimes_R M$ is supported in degrees $< n$ by assumption on $M$.

  If now $f : M \to M'$ induces an isomorphism in degrees $\ge n$, then we get exact sequences $0 \to X \to M \to Y \to 0$ and $0 \to Y \to M' \to Z \to 0$ such that the composition $M \to Y \to M'$ equals $f$ and $X_{\ge n} = Z_{\ge n} = 0$.
  By the previous paragraph and exactness of the functor $(\cdot)_{\ge n}$ we obtain isomorphisms $\Tor_i^R(N,M)_{\ge n} \congto \Tor_i^R(N,Y)_{\ge n} \congto \Tor_i^R(N,M')_{\ge n}$ for any $i$, which completes the proof.
\end{proof}

\subsection{Some homological arguments}
\label{sec:some-homol-argum}

We construct different kind of resolutions of $\Lambda$-modules or $\gr(\Lambda)$-modules. 

For convenience, we recall some definitions and useful facts in the following remark.
\begin{rem}\label{rem:facts}
Let $M$ (resp.~$N)$ be a finitely generated $\Lambda$-module (resp.~$\gr(\Lambda)$-module). 
\begin{enumerate}
\item A free resolution $P_{\bullet}$ of $M$ is called \emph{minimal} if the transition maps in the induced complex $\F\otimes_{\Lambda}P_{\bullet}$ are all zero. A standard argument shows that $P_{\bullet}$ is minimal if and only if $\rank_{\Lambda}(P_i)=\dim_{\F}\Tor_i^{\Lambda}(\F,M)$ for each $i\geq0$. Using that $(\Lambda,\m)$ is a noetherian  local  ring, the same argument as in \cite[\S~1.3]{BH93} shows that minimal free resolutions $P_\bullet$ of $M$ exist and that each term $P_i$ is finitely generated.  Similarly, we define a minimal gr-free resolution $G_{\bullet}$ of $N$ and show that $G_{\bullet}$ is minimal if and only if $\rank_{\gr(\Lambda)}G_{i}=\dim_{\F}\Tor_i^{\gr(\Lambda)}(\F,N)$ for each $i\geq 0$. As $\gr(\Lambda)$ is a noetherian graded local ring, minimal gr-free resolutions $G_\bullet$ of $N$ exist and each term $G_i$ is finitely generated.
\item Suppose that $M$ carries a good filtration and let $\gr(M)$ be the associated graded $\gr(\Lambda)$-module. Let $G_{\bullet}$ be a finite gr-free resolution of $\gr(M)$. By \cite[Cor.~I.7.2.9]{LiOy}, it can be ``lifted'' to a (strict) finite filt-free resolution $P_{\bullet} $ of $M$, i.e.~$\gr(P_{\bullet})\cong G_{\bullet}$.  
By (i), we see that $P_{\bullet}$ is minimal if and only if the following two conditions are satisfied: $G_{\bullet}$ is minimal and $\dim_{\F}\Tor_i^{\Lambda}(\F,M)= \dim_{\F}\Tor_i^{\gr(\Lambda)}(\F,\gr(M))$.
\item Suppose that $M$ carries a good filtration. Let $P_{\bullet}$ be a minimal free resolution of $M$ (as $\Lambda$-module). Using \cite[Prop.~I.6.6]{LiOy} we can always endow each $P_i$ with a good filtration such that $P_{\bullet}$ becomes a filtered complex (with each transition map having degree $0$), but in general $P_{\bullet}$ is not strict.  
(In fact, the filtration can be chosen such that $P_\bullet$ is strict or filt-free, but in general not both by (ii).)
\item
If $M$ carries a good filtration, then $\Tor_i^\Lambda(\F,M)$ (and more generally $\Tor_i^\Lambda(\Lambda/\m^n,M)$ for any $n \ge 0$) carries a canonical and functorial good filtration as a $\Lambda$-module.
If $P_\bullet \to M \to 0$ is any strict filt-free resolution of $M$, then the canonical filtration on $\Tor_i^\Lambda(\F,M)$ is the one induced by the complex $\F \otimes_\Lambda P_\bullet$, with each term carrying the tensor product filtration.
See section \ref{sec:append-canon-filtr} for more details.
\item Suppose that $M$ (or $N$) carries a compatible $H$-action. Then we can require the above minimal free resolutions to carry a compatible $H$-action. We only prove (i) for $M$.  By assumption we may view $M$ as an $\F\bbra{I/Z_1}$-module. Since $\F\bbra{I/Z_1}$ is a noetherian semi-local ring with Jacobson radical $\mathfrak{J}$ (say), we can show as in \cite[\S~1.3]{BH93} that minimal \emph{projective} resolutions of $M$ exist (by taking projective covers at each step), where a resolution $P_{\bullet}$ by $\F\bbra{I/Z_1}$-modules is called ``minimal''  if the transition maps  are all zero modulo $\mathfrak{J}$.  Note that $\F\bbra{I/Z_1}$ is finite free over $\Lambda$ and that $\mathfrak{J}=\m\F\bbra{I/Z_1}$. Hence, restricting to $\Lambda$ we obtain a minimal free resolution of $M$ by $\Lambda$-modules with compatible $H$-action. 
\end{enumerate}
\end{rem}

Denote by $N$ the left-hand side of \eqref{eq:gr(pi)-bis}, i.e.\ $N \defeq  \big(\bigoplus_{\lambda\in\mathscr{P}}\chi_{\lambda}^{-1}\otimes R/\mathfrak{a}(\lambda)\big)^{\oplus r}$. 
We first prove that $N$ enjoys a property analogous to assumption~\ref{it:assum-iv} in \S~\ref{sec:theorem}.
Note that $[\pi[\m]:\chi]\neq0$ if and only if $\chi=\chi_{\lambda}$ for some $\lambda\in\mathscr{P}$.

Recall from~(\ref{eq:grlambda}) that
\begin{equation}\label{eq:grj}
\gr(\Lambda)\cong \bigotimes_{j=0}^{f-1}\gr(\Lambda)_j,
\end{equation}
where $\gr(\Lambda)_j$ is the subalgebra generated by $h_j,y_j,z_j$ (it is denoted by $\F\langle y_j,z_j,h_j\rangle$ in (\ref{eq:grlambda}) and by $U(\overline{\mathfrak{g}}_j)$ in \cite[\S~5.3]{BHHMS1} or \cite[\S~9.2]{HuWang2}). Below, we denote by $\mathfrak{b}(\lambda)$ the preimage of the ideal $\mathfrak{a}(\lambda)$ of (\ref{eq:id:al}) in $\gr(\Lambda)$, namely \[\mathfrak{b}(\lambda)=(t_j,h_j:\, 0\leq j\leq f-1).\]

For $n \ge 1$ let $\mathcal{I}^{(n)}$ denote the $H$-stable graded ideal $(y_j^n,z_j^n,h_j:\, 0\leq j\leq f-1)$ of $\gr(\Lambda)$.
By abuse of notation, we also write $\mathcal{I}^{(n)}$ for its image $(y_j^n,z_j^n:\, 0\leq j\leq f-1)$ in $R$. We let $\mathcal{I} \defeq  \mathcal{I}^{(3)}$.

\begin{lem}\label{lem:min-graded-splitting}
There exists a minimal gr-free resolution $G_\bullet$ with compatible $H$-action of $N/\mathcal I N$, which admits an $H$-stable subcomplex $G'_{\bullet}$ that is a minimal gr-free resolution of $N$.
The induced map $H_0(G'_\bullet) \to H_0(G_\bullet)$ is the natural map $N \to N/\mathcal IN$.
Moreover, we have a decomposition $G_\bullet = G'_\bullet \oplus G''_\bullet$ of graded $\gr(\Lambda)$-modules with compatible $H$-action  (which may not be respected by the transition maps).
\end{lem}

By minimality we deduce that $\Tor_{i}^{\gr(\Lambda)}(\F,N) \cong \F \otimes_{\gr(\Lambda)} G'_i$ and likewise for $N/\mathcal IN$. We deduce:

\begin{cor}\label{cor:inj-Tori}
  The natural morphism $N\ra N/\mathcal{I}N$ induces injective graded morphisms with compatible $H$-action
  \[\Tor_{i}^{\gr(\Lambda)}(\F,N)\hookrightarrow \Tor_{i}^{\gr(\Lambda)}(\F,N/\mathcal{I}N)\]
  for $i\geq 0$.
\end{cor}

\begin{proof}[Proof of Lemma~\ref{lem:min-graded-splitting}]
This is essentially done in \cite[\S~9.1, 9.2]{HuWang2}. We recall the argument in our notation.
By decomposing $N$ and twisting, it suffices to prove this when $N$ is replaced by $\gr(\Lambda)/\mathfrak b$ and $N/\mathcal IN$ is replaced by $\gr(\Lambda)/(\mathfrak b+\mathcal I)$,
where $\mathfrak{b}$ is a homogeneous ideal of $\gr(\Lambda)$ of the form $(t_j,h_j:\, 0\leq j\leq f-1)$ with $t_j\in \{y_j,z_j,y_jz_j\}$. 
Define ideals $\mathfrak{b}_j \defeq (t_j,h_j)$ and $\mathcal{I}_j \defeq (y_j^3,z_j^3,h_j)$ of $\gr(\Lambda)_j$. 
We have graded isomorphisms with compatible $H$-action:
\[\frac{\gr(\Lambda)}{\mathfrak{b}}\cong \bigotimes_{j=0}^{f-1}\frac{\gr(\Lambda)_j}{\mathfrak{b}_j}, \quad \frac{\gr(\Lambda)}{\mathfrak{b}+\mathcal I}\cong \bigotimes_{j=0}^{f-1}\frac{\gr(\Lambda)_j}{\mathfrak{b}_j+\mathcal I_j}.\]
By Lemmas 9.8--9.10 of \cite{HuWang2} we have a minimal gr-free resolution of $\gr(\Lambda)_j/(\mathfrak{b}_j+\mathcal I_j)$ with compatible $H$-action:
\[0 \to G_{3}^{(j)} \to G_{2}^{(j)} \to G_{1}^{(j)} \to G_{0}^{(j)} \to \frac{\gr(\Lambda)_j}{\mathfrak{b}_j+\mathcal I_j} \to 0, \]
depending on $t_j$. Without recalling the transition maps, if $t_j = y_j$, then
\begin{gather*}
  G_{3}^{(j)}= \gr(\Lambda)_j(6)_{\alpha_j^{-2}}, \\
  G_{2}^{(j)}= \fbox{$\gr(\Lambda)_j(3)_{\alpha_j}$}\oplus \gr(\Lambda)_j(4)_{\alpha_j^{-2}}\oplus \gr(\Lambda)_j(5)_{\alpha_j^{-3}}, \\
  G_{1}^{(j)}= \fbox{$\gr(\Lambda)_j(1)_{\alpha_j}\oplus \gr(\Lambda)_j(2)_{1}$}\oplus \gr(\Lambda)_j(3)_{\alpha_j^{-3}}, \\
  G_{0}^{(j)}= \fbox{$\gr(\Lambda)_j(0)_{1}$},
\end{gather*}
where the final subscript indicates the $H$-action and where the boxed terms indicate a subcomplex $G'^{(j)}_i$ that is a minimal gr-free resolution of $\gr(\Lambda)_j/\mathfrak{b}_j$.
If $t_j = z_j$, then the terms have the same form, but the characters of $H$ are replaced by their inverses.
If $t_j = y_jz_j$, then
\begin{gather*}
  G_{3}^{(j)}= \gr(\Lambda)_j(6)_{\alpha_j^{2}} \oplus \gr(\Lambda)_j(6)_{\alpha_j^{-2}}, \\
  G_{2}^{(j)}= \gr(\Lambda)_j(5)_{\alpha_j^3}\oplus \gr(\Lambda)_j(4)_{\alpha_j^{2}}\oplus \fbox{$\gr(\Lambda)_j(4)_{1}$}\oplus \gr(\Lambda)_j(4)_{\alpha_j^{-2}}\oplus \gr(\Lambda)_j(5)_{\alpha_j^{-3}}, \\
  G_{1}^{(j)}= \gr(\Lambda)_j(3)_{\alpha_j^3}\oplus \fbox{$\gr(\Lambda)_j(2)_{1} \oplus \gr(\Lambda)_j(2)_{1}$}\oplus \gr(\Lambda)_j(3)_{\alpha_j^{-3}}, \\
  G_{0}^{(j)}= \fbox{$\gr(\Lambda)_j(0)_{1}$}.
\end{gather*}
By the K\"unneth formula (see e.g.\ \cite[Thm.\ 3.6.3]{weibel}) we can take $G_\bullet$ (resp.\ $G'_\bullet$) to be the tensor product of the complexes $G^{(j)}_\bullet$ (resp.\ $G'^{(j)}_\bullet$) for $0 \le j \le f-1$.
These complexes are still minimal resolutions, since the transition maps are defined by elements lying in the unique maximal graded ideal of $\gr(\Lambda)$.
\end{proof}

\begin{cor}\label{cor:dual-of-N}
  The graded right $\gr(\Lambda)$-module $\EE^{2f}_{\gr(\Lambda)}(N)$ is supported in degrees $\le 4f$,
  and $\F\otimes_{\gr(\Lambda)} \EE^{2f}_{\gr(\Lambda)}(N) $ is supported in degrees $d$ with $3f \le d \le 4f$.
\end{cor}

\begin{proof}
  We may again replace $N$ by $\gr(\Lambda)/\mathfrak b$, where $\mathfrak b = (t_j,h_j:\, 0\leq j\leq f-1)$ as in the proof of Lemma~\ref{lem:min-graded-splitting}.
  By the same proof, we know that $\gr(\Lambda)/\mathfrak b$ has a gr-free resolution of length $2f$ with
  degree-$2f$ term $\bigotimes_{j=0}^{f-1} G'^{(j)}_{2} \cong \gr(\Lambda)(3(f-d)+4d)$, where $d = |\{ j: t_j = y_jz_j\}|$.
  Hence $\EE^{2f}_{\gr(\Lambda)}(\gr(\Lambda)/\mathfrak b)$ is a quotient of $\gr(\Lambda)(-3(f-d)-4d)$, which is supported in degrees $\le 3(f-d)+4d \le 4f$.
  Likewise, $\F\otimes_{\gr(\Lambda)} \EE^{2f}_{\gr(\Lambda)}(\gr(\Lambda)/\mathfrak b) $ is a quotient of $\F(-3(f-d)-4d)$ as graded vector spaces, which is supported in degree $3(f-d)+4d \in [3f,4f]$.
\end{proof}

\begin{lem}\label{lem:Tor-N}
For each $i\geq 0$ we have an isomorphism of $H$-modules
\[\Tor_i^{\gr(\Lambda)}(\F,N)\cong \bigoplus_{\lambda\in\mathscr{P}} (\chi_{\lambda}^{-1})^{\oplus m_i},\]
(see assumption \ref{it:assum-iv} in \S~\ref{sec:theorem} for $m_i$) so in particular, $\dim_\F \Tor_i^{\gr(\Lambda)}(\F,N) = \dim_\F \Tor_i^{\Lambda}(\F,\pi^\vee)$.
Moreover, as graded $\F$-vector space $\Tor_i^{\gr(\Lambda)}(\F,N)$ is supported in degrees $[-2i,-i]$.
\end{lem}

\begin{proof}
Clearly, we may assume $r=1$ so that $m_i=\binom{2f}{i}$ for $0\leq i\leq 2f$.

Going back to the minimal gr-free resolution $G'_\bullet$ of $N$ in the proof of Lemma~\ref{lem:min-graded-splitting}, we obtain
\begin{equation}\label{eq:Tor-grj}\Tor_i^{\gr(\Lambda)_j}(\F,\gr(\Lambda)_j/\mathfrak{b}_j)\cong \F \otimes_{\gr(\Lambda)} G'^{(j)}_i \cong 
\begin{cases}\F(0)_1 & \text{if $i=0$,}\\
\F(d_j)_{\chi_{t_j}} \oplus\F(2)_1 & \text{if $i=1$},\\
\F(d_j+2)_{\chi_{t_j}}&\text{if $i=2$}, \end{cases}\end{equation}
where  $\mathfrak{b}_j \defeq (t_j,h_j)$, $\chi_{t_j}$ denotes the character of $H$ acting on $t_j$, and $d_j = 2$ (resp.\ $d_j = 1$) if $t_j = y_jz_j$ (resp.\ $t_j \in \{y_j,z_j\}$).
In particular, we see that there is an isomorphism of graded $H$-modules
\[\Tor_i^{\gr(\Lambda)_j}(\F,\gr(\Lambda)_j/\mathfrak{b}_j)\cong\bigwedge\nolimits^i\Tor_1^{\gr(\Lambda)_j}(\F,\gr(\Lambda)_j/\mathfrak{b}_j).\]
Using K\"unneth's formula
\[\Tor_i^{\gr(\Lambda)}(\F,\gr(\Lambda)/\mathfrak{b})\cong \bigoplus_{i_0+\cdots+i_{f-1}=i}\bigotimes_{j=0}^{f-1}\Tor_{i_j}^{\gr(\Lambda)_j}(\F,\gr(\Lambda)_j/\mathfrak{b}_j)\]
and a similar formula for $\bigwedge^i(-)$,
we deduce an isomorphism of graded $H$-modules
\begin{equation}\label{eq:Tor-wedge}
  \Tor_i^{\gr(\Lambda)}(\F,\gr(\Lambda)/\mathfrak{b})\cong \bigwedge\nolimits^i\Tor_1^{\gr(\Lambda)}(\F,\gr(\Lambda)/\mathfrak{b})\cong \bigwedge\nolimits^i\Big(\bigoplus_{j=0}^{f-1}(\F(d_j)_{\chi_{t_j}}\oplus\F(2)_{1})\Big)
\end{equation}
for $i\geq 0$.

For fixed $\lambda\in\mathscr{P}$ we now prove that 
\begin{equation}\label{eq:dimgeqmi}\dim_\F \Hom_H(\chi_{\lambda}^{-1},\Tor_i^{\gr(\Lambda)}(\F,N))\geq \binom{2f}{i}.\end{equation} 
This will finish the proof of the lemma, as from \eqref{eq:Tor-wedge} we know that
\begin{equation}\label{eq:dimension}
\dim_\F \Tor_i^{\gr(\Lambda)}(\F,N) = \binom{2f}{i} |\mathscr P|
\end{equation}
(so the inequality in \eqref{eq:dimgeqmi} is an equality).

Let $d_1\defeq f+|\{j: t_j=y_jz_j\}|$ and $d_2\defeq |\{j:t_j\in\{y_j,z_j\}\}|$, so $d_1+d_2=2f$.  
We claim that  for each subset
$S\subset\{0,\dots,f-1\}$ such that  $t_{j}\in \{y_{j},z_{j}\}$ for all $j\in S$ (thus $i_2\defeq |S|\leq d_2$), 
\begin{equation}
\dim_\F \Hom_H\big(\chi_{\lambda}^{-1},\Tor_i^{\gr(\Lambda)}(\F,\chi_{\lambda'}^{-1}\otimes \gr(\Lambda)/\mathfrak{b}(\lambda'))\big) =\binom{d_1}{i_1},\label{eq:dim-remains}
\end{equation}
where $i_1\defeq i-i_2$ and $\lambda'\in\mathscr{P}$ is the unique element  such that 
\begin{equation}\label{eq:lambda'}
\chi_{\lambda}^{-1}=\chi_{\lambda'}^{-1}\prod_{j\in S}\chi_{t_{j}}^{-1}.
\end{equation} 
(The existence of $\lambda'\in \mathscr{P}$ is ensured by Lemma \ref{lem:compare-I1-invts}(i) below.)
Summing up \eqref{eq:dim-remains} over all $S$ and using the binomial identity 
\[\binom{2f}{i}=\sum_{i_1+i_2=i}\binom{d_1}{i_1}\binom{d_2}{i_2},\]
we deduce \eqref{eq:dimgeqmi} from the claim.

To prove the claim, we write $\mathfrak a(\lambda') = (t_j': 0 \le j \le f-1)$. By Lemma \ref{lem:compare-I1-invts}(i) below, we have  $t_{j}'=y_{j}z_j/t_j$ for $j\in S$, and  $t_{j}=t_{j}'$ otherwise.  Namely, $\chi_{t_j'}=\chi_{t_j}^{-1}$ for $j\in S$. 
Noting that $H$ acts trivially on $y_jz_j$, we easily obtain \eqref{eq:dim-remains}  from \eqref{eq:Tor-wedge} and \eqref{eq:lambda'}.

The equality of dimensions in the statement follows from (\ref{eq:dimension}) and (\ref{eq:dimension0}).
The final statement of the lemma follows from a direct analysis of $\F \otimes_{\gr(\Lambda)} G'_i$ (or by reducing to $i = 1$ by (\ref{eq:Tor-wedge})).
\end{proof}

\begin{lem}\label{lem:compare-I1-invts}
Suppose that $\lambda \in \mathscr P$ and let $\mathfrak a(\lambda) = (t_j: 0 \le j \le f-1)$ as in (\ref{eq:id:al}).
\begin{enumerate}
\item If $S\subset\{0,\dots,f-1\}$ is a  subset 
such that $t_{j}\in\{y_{j},z_{j}\}$ for all $j\in S$, then there exists a unique element $\lambda'\in \mathscr{P}$ such that $\chi_{\lambda}=\chi_{\lambda'}\prod_{j\in S}\chi_{t_{j}}$. Moreover, if we write $\mathfrak{a}(\lambda')=(t_j':0\leq j\leq f-1)$, then $t_j'=y_jz_j/t_j$ for $j\in S$ and $t_j'=t_j $ for $j\notin S$. 
\item 
\label{it:compare-I1-invts:2}
Suppose that $\rhobar$ is $(m+1)$-generic.
 Then $\chi_\lambda (\prod_{j=0}^{f-1} \alpha_j^{i_j}) = \chi_\mu$ for some $\mu \in \mathscr P$ and some integers $i_j$ with $|i_j| \le m$ for all $j$ if and only if $|i_j| \le 1$ for all $j$ and
  $i_j = -1$ (resp.\ $i_j = 1$) implies $t_j = y_j$ (resp.\ $t_j = z_j$). 
 \end{enumerate}
\end{lem}

\begin{proof}
(i) For the uniqueness of $\lambda'$ we need to show that if $\chi_{\lambda'}=\chi_{\lambda''}$ with $\lambda',\lambda''\in\mathscr{P}$, then $\lambda'=\lambda''$. This follows from \cite[Lemma 2.1]{HuWang2} (noting that $\chi_{\mu}\neq \chi_{\mu}^s$ for any $\mu\in\P$).  

For the existence of $\lambda'$ and the last statement, we may assume $S\neq \emptyset$, otherwise we just take $\lambda'=\lambda$. By induction we may assume $|S|=1$, in which case the result follows from \cite[Rk.~3.58]{BHHMS2}. 

(ii) First note that the ``if'' part holds by (i), and it remains to prove ``only if''.
As $\rhobar$ is $(m+1)$-generic we can write $\rhobar|_{I_K}$ as in (\ref{eq:rhobar-generic-red}) or~(\ref{eq:rhobar-generic-irred}) with $n = m+1$. We deduce that $\lambda_j(r_j),\mu_j(r_j) \in [m+1,p-2-m]$ from the definition of the set $\mathscr P$ \cite[\S~4]{breuil-buzzati}.   
By \cite[\S~4]{breuil-buzzati} we know that for $a,d \in \F_q^\times$ we have 
\[
\chi_\lambda\Big(\smatr{a}{}{}{d}\Big) = a^{\sum_{j=0}^{f-1} \lambda_j(r_j) p^j}(ad)^{e_\lambda}
\]
  for some integer $e_\lambda \defeq  e(\lambda)(r_0,\dots,r_{f-1})$ (where the polynomial $e(\lambda)$ is defined in \emph{loc.~cit.}). 
  We remark that $e(\lambda)$ and $\chi_\lambda$ can be defined for any $f$-tuple $\lambda$ satisfying $\sum_{j=0}^{f-1} \lambda_j(0) \equiv 0 \pmod 2$ (this condition is missing in \cite{HuWang2}, \S~2).
  
  Thus the equality $\chi_\lambda (\prod_{j=0}^{f-1} \alpha_j^{i_j}) = \chi_\mu$ is equivalent to the two congruences
  \begin{align}
    \sum_{j=0}^{f-1} \lambda_j(r_j) p^j + e_\lambda + \sum_{j=0}^{f-1} i_j p^j &\equiv \sum_{j=0}^{f-1} \mu_j(r_j) p^j + e_\mu \pmod {p^f-1}, \notag \\
    e_\lambda - \sum_{j=0}^{f-1} i_j p^j &\equiv e_\mu \pmod {p^f-1}. \label{eq:e-lambda-mu}
  \end{align}
  By subtracting, we obtain 
  \begin{equation*}
    \sum_{j=0}^{f-1} (\lambda_j(r_j)+i_j) p^j \equiv \sum_{j=0}^{f-1} (\mu_j(r_j)-i_j) p^j \pmod {p^f-1}.
  \end{equation*}
  Under the genericity condition, the integers $\lambda_j(r_j)+i_j$, $\mu_j(r_j)-i_j$ (for $0 \le j \le f-1$) lie in the interval $[1,p-2]$. 
  Therefore,
  \begin{equation}\label{eq:lambda-mu}
    \lambda_j(r_j)+i_j = \mu_j(r_j)-i_j \qquad \text{for all $0 \le j \le f-1$}.
  \end{equation}
  In particular, 
  \begin{equation}
    \lambda_j(r_j) \equiv \mu_j(r_j) \pmod 2 \qquad \text{for all $0 \le j \le f-1$}.\label{eq:lambda-mu-mod-2}
  \end{equation}

  On the other hand, from~(\ref{eq:e-lambda-mu}), the definition of $e(\lambda)$ and (\ref{eq:lambda-mu}) we easily deduce that the polynomial $\lambda_{f-1}(x_{f-1})-\mu_{f-1}(x_{f-1})$ is constant, and hence by (\ref{eq:lambda-mu-mod-2}) that $\lambda_{f-1}(x_{f-1})-\mu_{f-1}(x_{f-1}) \in \{0, \pm 2\}$.
  By the definition of $\mathscr P$ we deduce by descending induction and~(\ref{eq:lambda-mu-mod-2}) that $\lambda_{j}(x_{j})-\mu_{j}(x_{j}) \in \{0, \pm 2\}$ for all $j$.
  Therefore, by~(\ref{eq:lambda-mu}), $|i_j| \le 1$ for all $j$.
  Assume first that $j > 0$ or that $\rhobar$ is reducible.
  If $i_j = 1$, then $\lambda_j(x_j) = x_j$ or $\lambda_j(x_j) = p-3-x_j$, so $t_j = z_j$.
  (If $\rhobar$ is nonsplit reducible, note that $\mu_j(x_j) = x_j+2$ in the first case, so $j \in J_{\rhobar}$ in either case.)
  Similarly, if $i_j = -1$, then $\lambda_j(x_j) = x_j+2$ or $\lambda_j(x_j) = p-1-x_j$, so $t_j = y_j$.
  (Again, $j \in J_{\rhobar}$ if $\rhobar$ is nonsplit reducible.)
  If $j = 0$ and $\rhobar$ is irreducible, the argument is similar.  
\end{proof}

Recall that just before Lemma \ref{lem:min-graded-splitting} we defined $\mathcal{I}^{(n)} = (y_j^n,z_j^n,h_j:\, 0\leq j\leq f-1)$, an $H$-stable graded ideal of $\gr(\Lambda)$.

\begin{lem}\label{lem:N/I-multifree}
Suppose that $n \ge 1$ and that $\rhobar$ is $(2n-1)$-generic.
For each character $\chi$ of $H$ such that $[N/\mathcal{I}^{(n)} N: \chi]\neq 0$, we have $[N/\mathcal{I}^{(n)} N:\chi]=r$.
\end{lem}
\begin{proof}
It is equivalent to prove that $N'/\mathcal{I}^{(n)}N'$ is multiplicity free, where $N'\defeq \bigoplus_{\lambda\in\mathscr{P}}\chi_{\lambda}^{-1}\otimes R/\mathfrak{a}(\lambda)$. 
We have $R/(\mathfrak{a}(\lambda)+\mathcal I^{(n)}) = \F[y_j,z_j : 0 \le j \le f-1]/(t_j,y_j^n,z_j^n : 0 \le j \le f-1)$
and hence the characters of $H$ occurring in $\chi_{\lambda}^{-1}\otimes R/(\mathfrak{a}(\lambda)+\mathcal I^{(n)})$ are given by $\chi_\lambda^{-1} (\prod_{j=0}^{f-1} \alpha_j^{i_j})$, where $|i_j| \le n-1$ and $i_j \le 0$ if $t_j = y_j$ (resp.\ $i_j \ge 0$ if $t_j = z_j$).
Suppose that $N'/\mathcal{I}^{(n)}N'$ fails to be multiplicity free.
Then there are $\lambda, \mu \in \mathscr P$ and integers $i_j$, $\ell_j$ in $[-(n-1),n-1]$ such that
$\chi_\lambda^{-1} (\prod_{j=0}^{f-1} \alpha_j^{i_j}) = \chi_\mu^{-1} (\prod_{j=0}^{f-1} \alpha_j^{\ell_j})$ and $(\lambda,\un i) \ne (\mu,\un \ell)$.
By symmetry we may assume that $\ell_{j_0} > i_{j_0}$ for some $j_0$.
For $0 \le j \le f-1$ let $t_j$ (resp.\ $t_j'$) be associated to $\lambda$ (resp.\ $\mu$) as in \eqref{eq:id:al}.
From Lemma~\ref{lem:compare-I1-invts}(ii) applied to $\chi_\lambda (\prod_{j=0}^{f-1} \alpha_j^{\ell_j-i_j}) = \chi_\mu$ with $m = 2n-2$ we obtain that $\ell_{j_0}-i_{j_0} = 1$ and $t_{j_0} = z_{j_0}$.
Applying the same lemma with the roles of $\lambda$ and $\mu$ interchanged, we also get $t_{j_0}' = y_{j_0}$.
By above this implies that $i_{j_0} \ge 0 \ge \ell_{j_0}$, contradicting that $\ell_{j_0} > i_{j_0}$.
\end{proof}

\subsection{The Iwahori representation \texorpdfstring{$\tau$}{tau}}
\label{sec:representation-tau}

We define a finite-dimensional subrepresentation $\tau = \tau^{(3)}$ of $\pi|_I$ and prove a crucial injectivity result on the level of Tor groups in Proposition~\ref{prop:Tor-inj}.

\begin{lem}\label{lem:tau-embed}
Suppose that $1\leq n\leq p$. 
There exists a finite-dimensional smooth representation $\tau^{(n)}$ of $I$ over $\F$ such that \[\gr_{\m}((\tau^{(n)})^{\vee})\cong N/\mathcal{I}^{(n)} N\]
as graded $\gr(\Lambda)$-modules  with compatible $H$-action.
More precisely, $\tau^{(n)} \cong (\bigoplus_{\lambda \in \mathscr P}\tau^{(n)}_{\lambda})^{\oplus r}$, where $\tau^{(n)}_{\lambda}$ satisfies 
\[\gr_\m((\tau^{(n)}_{\lambda})^{\vee})\cong \chi_{\lambda}^{-1}\otimes R/(\mathcal{I}^{(n)}+\mathfrak{a}(\lambda))\]
as graded $\gr(\Lambda)$-modules  with compatible $H$-action.
In particular, $\soc_I(\tau^{(n)}_{\lambda})=\tau^{(n)}_{\lambda}[\m] \cong \chi_\lambda$ for all $\lambda \in \mathscr P$.
\end{lem}
\begin{proof}
It suffices to show the existence of $\tau^{(n)}_{\lambda}$ 
for each $\lambda\in\mathscr{P}$, which follows by a similar argument as in \cite[Prop.~9.15]{HuWang2} (which considers $n = 3$, using slightly different notation). For convenience of the reader, we recall the argument below.

By \cite[Lemma\ 2.15(i)]{yongquan-algebra}, for $0\leq s\leq p-1$, there exists a unique $I$-representation which is trivial on $K_1$,  uniserial of length $s+1$ and whose socle filtration has graded pieces $\mathbf{1},\alpha_i^{-1},\dots,\alpha_i^{-s}$; we denote this representation by $E_i^-(s)$. For example, $E_0^-(s)$ is just the restriction to $I$ of the Serre weight $(s,0,\dots,0)$ twisted by $\eta^{-1}$, where $\eta$ is the character of $H$ acting on $(s,0,\dots,0)^{I_1}$.   By taking a conjugate action by $\smatr{0}1p0$, we obtain an $I$-representation $E_i^+(s)$ which is uniserial of length $s+1$ and whose socle filtration has graded pieces $\mathbf{1},\alpha_i,\dots,\alpha_i^{s}$. 
It is direct to check that
\[\gr_{\fm}(E_i^{-}(s)^{\vee})\cong \F[y_i,z_i]/(y_i^{s+1},z_i),\ \ \gr_{\fm}(E_i^+(s)^{\vee})\cong \F[y_i,z_i]/(y_i,z_i^{s+1}),\] 
where $\F[y_i,z_i]$ is viewed as a $\gr(\Lambda)$-module via the natural quotient map. 
Moreover, the amalgamated sum $E_{i}^-(s)\oplus_{\mathbf{1}}E_i^+(s)\defeq (E_i^+(s)\oplus E_i^-(s))/\mathbf{1}$ satisfies
\[\gr_{\fm}\big((E_{i}^-(s)\oplus_{\mathbf{1}}E_i^+(s) )^{\vee}\big)\cong \F[y_i,z_i]/(y_i^{s+1},y_iz_i,z_i^{s+1}).\]

Recall that $\fa(\lambda)=(t_i : 0\leq i\leq f-1)$ with $t_i\in\{y_i,z_i,y_iz_i\}$. Define $W_{\lambda,i} $ to be 
\[W_{\lambda,i}\defeq\left\{\begin{array}{ll}
E_i^+(n-1) & \text{if $t_i=y_i$},\\
E_i^-(n-1) & \text{if $t_i=z_i$},\\
E_i^-(n-1)\oplus_{\mathbf{1}}E_i^+(n-1) & \text{if $t_i=y_iz_i$},
\end{array}\right.\]
and $\tau_{\lambda}^{(n)}\defeq\chi_{\lambda}\otimes(\bigotimes_{i=0}^{f-1}W_{\lambda,i})$, where all tensor products in this proof are taken over $\F$.  

We claim that  $\gr_{\fm}((\tau_{\lambda}^{(n)})^{\vee})\cong \chi_{\lambda}^{-1}\otimes R/(\mathcal{I}^{(n)}+\mathfrak{a}(\lambda))$ as graded $\gr(\Lambda)$-modules  with compatible $H$-action.  For simplicity we write $M_i\defeq(W_{\lambda,i})^{\vee}$ and $M\defeq \bigotimes_{i=0}^{f-1}M_i$. Denote by $C_{\bullet}M$ the tensor product filtration on $M$, namely
\[C_{-d}M:=\sum_{d_0+\cdots +d_{f-1}=d}\bigotimes_{i=0}^{f-1}\fm^{d_i}M_i \quad\text{for $d \ge 0$}.\]
Then $\gr_{C_{\bullet}}(M)\cong\bigotimes_{i=0}^{f-1}\gr_{\m}(M_i)\cong R/(\mathcal{I}^{(n)}+\mathfrak{a}(\lambda))$ by construction of $M_i$. By \cite[Lemma 1.1(i)]{AJL}, we have an inclusion $\fm^d M\subset C_{-d}M$, which induces a morphism of graded $\gr(\Lambda)$-modules
\[\phi: \gr_{\fm}(M)\ra \gr_{C_{\bullet}}(M).\]
To prove the claim it suffices to prove that $\phi$ is an isomorphism, equivalently a surjection for dimension reasons.
It is clear that $\m^0M=C_0(M)=M$, so $\phi_0$ (the degree $0$ part of $\phi$) is surjective. Since $\gr_{C_{\bullet}}(M)$ is generated by  its degree $0$ part, we conclude by Nakayama's lemma.

  The last statement easily follows from this.
\end{proof}

By \cite[Thm.\ 3.67]{BHHMS2} we have a surjection $N\onto \gr_\m(\pi^{\vee})$ of graded $\gr(\Lambda)$-modules with compatible $H$-action.

\begin{lem}\label{lem:isom-modcI}
Suppose that $\rhobar$ is $(2n-1)$-generic. 
There exists an $I$-equivariant embedding $\tau^{(n)}\hookrightarrow \pi|_I$ such that the composite of the induced maps \[N\twoheadrightarrow
\gr_{\mathfrak{m}}(\pi^\vee)\twoheadrightarrow \gr_{\mathfrak{m}}(\tau^{(n)})^\vee\cong N/\mathcal{I}^{(n)}N\]   is identified with the natural quotient
map $N\twoheadrightarrow N/\mathcal{I}^{(n)}N$.
In particular, the surjections $N \onto \gr_\m(\pi^{\vee}) \onto \gr_\m((\tau^{(n)})^{\vee})$ are isomorphisms in degrees $\ge -(n-1)$ and $\tau^{(n)}[\m^n] = \pi[\m^n]$.
\end{lem}

\begin{proof}
(Note that the proof of the first statement is the same as that of \cite[Prop.~10.20]{HuWang2}.) From the last assertion of Lemma \ref{lem:tau-embed} we know that $\tau^{(n)}[\fm]$ is isomorphic to $ \pi[\fm]\!=\!(\bigoplus_{\lambda\in\mathscr{P}}\chi_{\lambda})^{\oplus r}$, and we may choose such an isomorphism $i: \tau^{(n)}[\m]\simto \pi[\m]$ that makes the diagram
\begin{equation}\label{eq:choice-i}\xymatrix{N_0\ar^-{\cong}[r]\ar^-{\cong}[d]&\pi^{\vee}/\m\pi^{\vee} \ar_{i^{\vee}}^-{\cong}[d]\\
(N/\mathcal{I}^{(n)}N)_0\ar^-{\cong}[r]&(\tau^{(n)})^\vee/\m(\tau^{(n)})^\vee}\end{equation}
commute, where $(-)_0$ denotes the degree $0$ part of a graded module.
Lemma \ref{lem:N/I-multifree} implies that 
\begin{equation}\label{eq:interJH}
\JH(\tau^{(n)}/\tau^{(n)}[\m])\cap \JH(\pi[\m])=\emptyset.
\end{equation}
By (\ref{eq:interJH}) and assumption~\ref{it:assum-iv} on $\pi$, we have in particular $\Ext^i_{I/Z_1}(\chi,\pi)=0$ for $\chi\in \JH(\tau^{(n)}/{\tau^{(n)}[\m]})$ and $i=0,1$, hence $\Ext^i_{I/Z_1}(\tau^{(n)}/{\tau^{(n)}[\m]},\pi)=0$ for $i=0,1$ by d\'evissage. We then deduce an isomorphism
\[\Hom_{I}(\tau^{(n)},\pi)\simto \Hom_{I}(\tau^{(n)}[\m],\pi),\]
so the above embedding $i: \tau^{(n)}[\fm]\cong \pi[\fm]\hookrightarrow \pi$ extends uniquely to an $I$-equivariant morphism $i':\tau^{(n)}\ra \pi|_I$ which must be injective (being injective on the socle). By the commutativity in \eqref{eq:choice-i} it is easy to see that $i'$ satisfies the required condition (as $N$ is generated by $N_0$).

We get the isomorphism in degrees $\ge -(n-1)$ since $h_j$ kills $N$, and this implies $\tau^{(n)}[\m^n] = \pi[\m^n]$ for dimension reasons.
\end{proof}

\begin{cor}\label{cor:tau-multfree}
  Suppose that $\rhobar$ is  $(2n-1)$-generic. 
    Then 
  \begin{enumerate}
  \item the $I$-representation $\bigoplus_{\lambda\in\mathscr{P}}\tau_{\lambda}^{(n)}$ is  multiplicity free, and
  \item all Jordan--H\"older factors of $\pi[\m^n] = \tau^{(n)}[\m^n]$ occur with multiplicity $r$.
  \end{enumerate}
\end{cor}

\begin{proof}
Note that the genericity condition implies $n\leq p$, so $\tau_{\lambda}^{(n)}$ is well-defined by Lemma~\ref{lem:tau-embed}. By Lemma~\ref{lem:tau-embed} again we have $\tau^{(n)} \cong (\bigoplus_{\lambda \in \mathscr P}\tau^{(n)}_{\lambda})^{\oplus r}$ and $\gr_{\m}((\tau^{(n)})^{\vee})\cong N/\mathcal{I}^{(n)} N$, so (i) follows from Lemma~\ref{lem:N/I-multifree}.
  By the last assertion of Lemma~\ref{lem:isom-modcI} we have $\pi[\m^n] = \tau^{(n)}[\m^n]$, so (ii) follows from $\tau^{(n)} \cong (\bigoplus_{\lambda \in \mathscr P}\tau^{(n)}_{\lambda})^{\oplus r}$ and (i).
\end{proof}

\begin{cor}
  Suppose that $\brho$ is $(2n-1)$-generic. 
 Then $\pi[\m^n]$ is isomorphic to the largest subrepresentation $V$ of $\Inj_{I/Z_1}(\pi^{I_1})[\m^n]$ containing $\pi^{I_1}$ such that $[V:\chi] = r$ if $\chi$ occurs in $\pi^{I_1}$.
  \end{cor}
\begin{proof}
  Since \ $\pi|_I \into \Inj_{I/Z_1}(\soc_I(\pi)) \ = \ \Inj_{I/Z_1}(\pi^{I_1})$, \ we \ have \ an \ injection \ $\pi[\m^n] \ \into \Inj_{I/Z_1}(\pi^{I_1})[\m^n]$.
  As $\brho$ is $(2n-1)$-generic, we have $[\pi[\m^n]:\chi] = r$ for all $\chi \in \JH(\pi^{I_1})$ by Corollary~\ref{cor:tau-multfree}(ii). Conversely, suppose that there is an $I$-representation $V$ such that $\pi^{I_1} \subset V \subset \Inj_{I/Z_1}(\pi^{I_1})[\m^n]$ and $[V:\chi] = r$ for all $\chi \in \JH(\pi^{I_1})$. {In particular we have $\JH(V/\pi^{I_1})\cap \JH(\pi^{I_1})=\emptyset$. 
  {By assumption~\ref{it:assum-iv} on $\pi$ we deduce $\Ext^1_{I/Z_1}(V/\pi^{I_1},\pi)=0$, so} 
    the inclusion $\pi^{I_1} \into \pi$ extends to a necessarily injective morphism $V \into \pi$. Since $V$ is killed by $\m^n$ by assumption, we have $V\into\pi[\m^n]\subseteq \pi$.} This proves the maximality of $\pi[\m^n]$.
  \end{proof}

Let $\tau \defeq  \tau^{(3)}$ denote the representation defined in Lemma~\ref{lem:tau-embed} for $n = 3$ ({well-defined as $p>2$}), 
so $\gr_{\m}(\tau^{\vee})\cong N/\mathcal{I} N$ as graded $\gr(\Lambda)$-modules with compatible $H$-action, where we recall that $\mathcal I = \mathcal I^{(3)}$ (see above Lemma \ref{lem:min-graded-splitting}).

Recall from Lemma~\ref{lem:min-graded-splitting} the minimal gr-free resolution $G_{\bullet}$ of $\gr_\m(\tau^{\vee})\cong N/\mathcal{I} N$ which decomposes as $G_\bullet = G'_\bullet \oplus G''_\bullet$, 
with $G'_\bullet$ being a minimal gr-free resolution of $N$.
More precisely, recall that $\tau^{\vee}\cong (\bigoplus_{\lambda\in\mathscr{P}} \tau_{\lambda}^\vee)^{\oplus r}$ and by construction {$G_{\bullet}=\bigoplus_{\lambda\in\mathscr{P}}G_{\lambda,\bullet}$, where} $G_{\lambda,\bullet}$ is a minimal gr-free resolution of $\gr_\m(\tau_{\lambda}^{\vee})$ with compatible $H$-action for each $\lambda \in \P$.
By \cite[Cor.\ I.7.2.9]{LiOy} we can lift $G_{\lambda,\bullet}$ to a (strict) filt-free resolution $L_{\lambda,\bullet}$ of $\tau_\lambda^{\vee}$.
By Remark \ref{rem:facts}(v), we may and will also require that $L_{\lambda,\bullet}$ carries a compatible $H$-action.
Then $L_{\bullet} \defeq  \bigoplus_{\lambda\in\mathscr{P}} L_{\lambda,\bullet}$ is a (strict) filt-free resolution of $\tau^{\vee}$ with compatible $H$-action.

\begin{lem}\label{lem:decomp}
  For any $i \ge 0$ there exists a decomposition $L_i = L_i' \oplus L_i''$ as filt-free $\Lambda$-modules with compatible $H$-action that reduces
  to $G_i = G_i' \oplus G_i''$ on graded pieces.
\end{lem}

Note that we do not require that the map $L_i \to L_{i-1}$ sends $L_i'$ to $L_{i-1}'$.

\begin{proof}
  We fix $i$.
  Lift $G_i'$ and $G_i''$ to filt-free $\Lambda$-modules $F_i'$ and $F_i''$ with compatible $H$-action.
  Then $L_i$ and $F_i' \oplus F_i''$ are two filt-free $\Lambda$-modules that lift $G_i$, so by \cite[Lemma I.6.2(6)]{LiOy} there exists a filtered morphism $f : L_i \to F_i' \oplus F_i''$ that lifts the given isomorphism $G_i = G_i'\oplus G_i''$.
As in Remark \ref{rem:facts}(v), we may demand in addition that $f$ is $H$-equivariant.
  By \cite[Thm.\ I.4.2.4(5)]{LiOy} the map $f$ is a strict isomorphism, so we may define $L_i'$ and $L_i''$ as pre-images of $F_i'$ and $F_i''$ in $L_i$.
\end{proof}

\begin{lem}\label{lem:resol-tau}
Suppose that $\brho$ is 5-generic.
With the above notation, $L_{\bullet}$ is also a minimal free resolution of $\tau^{\vee}$.
Moreover, for $i\in\{0,1,2\}$, $L_{i} = L_i' \oplus L_i''$ defined in Lemma~\ref{lem:decomp} satisfies conditions (i), (ii) of Lemma \ref{lem:crit-free-bis}.
\end{lem}

\begin{proof}
 For the first claim it suffices to prove the minimality of $L_{\lambda,\bullet}$ for each $\lambda \in \P$.
This is proven in \cite[Prop.~9.21]{HuWang2}. {We remark that the proof reduces to the case $\chi_{\lambda}$ is trivial (by twisting), so does not require any genericity condition on $\chi_{\lambda}$; it rather requires $p\geq 7$ to verify the property (\textbf{Min}) in \emph{loc.~cit.} which guarantees that \cite[Lemma A.11]{HuWang2} applies.}

Since $\gr_\m(\Lambda(k)) \cong \gr(\Lambda)(k)$ and $\F \otimes_{\gr(\Lambda)} \gr_\m(M) \cong \F \otimes_\Lambda M$ for any filt-free $\Lambda$-module $M$ with compatible $H$-action,
it remains to check the analogues of conditions (i), (ii) for $G_{i} = G'_i \oplus G''_i$.

Suppose that $i=2$. 
It is easy to see that if $\gr(\Lambda)(k)$ occurs in $G_{2}'$ as a direct summand, then $k\in\{2,3,4\}$, while if it occurs in $G_{2}''$ then $k\geq 4$. Hence condition (i) holds.
On the other hand, the characters of $H$ occurring in $\F \otimes_{\gr(\Lambda)} G'_2$ are of the form $\chi_\lambda^{-1} (\prod_{j=0}^{f-1} \alpha_j^{\ve'_j})$, where $\lambda \in \mathscr P$ and $|\ve'_j| \le 1$ for all $j$, and $\ve'_j = 1$ (resp.\ $\ve'_j = -1$) implies $t_j = y_j$ (resp.\ $t_j = z_j$).
Similarly, the characters of $H$ occurring in $\F \otimes_{\gr(\Lambda)} G''_2$ are of the form $\chi_\mu^{-1} (\prod_{j=0}^{f-1} \alpha_j^{\ve''_j})$, where $\mu \in \mathscr P$, $|\ve''_j| \le 3$ for all $j$ and $|\ve''_j| \ge 2$ for at least one $j$.
(In fact, also at most two $\ve'_j$ are nonzero, and likewise for the $\ve''_j$.)
Then Lemma~\ref{lem:compare-I1-invts}(ii) (applied to $\chi_\lambda (\prod_{j=0}^{f-1} \alpha_j^{\ve''_j-\ve'_j}) = \chi_\mu$ with $m = 4$;  here we use that $\brho$ is 5-generic) implies that for some $j$ we have $(\ve'_j,\ve''_j,t_j) = (1,2,z_j)$ or $(\ve'_j,\ve''_j,t_j) = (-1,-2,y_j)$ but this contradicts the information about $t_j$ above.
Therefore condition (ii) holds.

The cases $i = 0$ and $i = 1$ are similar but easier.
\end{proof}

\begin{rem}\label{rem:resol-tau}
The second statement in Lemma \ref{lem:resol-tau} need not be true for $i\gg0$. Fortunately, for the proof of Theorem \ref{thm:CMC} below we only need to treat the terms $L_i$ for $i\in \{0,1,2\}$.
\end{rem}

The following is a consequence of the first assertion of Lemma \ref{lem:resol-tau}.
\begin{cor}\label{cor:Tor-tau}
Suppose that $\brho$ is 5-generic. 
For any $i \ge 0$ there is a canonical isomorphism
\[\Tor_i^{\gr(\Lambda)}(\F,\gr_\m(\tau^{\vee}))\cong \gr(\Tor_i^{\Lambda}(\F,\tau^{\vee})).\]
(Here, $\Tor_i^{\Lambda}(\F,\tau^{\vee})$ carries the canonical filtration, cf.\ Remark~\ref{rem:facts}(iv).) 
\end{cor}

\begin{proof}
Using the spectral sequence introduced in the proof of Proposition \ref{prop:Tor-inj} below, we know that $\gr(\Tor_i^{\Lambda}(\F,\tau^{\vee}))$ is isomorphic to a subquotient of $\Tor_i^{\gr(\Lambda)}(\F,\gr_\m(\tau^{\vee}))$. 
But
\begin{equation*}
\dim_\F \gr(\Tor_i^{\Lambda}(\F,\tau^{\vee})) = \dim_\F \Tor_i^{\Lambda}(\F,\tau^{\vee}) = \dim_\F \Tor_i^{\gr(\Lambda)}(\F,\gr_\m(\tau^{\vee})),
\end{equation*}
where the second equality follows from the first assertion of Lemma~\ref{lem:resol-tau} and the minimality of $G_\bullet$ (see Remark \ref{rem:facts}(ii)), which concludes the proof.
\end{proof}

Next, we compare $\Tor_i^{\Lambda}(\F,\pi^{\vee})$ and $\Tor_i^{\Lambda}(\F,\tau^{\vee})$.
Recall that by Lemma \ref{lem:isom-modcI} we have a surjection of $\F\bbra{I/Z_1}$-modules $\pi^\vee \onto \tau^\vee$, provided $\brho$ is 5-generic.

\begin{prop}\label{prop:Tor-inj}
{Assume that $\brho$ is $9$-generic.}
The natural morphism
\[\Tor_i^{\Lambda}(\F,\pi^{\vee})\ra \Tor_i^{\Lambda}(\F,\tau^{\vee})\]
is injective for any $0 \le i \le 2$.
\end{prop}
\begin{proof}
Let $\varphi_i:\Tor_i^{\Lambda}(\F,\pi^{\vee})\ra \Tor_i^{\Lambda}(\F,\tau^{\vee})$ denote the natural morphism.
It suffices to prove the following statement: there exist separated filtrations on the finite-dimensional $\F$-vector spaces $\Tor_i^{\Lambda}(\F,\pi^{\vee})$ and $\Tor_i^{\Lambda}(\F,\tau^{\vee})$, {with respect to which $\varphi_i$ becomes a filtered morphism and} 
such that the induced graded morphism {$\gr(\varphi_i)$} is injective. To show this, we use a spectral sequence which computes $\gr(\Tor_i^{\Lambda}(\F,-))$ using $\Tor_i^{\gr(\Lambda)}(\F,\gr(-))$, analogous to the one introduced in the proof of \cite[Prop.\ 3.84]{BHHMS2}.

Starting from a minimal gr-free resolution of $\gr_\m(\pi^{\vee})$, by Remark \ref{rem:facts}(ii) we can lift it to a filt-free resolution of $\pi^{\vee}$, say $M_{\bullet}$. Tensoring with $\F$, we obtain
a filtered complex $\F\otimes_{\Lambda}M_{\bullet}$ and we pass to the associated graded complex, $\gr(\F\otimes_{\Lambda}M_{\bullet})$. As in the proof of \cite[Prop.\ 3.84]{BHHMS2} (cf.\ \cite[\S~III.2.2]{LiOy}),  we obtain a spectral sequence $\{E^r_i, r\geq 0, i\geq 0\}$, with the following properties {(we use homological indexing on $M_{\bullet}$ while cohomological indexing is used in \cite[\S~III.1]{LiOy})}:
\begin{itemize}
\item[(a)] $E_i^0=\gr(\F\otimes_{\Lambda}M_i)\cong \F\otimes_{\gr(\Lambda)}\gr(M_i)$ (by \cite[Lemma~I.6.14]{LiOy}), $E_i^1=\Tor_i^{\gr(\Lambda)}(\F,\gr_\m(\pi^{\vee}))$;
\item[(b)] for any fixed $r\geq 1$, there is a complex
\[\cdots \ra E_1^r\ra E_0^r\ra0\]
whose homology gives $E_i^{r+1}$;
\item[(c)] for $r$ large enough (depending on $i$), $E^r_i\cong E_i^{\infty}= \gr(\Tor_i^{\Lambda}(\F,\pi^{\vee}))$.
\end{itemize}
Note that the filtration on $\Tor_i^{\Lambda}(\F,\pi^{\vee})$ is induced from the one on $\F\otimes_\Lambda M_{i}$, see \cite[\S~III.1, p.~128]{LiOy}.
It is in particular separated.
Similarly, replacing $\pi^{\vee}$ by $\tau^{\vee}$ and using the  filt-free and minimal free resolution $L_\bullet$ of $\tau^\vee$, we have another spectral sequence $\{E'^{r}_i, r\geq0, i\geq 0\}$, converging to $\Tor_i^{\Lambda}(\F,\tau^{\vee})$. Moreover, using \cite[Prop.~I.6.5(2)]{LiOy} a standard argument shows that there is a filtered morphism of complexes of $\Lambda$-modules with compatible $H$-actions $M_{\bullet}\ra L_{\bullet}$ extending $\pi^{\vee}\onto \tau^{\vee}$. Hence by functoriality we obtain a morphism between the spectral sequences:
\begin{equation}\label{eq:mor-spec-seq}
  \begin{gathered}
    \xymatrix{E_i^r\ar@{=>}[r]\ar[d]&\Tor_i^{\Lambda}(\F,\pi^{\vee})\ar^{\varphi_i}[d]\\
      E'^r_i\ar@{=>}[r]&\Tor_i^{\Lambda}(\F,\tau^{\vee})}
  \end{gathered}
\end{equation}
{and that $\varphi_i$ is a filtered morphism with respect to the canonical filtrations on $\Tor_i^{\Lambda}(\F,\pi^{\vee})$ and $\Tor_i^{\Lambda}(\F,\tau^{\vee})$.} Note that the bottom spectral sequence degenerates at the page $r=1$, by Corollary \ref{cor:Tor-tau}.
As explained above, it suffices to show that the natural map \[\gr(\varphi_i):\ E_i^\infty = \gr(\Tor_i^\Lambda(\F,\pi^\vee)) \to \gr(\Tor_i^\Lambda(\F,\tau^\vee)) = E'^\infty_i\] is injective for $0 \le i \le 2$.

\textbf{Step 1.} Suppose $i = 0$. 
Then the natural surjection \[E_0^1 = \F \otimes_{\gr(\Lambda)} \gr_\m(\pi^\vee)\cong \gr_{\m}^0(\pi^{\vee}) \onto  \gr_{\m}^0(\tau^{\vee})\cong\F \otimes_{\gr(\Lambda)} \gr_\m(\tau^\vee) = E'^1_0\] is an isomorphism by Lemma~\ref{lem:isom-modcI}. 
We then have a commutative diagram
\[\xymatrix{E_0^1\ar@{->>}[r]\ar^{\cong}[d]&E_0^\infty\ar[d]\\
E'^1_0\ar@{->}^{\cong}[r]&E'^\infty_0}\]
where the bottom map is an isomorphism by Corollary~\ref{cor:Tor-tau}.
It follows that the top map and the natural map $E_0^\infty \to E'^\infty_0$ are both isomorphisms.

\textbf{Step 2.} Suppose $i = 1$. 
By the previous step we know that the map $E^1_0 \onto E^\infty_0$ is an isomorphism, so the map $E^r_1 \to E^r_0$ is zero for all $r \ge 1$.
Hence we get a natural surjection $E^r_1 \onto E^{r+1}_1$ for $r \ge 1$ and, in particular, $E^1_1 \onto E^\infty_1$.
On the other hand, let $V_i'\defeq \Tor_i^{\gr(\Lambda)}(\F,N)$ for any $i \ge 0$.
Corollary \ref{cor:inj-Tori} and the isomorphism $\gr_{\m}(\tau^{\vee})\cong N/\mathcal{I} N$ (Lemma \ref{lem:isom-modcI}) imply that the composition
\[V_i'\ra E_i^1\ra E'^1_i\]
is injective for any $i \ge 0$. 
We obtain a commutative diagram
\[\xymatrix{
V_1'\ar[dr]\ar@{_{(}->}[ddr]\ar[drr]&&\\
& E_1^1\ar@{->>}[r]\ar[d]&E_{1}^{\infty}\ar[d]\\
& E'^1_1\ar^{\cong}[r]&E'^\infty_1}\]
where we use again Corollary~\ref{cor:Tor-tau} (for $i=1$) for the bottom isomorphism.
Therefore, the top diagonal map $V_1' = \Tor_1^{\gr(\Lambda)}(\F,N)\to \gr(\Tor_1^{\Lambda}(\F,\pi^{\vee})) = E_1^\infty$ is injective.
For dimension reasons (Lemma~\ref{lem:Tor-N}),  it is actually an isomorphism, hence the vertical map $E_1^\infty \to E'^\infty_1$ is injective.

\textbf{Step 3.} Suppose $i = 2$. 
We cannot use exactly the same argument as in Step 2, since we do not (yet) know that the map $E_1^1 \onto E_1^\infty$ is an isomorphism, but fortunately it suffices to prove this in graded degrees $\ge -4$ as we now explain.
{Recall the exact functor $M \mapsto M_{\ge -4}$ for a graded $\gr(\Lambda)$-module $M$ introduced just before Lemma \ref{lem:graded-Tor}}. 
By Lemma~\ref{lem:isom-modcI} with $n = 5$ we know that the natural surjection $N \onto \gr_\m(\pi^\vee)$ is an isomorphism in degrees $\ge -4$; here we use the assumption that $\brho$ is $9$-generic.
The same is then true for the induced map of graded vector spaces $V'_1 = \Tor_1^{\gr(\Lambda)}(\F,N) \to \Tor_1^{\gr(\Lambda)}(\F,\gr_\m(\pi^{\vee})) = E^1_1$ by Lemma~\ref{lem:graded-Tor}(ii).
The diagram in Step 2 implies that the surjection $E_1^1 \onto E_1^\infty$ is an isomorphism in degrees $\ge -4$. 
Consider now the truncation in degrees $\ge -4$ of the spectral sequences associated to the above filtered complexes, which have terms $(E_i^r)_{\ge -4}$ and $(E'^r_i)_{\ge -4}$.
Exactly the same argument as in Step 2 (truncated in degrees $\ge -4$) gives us a map $\alpha : (V'_2)_{\ge -4} \to (E_2^\infty)_{\ge -4}$ fitting into a diagram
\[\xymatrix{
(V'_2)_{\ge -4}\ar[r]^\alpha &  (E_2^\infty)_{\ge -4} \ar[r]^\beta & (E'^\infty_2)_{\ge -4} \\
V'_2 \ar@{->>}[u]^\gamma &  E_2^\infty \ar@{->>}[u]^\delta \ar[r]^\epsilon & E'^\infty_2 \ar@{->>}[u]} \]
where the horizontal composition $\beta \circ \alpha$ is injective.
In particular, $\alpha$ is injective.
As $\gamma$ is an isomorphism by the last statement in Lemma~\ref{lem:Tor-N}, as $\dim_\F V'_2 = \dim_\F E_2^\infty$ (again by Lemma~\ref{lem:Tor-N}) we deduce that $\alpha$ and $\delta$ are isomorphisms.
Therefore $\beta$ is injective, so $\epsilon$ is injective, as desired.
\end{proof}

\subsection{Proof of the theorem}
\label{sec:proof-theorem}

We prove Theorem \ref{thm:CMC}, using our Tor injectivity result (Proposition \ref{prop:Tor-inj}).

\begin{proof}[Proof of Theorem \ref{thm:CMC}]
We first show that $N$ is Cohen--Macaulay and is essentially self-dual of grade $2f$ {(in the sense that $\EE^{2f}_{\gr(\Lambda)}(N)\cong N\otimes (\det(\brho)\omega^{-1})$)}.
Write again $\mathfrak{b}(\lambda)=(t_j,h_j:\, 0\leq j\leq f-1)$.
By \cite[Thm.\ 4.3]{MR1181768} we know that if $M$ is a finitely generated module over an Auslander--Gorenstein ring $R$ and $f : M \to M$ is injective $R$-linear, then $j_R(M/f(M)) \ge j_R(M)+1$, where $j_R(-)\defeq \min\{i:\Ext^i_R(-,R)\neq0\}$ denotes the grade.
We apply this inductively with the central regular sequence $h_0,\dots,h_{f-1}$ and then $t_0,\dots,t_{f-1}$ (and $M = \gr(\Lambda)$) to deduce that $j_{\gr(\Lambda)}(N) \ge 2f$.
By \cite[Prop.\ 3.66]{BHHMS2} we deduce that $j_{\gr(\Lambda)}(N) = 2f$ and the essential self-duality  holds.
In Lemma~\ref{lem:min-graded-splitting} we constructed a free resolution of $N$ of length $2f$, hence $\EE^i_{\gr(\Lambda)}(N) = 0$ for $i > 2f$ and $N$ is Cohen--Macaulay.

Recall that we already have a surjection $N\twoheadrightarrow \gr_\m(\pi^{\vee})$ by \cite[Thm.\ 3.67]{BHHMS2}. In particular, we have $\mathcal{Z}(N)\geq \mathcal{Z}(\gr_\m(\pi^{\vee}))$, where the characteristic cycle is defined in \cite[\S~3.3.4]{BHHMS2}. 
(This is just the usual cycle as $\gr(\Lambda)/J$-module, since the modules are annihilated by $J$ here.)
As $N$ is essentially self-dual, it is pure by \cite[Prop.~III.4.2.8(1)]{LiOy}, so any of its nonzero submodules is of grade $2f$ over $\gr(\Lambda)$ and hence of grade 0 over $\gr(\Lambda)/J$ by the second statement in \cite[Lemma 3.65]{BHHMS2}.
In particular, any nonzero submodule of $N$ has a nonzero cycle.
Therefore, to prove the injectivity of $N\twoheadrightarrow \gr_\m(\pi^{\vee})$, it suffices to prove that $\mathcal{Z}(N)=\mathcal{Z}(\gr_\m(\pi^{\vee}))$.

Let $P_{\bullet}$ be a minimal free resolution of $(\pi|_I)^{\vee}$ with compatible $H$-action, see Remark \ref{rem:facts}. Note that initially $P_{\bullet}$ is not yet given a filtration.

\textbf{Step 1.} It suffices to prove that there exists a good filtration on $P_i$  ($i \in \{0,1,2\}$), such that $P_2 \to P_1 \to P_0$ becomes a complex of filtered $\Lambda$-modules, satisfying the following properties:
\begin{itemize}
\item[(a)] the associated graded complex $\gr(P_{\bullet})$ is exact in degree $1$, i.e.\ $H_1(\gr(P_{\bullet})) = 0$;
\item[(b)] there is an isomorphism of graded $\gr(\Lambda)$-modules $H_0(\gr(P_{\bullet}))\cong N$.
\end{itemize}
Indeed, by (a) the sequence $\gr(P_2) \to \gr(P_1) \to \gr(P_0)$ is exact in the middle, so the map $P_1 \to P_0$ is strict by \cite[Thm.~I.4.2.4(2)]{LiOy}.
  Hence the sequence $P_1 \to P_0 \to \pi^\vee \to 0$ is strict exact, where $\pi^\vee$ carries the induced filtration, say $F$.
  Therefore, by \cite[Thm.~I.4.2.4(1)]{LiOy} the sequence $\gr(P_1) \to \gr(P_0) \to \gr_F(\pi^\vee) \to 0$ is exact, i.e.\ we have isomorphisms $\gr_F(\pi^\vee) \cong H_0(\gr(P_\bullet)) \cong N$ of graded $\gr(\Lambda)$-modules (using also (b)).
In particular, we have $\mathcal{Z}(N)=\mathcal{Z}(\gr_F(\pi^{\vee}))$ and we conclude by the discussion preceding Step 1, as $\mathcal{Z}(\gr_F(\pi^{\vee})) = \mathcal{Z}(\gr_\m(\pi^{\vee}))$ by \cite[Lemma 3.81]{BHHMS2}.

\textbf{Step 2.}
 Recall that in Lemma~\ref{lem:min-graded-splitting} we constructed a minimal gr-free resolution $G_\bullet$ with compatible $H$-action of $\gr_\m(\tau^\vee)$, which 
  we decomposed $G_\bullet = G'_\bullet \oplus G''_\bullet$ as graded $\gr(\Lambda)$-modules with compatible $H$-action, where $G'_\bullet$ is a subcomplex.
  In Lemma \ref{lem:resol-tau} we showed that $G_\bullet$ lifts to a filt-free resolution $L_{\bullet}$ of $\tau^{\vee}$ (with compatible $H$-action) that is moreover minimal.
From Lemma~\ref{lem:decomp} we also get, for each $i$, a decomposition $L_i = L_i' \oplus L_i''$ as filt-free $\Lambda$-modules with compatible $H$-action, which lifts the decomposition $G_i = G_i' \oplus G_i''$ of associated graded modules.

As in Remark \ref{rem:facts}(v), we can extend the morphism $\pi^{\vee}\twoheadrightarrow \tau^{\vee}$ to a morphism of complexes of $\Lambda$-modules with compatible $H$-actions
\[\phi_{\bullet}:\ P_{\bullet}\ra L_{\bullet}.\]
 Suppose for the remainder of Step 2 that $i\in\{0,1,2\}$.
Using that $P_{\bullet}$ and $L_{\bullet}$ are minimal, Proposition \ref{prop:Tor-inj} implies that
 \[\F\otimes_{\Lambda}P_i\cong\Tor_i^{\Lambda}(\F,\pi^{\vee})\ra \Tor_i^{\Lambda}(\F,\tau^{\vee})\cong \F\otimes_{\Lambda}L_i\]
 is injective.  
By Lemma \ref{lem:free-CA}, we deduce that $\phi_i$ is injective and identifies $P_i$ with a direct summand of $L_i$ as $\Lambda$-modules.    We equip $P_i$ with the induced filtration from $L_i$. 

\textbf{Step 3.} Suppose that $i\in\{0,1,2\}$. We prove that $P_i$ is filt-free and that inside $\gr(L_i) \;(=G_i)$ the injective map $\phi_i$ induces an equality
\[\gr(P_i)=\gr(L_i') \;(=G_i'). \]
By Step 2 we know that $\phi_i$ identifies $P_i$ with a direct summand of $L_i$.
As $\F\otimes_{\Lambda}P_i \cong \Tor_i^\Lambda(\F,\pi^\vee)$ and $\F\otimes_{\Lambda}L_{i}' \cong \F\otimes_{\gr(\Lambda)}\gr(L_{i}') \cong \Tor_i^{\gr(\Lambda)}(\F,N)$,
we deduce by Lemma~\ref{lem:Tor-N} {and Remark \ref{rem:hyp4}} that $\F\otimes_{\Lambda}P_i \cong \F\otimes_{\Lambda}L_{i}'$ as $H$-modules. By Lemma \ref{lem:resol-tau}, the decomposition $L_i = L_{i}'\oplus L_{i}''$ satisfies conditions (i) and (ii) of Lemma \ref{lem:crit-free-bis}. 
Hence, by Lemma \ref{lem:crit-free-bis} and Remark \ref{rem:cond-free} we deduce the claim.

Finally, as $G_2'\ra G_1'\ra G_0'\ra N\ra0$ is an exact sequence of graded $\gr(\Lambda)$-modules (as $G'_{\bullet}$ is a resolution of $N$), the equality $G_i'=\gr(P_i)$ for $i\in\{0,1,2\}$ implies (a), (b) in Step 1.
\end{proof}

\begin{cor}\label{cor:Tor-pi}
Suppose that $\brho$ is 9-generic.
  \begin{enumerate}
  \item For any $i \ge 0$ there is a canonical isomorphism compatible with $H$-action
    \[\Tor_i^{\gr(\Lambda)}(\F,\gr_\m(\pi^{\vee}))\cong \gr(\Tor_i^{\Lambda}(\F,\pi^{\vee})).\]
    (Here, $\Tor_i^{\Lambda}(\F,\pi^{\vee})$ carries the canonical filtration, cf.\ Remark~\ref{rem:facts}(iv).) 
  \item The natural morphism \[\Tor_i^{\Lambda}(\F,\pi^{\vee})\ra \Tor_i^{\Lambda}(\F,\tau^{\vee})\]
    is injective for any $i \ge 0$.
  \end{enumerate}
\end{cor}

\begin{proof}
  (i) The proof is exactly as the proof of Corollary~\ref{cor:Tor-tau}, using Lemma~\ref{lem:Tor-N} together with Theorem~\ref{thm:CMC} instead of Lemma~\ref{lem:resol-tau} to check that both spaces have the same dimension.

  (ii) Consider again the morphism of spectral sequences~(\ref{eq:mor-spec-seq}) of the proof of Proposition~\ref{prop:Tor-inj}.
  By part (i) and Corollary~\ref{cor:Tor-tau}, both spectral sequences degenerate at the page $r = 1$.
  The map $E_1^r \to E'^r_1$ is injective by Corollary~\ref{cor:inj-Tori} together with Theorem~\ref{thm:CMC}, hence the claim follows (cf.\ the first paragraph of the proof of Proposition~\ref{prop:Tor-inj}).
\end{proof}

\subsection{Verifying assumption~\ref{it:assum-iv}}
\label{sec:verify-assumpt-iv}

We prove that a globally defined $\pi$ satisfies assumption~\ref{it:assum-iv}.

We first recall our global setup and refer the reader to \cite[\S~8.1]{BHHMS1} for more details.
Assume $p > 3$.
We fix a totally real number field $F$ with ring of integers $\cO_F$ and let $S_p$ denote the set of places of $F$ above $p$.
We assume that $F$ is unramified at all places in $S_p$.
For each finite place $w$ of $F$ we denote by $F_w$ the completion of $F$ at $w$, by $\cO_{F_w}$ its ring of integers and by $\Frob_w$ a choice of a geometric Frobenius element of $\Gal(\ovl{F}_w/F_w)$.
We fix a quaternion algebra $D$ over $F$, with center $F$ such that $D$ splits at all places in $S_p$ and at most one infinite place. 
We let $S_D$ denote the set of places of $F$ where $D$ ramifies.
We fix a maximal order $\cO_D$ in $D$ and isomorphisms $(\cO_D)_w\defeq \cO_D\otimes_{\cO_F}\cO_{F_w}\congto \M_2(\cO_{F_w})$ {for $w\notin S_{D}$}.

We fix a continuous representation $\rbar:\Gal(\ovl{F}/F)\ra\GL_2(\F)$ and let $S_{\rbar}$ denote the set of places where $\rbar$ ramifies. We write $\rbar_w$ for $\rbar|_{\Gal(\ovl{F}_w/F_w)}$. We assume that:
\begin{itemize}
\item $\rbar|_{\Gal(\ovl{F}/F(\sqrt[p]{1}))}$ is absolutely irreducible;
\item for all $w\in S_p$, $\rbar_w$ is $0$-generic {(so $S_p\subset S_{\rbar}$)};
\item for all $w\in (S_{D}\cup S_{\rbar})\setminus S_p$ the universal framed deformation ring of $\rbar_w$ is formally smooth over $W(\F)$.
\end{itemize}
If $D$ splits at exactly one infinite place (the ``indefinite case''), we make the following choices. Given a compact open subgroup $V$ of $(D\otimes_{F}\mathbb{A}_F^\infty)^\times$ (where $\mathbb{A}_F^\infty$ denotes the finite ad\`eles of $F$) we first let $X_V$ denote the smooth projective Shimura curve over $F$ associated to $V$ constructed with the convention ``$\varepsilon=-1$'' (see \cite[\S~3.1]{BD} and \cite[\S~2]{BDJ}).
We choose:
\begin{enumerate}
\item\label{it:asum3} a finite set $S$ of finite places of $F$ such that:
\begin{enumerate}[label=(\alph*)]
\item $S_D \cup S_{\rbar}\subset S$; 
\item for all $w\in S\backslash S_p$ the framed deformation ring $R_{\rbar_w^\vee}$ of $\rbar_w^\vee$ is formally smooth over $W(\F)$;
\end{enumerate}
\item\label{it:asum4} compact open subgroups $V=\prod_wV_w\subseteq U=\prod_wU_w$ of $({\mathcal O}_D)_{w}^\times$ such that:
\begin{enumerate}[label=(\alph*)]
\item\label{it:asum4a} $U_w=({\mathcal O}_D)_{w}^\times$ for $w\notin S$ 
  or $w\in S_p$;
\item $V_w=U_w$ for $w\notin S_p$ and $V_w\subset 1+p\M_2({\mathcal O}_{F_w})$, $V_w\triangleleft ({\mathcal O}_D)_{w}^\times$ for $w\in  S_p$;
\item\label{it:asum4b} we have
\begin{equation}
\label{Vnonzero}
\Hom_{\Gal(\ovl{F}/F)}\bigg(\rbar,H^1_{\rm \acute{e}t}(X_V\times_F\ovl{F},\F)\bigg)\neq 0.
\end{equation}
\end{enumerate} 
\end{enumerate}
If $D$ splits at no infinite places (the ``definite case'') we make the same choices as \refeq{it:asum3}--\refeq{it:asum4} above, replacing \eqref{Vnonzero} by the condition $S(V,\F)[\fm]\neq 0$, where:
\begin{itemize}
\item $S(V,\F)\defeq \{f:D^\times\backslash (D\otimes_F{\mathbb A}_F^\infty)^\times/V \rightarrow \F\}$;
\item  $\mathfrak m$ is generated by $T_w-S_w{\rm tr}(\rbar({\rm Frob}_w)),\ {\rm Norm}(w)-S_w{\det}(\rbar({\rm Frob}_w))$ for $w\notin S$ 
  such that $V_w=({\mathcal O}_D)_{w}^\times$, with $T_w$, $S_w$ acting on $S(V,\F)$ (via right translation on functions) by $V\begin{pmatrix}\varpi_w & 0\\ 0&1\end{pmatrix}V$, $V\begin{pmatrix}\varpi_w & 0\\ 0&\varpi_w\end{pmatrix}V$ respectively (where $\varpi_w$ is any choosen uniformizer of $F_w$).
\end{itemize}

\begin{rem}
  In \cite[\S~8]{BHHMS1} (and an earlier version of this paper) we fixed a certain place $w_1$ to ensure that our level $U$ was sufficiently small.
  But actually our level is automatically sufficiently small, as $p > 3$ and $p$ is unramified in $F$, cf.\ the discussion just before \cite[Prop.\ 5.4]{MR4573253}.
  Therefore, the place $w_1$ can also be removed in \cite[\S~8]{BHHMS1}, which means that some results, like \cite[Cor.\ 8.5.1]{BHHMS1}, hold under a mildly weaker assumption on the level.
\end{rem}

Fix now a place $v \in S_p$.
For each $w\in S_p\setminus\{v\}$ we fix a Serre weight $\sigma_w\in W(\rbar^\vee_w)$ and write $K\defeq F_v$, $\rhobar\defeq\rbar_v^\vee$. We define the admissible smooth representation of $\GL_2(K)$ over $\F$ (which is nonzero by (\ref{Vnonzero}) above):
\begin{alignat*}{2}
\pi(\rhobar)&\defeq\varinjlim_{V_v}\Hom_{U^v/V^v}\!\Big(\!\bigotimes _{w\in S_p\backslash\{v\}} \sigma_w, \Hom_{\Gal(\o F/F)}\!\big(\rbar, H^1_{{\rm \acute et}}(X_{V^vV_v} \times_F \overline F, \F)\big)\Big) &&\ {\rm in\ the\ indefinite\ case},\\
\pi(\rhobar)&\defeq\varinjlim_{V_v}\Hom_{U^v/V^v}\!\Big(\!\bigotimes _{w\in S_p\backslash\{v\}} \sigma_w, S(V^vV_v,\F)[{\mathfrak m}]\Big)&&\ {\rm in\ the\ definite\ case},
\end{alignat*}
where the limit is over all compact open subgroups $V_v\triangleleft ({\mathcal O}_D)_{v}^\times\cong \GL_2(\cO_K)$ which are contained in $1+p\M_2(\cO_K)$.
We caution the reader that, despite the notation, the representation $\pi(\rhobar)$  \emph{a priori depends on all of our global choices and not just on $\rhobar$}.

We now check that, when $\brho$ is $12$-generic, the globally defined representation $\pi = \pi(\rhobar)$ satisfies assumption~\ref{it:assum-iv} of \S~\ref{sec:theorem}. 
For this, we fix a patched module $\mathbb{M}_{\infty}$ over a suitable formally smooth local $\cO$-algebra  $R_{\infty}$  as in \cite{CEGGPS} (see also \cite[\S~8.4]{BHHMS1}) where $\cO\defeq W(\F)$, such that 
\begin{equation}
\label{eq:big:patch}
\mathbb{M}_{\infty}\otimes_{R_{\infty}}\F\cong\pi^{\vee}.
\end{equation}
We do not recall the construction {and properties of $\mathbb{M}_{\infty}$} here but we refer the reader to \cite[\S~3.1]{CEGGPS2} and item (ii) in the proof of \cite[Thm.~8.4.1]{BHHMS1}.

{In fact}, we will consider the fixed central character version of $\mathbb{M}_{\infty}$, see \cite[\S~4.22]{CEGGPS}. This amounts to taking the maximal quotient of $\mathbb{M}_{\infty}$ on which the centre $Z$ of $\GL_2(K)$ acts via a fixed character $\zeta:Z\ra \cO^{\times}$ lifting that of $\pi^{\vee}$.   In particular, setting 
\[M_{\infty}(\sigma)\defeq \Hom_{\cO\bbra{\GL_2(\cO_K)}}^{\cont}(\mathbb{M}_{\infty},\sigma^{\vee})^{\vee}\]
for any continuous $\GL_2(\cO_K)$-representation $\sigma$ on a finitely generated $\cO$-module with central character  $\zeta^{-1}$,  we obtain a patching  functor $M_\infty$ as in \cite[\S~6]{EGS} or \cite[\S~8.1]{BHHMS1}. Here,  for a linear-topological $\cO$-module $A$, $A^{\vee}$ denotes the Pontryagin dual $\Hom_{\cO}^{\rm cont}(A,\cO[\frac{1}{p}]/\cO)$ with compact-open topology.   We recall that $M_\infty(\sigma)$ is a finitely generated $R_\infty$-module. For convenience, below we assume that the action of $Z_1$ on $\mathbb{M}_{\infty}$ is trivial; this can be achieved up to twist (as $Z_1$ acts trivially on $\pi$). 

\begin{lem}\label{lem:isom-Tor}
Suppose that $\mathbb{M}_{\infty}$ is flat over $R_{\infty}$.  
For any finite-dimensional smooth $\GL_2(\cO_K)$-representation $W$ over $\F$ and any integer $i\geq 0$, there are natural isomorphisms  
\[  \Tor_i^{\o R_\infty}(\F,M_\infty(W))\cong \Tor_i^{\Lambda'}(W,\pi^{\vee})\cong \Ext^i_{\Lambda'}(W,\pi)^{\vee},\]
where $\overline{R}_{\infty}\defeq R_{\infty}\otimes_{\cO}\F$ and $\Lambda'\defeq \F\bbra{\GL_2(\cO_K)/Z_1}$. 
\end{lem}

Whenever necessary, e.g.\ {in $\Tor_i^{\Lambda'}(W,\pi^{\vee})$ in Lemma \ref{lem:isom-Tor}}, we consider $W$ as \emph{right} $\Lambda'$-module via the inversion on $\GL_2(\cO_K)/Z_1$.

\begin{proof}
Note that $\overline{R}_{\infty}$ is a regular local $\F$-algebra whose maximal ideal is generated by a regular sequence, say $\underline{y}$. 
 By  \cite[\S~3.1]{CEGGPS2}, $\mathbb{M}_{\infty}$ is projective as a pseudocompact $\cO\bbra{\GL_2(\cO_K)/Z_1}$-module, hence  $\overline{\mathbb{M}}_{\infty}\defeq \mathbb{M}_{\infty}\otimes_{\cO}\F$ is projective as a pseudocompact $\Lambda'$-module.  Since $\overline{\mathbb{M}}_{\infty}\otimes_{\overline{R}_{\infty}}\F\cong \pi^{\vee}$, we obtain  a Koszul complex $K_{\bullet}(\underline{y},\overline{\mathbb{M}}_{\infty}) =  \overline{\mathbb{M}}_{\infty} \otimes_{\o R_\infty}K_{\bullet}(\underline{y})$ of $\overline{R}_{\infty}$-modules whose homology in degree $0$ gives $\pi^{\vee}$. Since  $\overline{\mathbb{M}}_{\infty}$ is   flat over $\overline{R}_{\infty}$ by assumption, 
 $K_{\bullet}(\underline{y},\overline{\mathbb{M}}_{\infty})$   provides  a  resolution of $\pi^{\vee}$ by projective pseudocompact $\Lambda'$-modules.

We claim that we have a canonical isomorphism $W \otimes_{\Lambda'} \overline{\mathbb{M}}_{\infty} \cong M_\infty(W)$ of $\o R_\infty$-modules.
Working in the category of pseudocompact $\Lambda'$-modules (resp.\ $\F$-modules) we have by \cite[Lemma 2.4]{brumer} that
\begin{equation}\label{eq:hom-tensor}
  \Hom_{\Lambda'}^{\cont}(\overline{\mathbb{M}}_{\infty},\Hom_\F^{\cont}(W,\F)) \cong \Hom_\F^{\cont}(W \hotimes_{\Lambda'} \overline{\mathbb{M}}_{\infty},\F),
\end{equation}
where every space of continuous homomorphisms carries the discrete topology, and clearly this isomorphism is $\o R_\infty$-equivariant.
As $W$ is a finitely presented $\Lambda'$-module, we have
\begin{equation}\label{eq:hat-tensor}
  W \hotimes_{\Lambda'} \overline{\mathbb{M}}_{\infty} \cong W \otimes_{\Lambda'} \overline{\mathbb{M}}_{\infty}
\end{equation}
by \cite[Lemma 2.1]{brumer}. 
The claim follows by dualizing \eqref{eq:hom-tensor}.

By the Koszul resolution of $\pi^\vee$ above, we see that $\Tor_i^{\Lambda'}(W,\pi^{\vee})$ is computed as the $i$-th homology group of
\begin{equation*}
W \otimes_{\Lambda'} K_{\bullet}(\underline{y},\overline{\mathbb{M}}_{\infty}) = K_{\bullet}(\underline{y},W \otimes_{\Lambda'} \overline{\mathbb{M}}_{\infty}),
\end{equation*}
which is precisely the Koszul complex of $W \otimes_{\Lambda'} \overline{\mathbb{M}}_{\infty} \cong M_\infty(W)$ as $\o R_\infty$-module, and hence also computes $\Tor_i^{\o R_\infty}(\F,M_\infty(W))$.

The second isomorphism is a general fact, by using \cite[Cor.~2.6]{brumer} and 
noting that
\[\Ext^i_{\Lambda'}(\pi^{\vee},W^{\vee})^{\vee}\cong \Ext^i_{\Lambda'}(W,\pi)^{\vee}.\qedhere \]
\end{proof}
 
\begin{prop}\label{prop:dim-Ext}
  {If $\brho$ is $12$-generic}, then assumption~\ref{it:assum-iv} holds for $\pi=\pi(\brho)$. 
  As a consequence, Theorem \ref{thm:CMC} holds for $\pi$. 
\end{prop}

\begin{proof}
Under the genericity condition, $\mathbb{M}_{\infty}$ is flat over $R_{\infty}$ by \cite[Thm.~8.4.3]{BHHMS1} (for $\brho$ semisimple), \cite[Thm.\ 8.15]{HuWang2} (for $\brho$ nonsplit reducible and  $r=1$)  and \cite[Thm.~6.3]{YitongWangGKD} (for $\brho$ nonsplit reducible and general $r$).
If $\chi:I\ra \F^{\times}$ is a smooth character, then by Lemma \ref{lem:isom-Tor} and Frobenius reciprocity we have
  \begin{equation*}
    \Ext^i_{I/Z_1}(\chi,\pi)\cong \Tor_i^{\o R_\infty}(\F,M_\chi)^\vee,
  \end{equation*}
  where we write 
  \[M_{\chi}\defeq M_{\infty}\big(\Ind_I^{\GL_2(\mathcal{O}_K)}\chi\big)\cong \Hom^{\cont}_{\cO\bbra{I}}(\mathbb{M}_{\infty},\chi^{\vee})^{\vee}.\]

If $\chi \notin \JH(\pi^{I_1})$, then $M_\chi = 0$ by d\'evissage and \cite[Prop.\ 4.2]{breuil-buzzati}, as $M_\infty(\sigma) = 0$ if $\sigma$ is a Serre weight that is not in $W(\brho)$, so we are done. 
    Otherwise, $\chi=\chi_{\lambda}$ for some $\lambda\in\mathscr{P}$. 
Let $\mathcal{I}_{\chi}\subset \overline{R}_{\infty}$ be the annihilator of $M_{\chi}$. By \cite[Prop.~8.2.3]{BHHMS1} if $\rhobar$ is semisimple, \cite[Prop.~6.1]{YitongWangGKD} if $\rhobar$ is nonsplit reducible, $M_{\chi}$ is free of rank $r$ over $\overline{R}_{\infty}/\mathcal{I}_{\chi}$, which is isomorphic to $\overline R_\infty^{(1,0),\tau}$ of \emph{loc.~cit.}, where $\tau$ is the inertial type corresponding to $\Ind_I^{\GL_2(\cO_K)} (\chi)$. 
By \cite[Thm.\ 7.2.1]{EGS}, $\overline{R}_{\infty}/\mathcal{I}_{\chi}$ is a local complete intersection ring.
Since $\mathbb{M}_{\infty}$ is a finite projective $S_{\infty}\bbra{\GL_2(\cO_K)/Z_1}$-module,  where $S_{\infty}$ is a certain $\cO$-subalgebra of $R_{\infty}$ in the patching construction (see the proof of \cite[Lemma 4.18]{CEGGPS}), $M_{\chi}$ is a finite free $\o S_{\infty}\defeq S_{\infty}\otimes_{\cO}\F$-module. Hence 
\[\dim (\o R_\infty)-\dim (\o R_{\infty}/\mathcal{I}_{\chi}) = \dim (\o R_\infty)-\dim (\o S_{\infty})= 2f,\]
where the last equality follows from \cite[(81)]{BHHMS1} ({note that the assumption $\rhobar$ semisimple there is unnecessary, see e.g.~the proof of \cite[Thm.~6.3(i)]{YitongWangGKD})}.
 We deduce from \cite[Thm.\ 2.3.3(c)]{BH93} that $\mathcal{I}_\chi$ is generated by a regular sequence in $\overline{R}_{\infty}$ of length $2f$, say $\underline{a}$.
Also note that $\overline{R}_{\infty}$ is a regular local $\F$-algebra whose maximal ideal is generated by a regular sequence, say $\underline{y}$.

 By \cite[Thm.\ 2.3.9]{BH93} applied to $S=\overline{R}_{\infty}$,  $\mathbf{a}=\underline{a}$ and $\mathbf{y}=\underline{y}$,  $H_{i}(K_{\bullet}(\underline{y},\overline{R}_{\infty}/\mathcal{I}_{\chi}))$ is isomorphic to $\bigwedge^i(\F^{\oplus 2f})$ for any $i\geq 0$, hence has dimension $\binom{2f}{i}$ over $\F$ (recall $\overline{R}_{\infty}/(\underline{y})=\F$). Since $M_{\chi}$ is free of rank $r$ over $\overline{R}_{\infty}/\mathcal{I}_{\chi}$,  we have
\[K_{\bullet}(\underline{y}, M_{\chi})\cong \big(K_{\bullet}(\underline{y},\overline{R}_{\infty}/\mathcal{I}_{\chi})\big)^{\oplus r}.\]
Taking homology we obtain
$\dim_{\F}\Tor_i^{\overline{R}_{\infty}}(\F,M_{\chi})=\binom{2f}{i}r=m_i$, as desired.
\end{proof}

\section{Finite length in the split reducible case}
\label{sec:finite-length-ss}

We prove that a smooth mod $p$ representation $\pi$ of $\GL_2(K)$ satisfying assumptions \ref{it:assum-i}--\ref{it:assum-iv} of \S~\ref{sec:theorem} with $r=1$ has finite length when the underlying Galois representation $\rhobar$ is split reducible.
We also establish several structural results on $\pi$ as an $I$- and $\GL_2(\cO_K)$-representation.

We assume that $\brho:\Gal(\o K/K)\ra\GL_2(\F)$ is 
split reducible and $0$-generic. 
Throughout this section, $\pi$ is an admissible smooth representation {of $\GL_2(K)$ over $\F$} satisfying assumptions {\refeq{it:assum-i}--\refeq{it:assum-iv} of \S~\ref{sec:theorem}}. 
As seen in \S~\ref{sec:verify-assumpt-iv}, recall that $\pi=\pi(\brho)$ as defined in \S~\ref{sec:verify-assumpt-iv} satisfies assumption \ref{it:assum-iv} for any $r\geq 1$. 
{From \cite[Thm.~1.9 and Thm.~1.5]{BHHMS1} one deduces that $\pi(\rhobar)$ satisfies assumptions \ref{it:assum-i} and \ref{it:assum-ii} of \S~\ref{sec:theorem}.
(See also \cite[Cor.~3.95, Thm.~3.100]{BHHMS2} for similar results in the definite case.)
It also satisfies assumption \ref{it:assum-iii} (for any $r\geq 1$) by \cite[Thm.~8.2]{HuWang2} with \cite[Thm.~8.4.1]{BHHMS1}}.

\emph{We now assume moreover that $\pi$ is minimal, i.e.\ $r=1$ in assumptions \refeq{it:assum-i} and \refeq{it:assum-iv}.} 

\subsection{Preliminaries}
\label{sec:preliminaries-fin-len}

Given a  character $\psi:I\ra \F^{\times}$ satisfying $\psi\neq \psi^s$,
 by \cite[Lemma~2.2]{BP} there is an injective parametrization $\JH(\Ind_I^{\GL_2(\cO_K)}\psi^s) \into \mathcal{P}$ (where $\mathcal P \defeq \mathcal{P}(x_0,\dots,x_{f-1})$ with $|\mathcal{P}| = 2^f$ is defined in \cite[\S~2]{BP}), which is bijective if $\psi$ is $1$-generic (actually this condition can be slightly weakened).
  Note that $\mathcal P$ should not be confused with the set $\P$ of \S~\ref{sec:notation}!

For $\xi\in \mathcal{P}$ set (following \cite[\S~19]{BP})
\begin{equation}\label{eq:S-xi1}
\mathcal{S}(\xi)\defeq \{j\in \{0,\dots,f-1\} : \xi_j(x_j)\in \{x_j-1,p-1-x_j\}\}.
\end{equation}
We remark that the set
\begin{equation}\label{eq:S-J}\delta(\mathcal{S}(\xi))=\{j\in\{0,\dots,f-1\} : \xi_j(x_j)\in\{p-2-x_j,p-1-x_j\}\},\end{equation}
is denoted by $J(\xi)$ in \cite[\S~2]{BP}, \cite[\S~3]{HuWang2}, but for our purposes $\mathcal{S}(\xi)$ will be more convenient.
{The function} $\xi\mapsto \mathcal{S}(\xi)$ induces a bijection between $\mathcal{P}$ and the set of subsets of $\{0,\dots,f-1\}$. {In this way, any Jordan--H\"older factor of $\Ind_I^{\GL_2(\cO_K)}\psi^s$ is parametrized by a subset of $\{0,\dots,f-1\}$ and, if $\psi$ is $1$-generic, this parametrization is a bijection.}

\begin{rem}
\label{rmk:JH:conventions}
In the following we will usually talk about a Jordan--H\"older factor of $\Ind_I^{\GL_2(\cO_K)}\chi$ (rather than $\Ind_I^{\GL_2(\cO_K)}\chi^s$) parametrized by an element $\xi \in \mathcal{P}$, by which we mean the Jordan--H\"older factor of $\Ind_I^{\GL_2(\cO_K)}\psi^s$ parametrized by $\xi$ in the case where $\psi = \chi^s$.
With this convention, $\emptyset$ (resp.\ $\{0,1,\dots,f-1\}$) corresponds to the socle (resp.\ cosocle) of $\Ind_I^{\GL_2(\cO_K)}\chi$. 
Concretely, if $\chi(\smatr{a}00d)=a^s\eta(ad)$ for some character $\eta:\F_q^{\times}\ra \F^{\times}$ and integer $s=\sum_{j=0}^{f-1}p^j s_j$ with $0\leq s_j\leq p-1$, then $\xi \in \mathcal{P}$ corresponds to the Jordan--H\"older factor $\xi^{c}(s_0,\dots,s_{f-1})\otimes {\det}^{e(\xi^{c})(s_0,\dots,s_{f-1})}\eta$ {(provided $0\leq \xi_i^c(s_i)\leq p-1$ for all $i$)}, where $\xi^{c}\defeq  \xi(p-1-x_0,\dots,p-1-x_{f-1})$.
(We remark that $\xi^c \in \mathcal{P}$ and that $\mathcal{S}(\xi^c) = \mathcal{S}(\xi)^c$.)
\end{rem}

If $\sigma\in\JH(\Ind_I^{\GL_2(\cO_K)}\chi)$ is the Serre weight corresponding to $\xi\in \mathcal{P}$ (via Remark \ref{rmk:JH:conventions}), we also write $\mathcal{S}(\sigma)=\mathcal{S}(\xi)$.

{Assume that $\brho$ is $0$-generic.} Recall from \S~\ref{sec:preliminaries} that we have a  decomposition 
\[ D_0(\rhobar) = \bigoplus_{\tau\in W(\brho)}D_{0,\tau}(\brho)=\bigoplus_{i=0}^f D_0(\rhobar)_i,\]
where $D_{0}(\brho)_i\defeq \bigoplus_{\ell(\tau)=i}D_{0,\tau}(\brho)$. Recall also the set $\mathscr{P}$ from \S~\ref{sec:preliminaries}. We have an involution $\lambda \mapsto \lambda^*$ of $\mathscr{P}$ defined in \cite[\S~3.3.1]{BHHMS2}. By \cite[Lemma 3.63]{BHHMS2} we deduce:

\begin{cor}\label{cor:involution-components}
  The \ map \ $\chi_\lambda \mapsto \chi_{\lambda^*}$ \ induces \ a \ bijection \ between \ $\JH_H(D_0(\rhobar)^{I_1}_i)$ \ and $\JH_H(D_0(\rhobar)^{I_1}_{f-i})$.
\end{cor}

\begin{lem}\label{lem:I1-invt-component}
 Suppose that $\lambda \in \mathscr{P}$.
 Then $\chi_\lambda$ occurs in $D_{0,\tau}(\rhobar)^{I_1}$, where $\tau\in W(\brho)$ 
 is determined by $J_{\tau}=J_{\lambda}$. 
 Moreover, as a Jordan--H\"older factor of $\Ind_I^{\GL_2(\cO_K)}\chi_{\lambda}$, $\tau$ is parametrized (via Remark \ref{rmk:JH:conventions} and \eqref{eq:S-xi1}) by the following subset of $\{0,\dots,f-1\}$:
 \begin{equation}\label{eq:X-lambda-ss}X^{\rm ss}(\lambda)\defeq \{j : \lambda_j(x_j)\in\{x_j,x_j+1,p-2-x_j,p-3-x_j\}\}.\end{equation} 
\end{lem}

We will prove a more general version of Lemma \ref{lem:I1-invt-component} below, see Lemma~\ref{lem:I1-invt-nonss}.

\subsection{Finite length}
\label{sec:finite-length-ss-sub}

We prove that $\pi$ is of finite length (as $\GL_2(K)$-representation) and some structural results on $\pi$ as an $I$-representation.

Recall from \S~\ref{sec:append-canon-filtr} that if $M$ is a finitely generated (left) $\Lambda$-module equipped with a good filtration, then the right $\Lambda$-module  
$\EE^i_\Lambda(M)$ carries a canonical and functorial good filtration.
If furthermore $M$ has grade $j$ we obtain a canonical injection $0 \to \gr(\EE^j_\Lambda(M)) \to \EE^j_{\gr(\Lambda)}(\gr(M))$ of graded $\gr(\Lambda)$-modules,
which is an isomorphism if $\gr(M)$ is Cohen--Macaulay, see \cite[Prop.\ 3.84]{BHHMS2} (see also \cite[Prop.\ 5.6]{BE90}).

Applying the above paragraph to $M = \pi^\vee$ with its $\m$-adic filtration (where we recall that $\pi$ 
is assumed to satisfy assumptions \refeq{it:assum-i}--\refeq{it:assum-iv}), we {deduce using the second assertion of} Theorem~\ref{thm:CMC} a canonical isomorphism
\begin{equation}\label{eq:E-grpi}
  \gr(\EE^{2f}_\Lambda(\pi^\vee)) \congto \EE^{2f}_{\gr(\Lambda)}(\gr_\m(\pi^\vee)).
\end{equation}
Since all these constructions are canonical, one can check that both terms are endowed with an action of $H$ and that the above isomorphism is $H$-equivariant.

\begin{rem}\label{rem:E-grpi}
Recall that assumption~\ref{it:assum-iii} says that there is a $\GL_2(K)$-equivariant isomorphism of $\Lambda$-modules 
$\EE^{2f}_{\Lambda}(\pi^{\vee})\cong\pi^{\vee}\otimes(\det(\brho)\omega^{-1})$. 
By Remark \ref{rem:intro} and the isomorphism \eqref{eq:E-grpi}, we see that 
 the canonical filtration on $\EE^{2f}_{\Lambda}(\pi^{\vee})$  does not correspond to the $\m$-adic filtration on $\pi^{\vee}\otimes (\det(\brho)\omega^{-1})$  under the isomorphism.
\end{rem}

We denote again by $N$ the graded module defined in \S~\ref{sec:some-homol-argum} 
(with $r=1$), namely
\[N= \bigoplus_{\lambda\in\mathscr{P}}\chi_{\lambda}^{-1}\otimes \frac{R}{\mathfrak{a}(\lambda)}.\]
By Theorem \ref{thm:CMC} and our assumptions on $\pi$, we have $\gr_{\m}(\pi^{\vee})\cong N$ {provided $\rhobar$ is $9$-generic}.

{Recall that in \cite[\S~2.1.1]{BHHMS2} and \cite[Thm.~3.29]{BHHMS2} we generalized the Colmez functor from $\GL_2(\Qp)$ to $\GL_2(K)$ by associating to any smooth admissible representation $\pi'$ of $\GL_2(K)$ over $\F$ which lies in the abelian category $\cC$ of \cite[\S~3.1.2]{BHHMS2} a (finite-dimensional \'etale cyclotomic) $(\varphi,\Gamma)$-module $D_\xi^\vee({\pi'})$ over $\F\ppar{X}\cong \F\bbra{\Zp}[1/([1]-1)]$. 
The functor $D_\xi^\vee$ is contravariant and exact by \cite[Thm.~3.29]{BHHMS2}.
For instance, if the action of $\gr(\Lambda)$ on $\gr_\m(\pi'^{\vee})$ factors through its quotient $\o R$ of \S~\ref{sec:notation}, then $\pi'$ lies in $\cC$. In particular the representation $\pi$ and its subquotients all lie in $\cC$ (assumption \ref{it:assum-ii} implies that $\gr_{\m}(\pi^{\vee})$ is killed by the ideal $J\subset \gr(\Lambda)$ by the proof of \cite[Cor.~5.3.5]{BHHMS1}).   
This allows us to use the 
functor $D_\xi^\vee$ in the following proof.}

\begin{prop}\label{prop:split-I1}
{Assume that 
$\rhobar$ is $\max\{9,2f+1\}$-generic.}
  Let $0 \subsetneq \pi_1 \subsetneq \pi$ be a subrepresentation of $\pi$ and let $\pi_2\defeq \pi/\pi_1$. 
  Then both $\gr_\m(\pi_1^{\vee})$ and $\gr_F(\pi_2^{\vee})$ are Cohen--Macaulay $\gr(\Lambda)$-modules of grade $2f$, where $F$ denotes the filtration induced from $\pi^\vee$.
  In particular, $\pi_1^\vee$ and $\pi_2^\vee$ are Cohen--Macaulay $\Lambda$-modules of grade $2f$. More precisely, there is an isomorphism of graded $\gr(\Lambda)$-modules with compatible $H$-action
    \begin{equation}
      \label{eq:split-I1-N1}
      \bigoplus_{\lambda\in\mathscr{P}:~\chi_{\lambda}\in
        \JH(\pi_1^{I_1})}\chi_{\lambda}^{-1}\otimes
      \frac{R}{\mathfrak{a}(\lambda)}\simto\gr_{\m}(\pi_1^{\vee}).
    \end{equation}
 Moreover there exists a unique $\Sigma\subset\{0,1,\dots,f\}$
    such that $\pi_1^{K_1} = \bigoplus_{i \in \Sigma}
    D_0(\rhobar)_i$. In particular, $\pi_1^{K_1}$, $\pi_1^{I_1}$ and
    {$\gr_{\m}(\pi_1^{\vee})$} depend only on $\soc_{\GL_2(\cO_K)}(\pi_1)$.
\end{prop}

{We note that $\brho$ is in particular $(2f-1)$-generic, so we may apply \cite[\S~3.3.5]{BHHMS2} in the proof.}

\begin{proof}
 Let 
 \begin{gather*}
   \tau\defeq \soc_{\GL_2(\cO_K)}(\pi)=\bigoplus_{\sigma\in W(\brho)}\sigma, \\
   \tau_1\defeq \soc_{\GL_2(\cO_K)}(\pi_1),\ \ \ \tau_2\defeq \tau/\tau_1.
 \end{gather*}
Then $\tau_2\hookrightarrow \soc_{\GL_2(\cO_K)}(\pi_2)$ (note that \emph{a priori} this might be a strict inclusion).

Recall that $\pi^{K_1}=\bigoplus_{\sigma\in W(\brho)}D_0(\brho)$ by assumption~\ref{it:assum-i} in \S~\ref{sec:theorem} (with $r=1$). By the proof of \cite[Thm.\ 19.10]{BP} we have $D_{0,\sigma}(\rhobar) \subset \pi_1^{K_1}$ for any Serre weight $\sigma \subset \tau_1$. 
It follows that $\pi_1^{K_1} = \bigoplus_{\sigma \subset \tau_1} D_{0,\sigma}(\rhobar)$.
As $\smatr{0}{1}{p}{0}$ preserves $\pi_1^{I_1}$, $(\pi_1^{I_1}\hookrightarrow \pi_1^{K_1})$ is a direct summand of $(D_1(\brho)\hookrightarrow D_0(\brho))$ as a diagram, so we deduce from \cite[Thm.~15.4]{BP} that $\pi_1^{K_1} = \bigoplus_{i \in \Sigma} D_0(\rhobar)_i$ for {a unique} $\Sigma \subset \{0,1,\dots,f\}$ {which depends only on $\tau_1 = \soc_{\GL_2(\cO_K)}(\pi_1)$}.
In particular, the direct sum decomposition $\tau=\tau_1\oplus\tau_2$ induces a decomposition of $\pi^{K_1} = D_0(\brho)$ of the form:
\[D_0(\brho)=D_0(\brho)^{(1)}\oplus D_0(\brho)^{(2)}\]
with {$D_0(\brho)^{(1)} = \pi_1^{K_1}$} and $\soc_{\GL_2(\cO_K)}(D_0(\brho)^{(i)})=\tau_i$. 
This in turn induces a decomposition $\mathscr{P}=\mathscr{P}_1\sqcup \mathscr{P}_2$, {where $\mathscr{P}_i = \{ \lambda \in \mathscr{P} : \chi_\lambda \in \JH((D_0(\brho)^{(i)})^{I_1})\}$,} hence a decomposition $\gr_{\m}(\pi^{\vee})\cong N=N_1\oplus N_2$, with $N_i\defeq \bigoplus_{\lambda\in\mathscr{P}_i}\chi_{\lambda}^{-1}\otimes R/\mathfrak{a}(\lambda)$.
By construction, the degree $0$ part of $N_1$ is dual to $\pi_1^{I_1}$ and the degree $0$ part of $N_2$ is dual to $\bigoplus_{i\notin \Sigma}D_{0}(\brho)_i^{I_1}$ (as follows from the proof of \cite[Thm.~3.67]{BHHMS2}).

\textbf{Step 1.} Consider the induced short exact sequence
\[0\ra \gr_{F}(\pi_2^{\vee})\ra \gr_{\m}(\pi^{\vee})\ra \gr_{\m}(\pi_1^{\vee})\ra0,\]
where $F$ is the filtration on $\pi_2^{\vee}$ induced from the $\m$-adic filtration on $\pi^{\vee}$. 
The composite morphism $N_2\hookrightarrow N\twoheadrightarrow \gr_{\m}(\pi_1^{\vee})$ is identically zero, as $N_2$ is generated by its degree $0$ part, which is sent to zero in $\gr_{\m}(\pi_1^{\vee})$.
So we get an induced commutative diagram
\begin{equation}\label{eq:cycles-diagram-ss}
  \begin{gathered}
    \xymatrix{0\ar[r]& \gr_{F}(\pi_2^{\vee})\ar[r]& \gr_{\m}(\pi^{\vee})\ar[r]& \gr_{\m}(\pi_1^{\vee})\ar[r]&0 \\
      0\ar[r]& N_2\ar[r]\ar@{^{(}->}[u]& N\ar[r]\ar[u]^\cong & N_1\ar[r]\ar@{->>}[u]& 0}
  \end{gathered}
\end{equation}
with injective (resp.\ surjective) vertical map on the left (resp.\ right).
Thus 
\begin{equation}\label{eq:cycles-ineq-ss}\mathcal{Z}(N_1)\geq \mathcal{Z}(\gr_{\fm}(\pi_1^{\vee})),\ \ \mathcal{Z}(N_2)\leq \mathcal{Z}(\gr_F(\pi_2^{\vee})),\end{equation}
{where we use here the characteristic cycle of $\o R$-modules defined in (\ref{eq:cycleintro}) (see \cite[\S~3.3.4]{BHHMS2}).}

\textbf{Step 2.} We show that $\gr_\m(\pi_1^{\vee})$ and $\gr_F(\pi_2^{\vee})$ are Cohen--Macaulay.

Recall that by assumption $\pi$ satisfies assumption~\ref{it:assum-iii} in \S~\ref{sec:theorem}, namely $\EE_{\Lambda}^{2f}(\pi^\vee) \cong \pi^\vee \otimes \eta$ as $\GL_2(K)$-representations, where $\eta \defeq  \det(\rhobar)\omega^{-1}$.
As in the proof of \cite[Prop.\ 3.87(iii)]{BHHMS2} we may construct a subrepresentation $\widetilde{\pi}_2\subset \pi$ such that $\mathcal{Z}(\gr(\pi_2^{\vee}))=\mathcal{Z}(\gr(\widetilde{\pi}_2^{\vee}))$ (with respect to any good filtrations by \cite[Lemma 3.81]{BHHMS2}) and consequently by \cite[Prop.~3.87(i)]{BHHMS2}:
\begin{equation}\label{eq:D-xi}
  \dim_{\F\ppar{X}}D_{\xi}^{\vee}(\pi_2)=\dim_{\F\ppar{X}}D_{\xi}^{\vee}(\widetilde{\pi}_2).
\end{equation}
Concretely, the $\GL_2(K)$-representation $\wt\pi_2$ is defined by dualizing (and untwisting) the exact sequence
\begin{equation}\label{eq:tilde-pi2}
  0 \to \EE_{\Lambda}^{2f}(\pi_1^\vee) \to \EE_{\Lambda}^{2f}(\pi^\vee) \to \wt\pi_2^\vee \otimes \eta \to 0.
\end{equation}
The first two terms carry their canonical filtrations (\S~\ref{sec:append-canon-filtr}) and the morphism between them is strict by Lemma~\ref{lem:E-strict}.
We give $\wt\pi_2^\vee \otimes \eta$ the induced filtration, so that the induced sequence of their graded modules is again exact. 
We consider the following commutative diagram  with exact rows of graded $\gr(\Lambda)$-modules with compatible $H$-action, 
 where the upper vertical maps are explained above and the lower vertical maps arise from Step 1:
\begin{equation*}
  \xymatrix{0 \ar[r] & \gr(\EE_{\Lambda}^{2f}(\pi_1^\vee)) \ar[r]\ar@{_{(}->}[d] & \gr(\EE_{\Lambda}^{2f}(\pi^\vee)) \ar[r]\ar[d]^{\cong} & \gr(\wt\pi_2^\vee \otimes \eta) \ar[r]\ar@{-->>}[dd] & 0 \\
  0 \ar[r] & \EE_{\gr(\Lambda)}^{2f}(\gr_\m(\pi_1^\vee)) \ar[r]\ar@{_{(}->}[d] & \EE_{\gr(\Lambda)}^{2f}(\gr_\m(\pi^\vee)) \ar@{=}[d] \\ 
  0 \ar[r] & \EE_{\gr(\Lambda)}^{2f}(N_1) \ar[r] & \EE_{\gr(\Lambda)}^{2f}(N) \ar[r] & \EE_{\gr(\Lambda)}^{2f}(N_2) \ar[r] & 0.
}  
\end{equation*}

The surjection on the right gives us surjections of $H$-modules
$\gr(\wt\pi_2^\vee \otimes \eta) \onto \EE_{\gr(\Lambda)}^{2f}(N_2)
\onto \F\otimes_{\gr(\Lambda)} \EE_{\gr(\Lambda)}^{2f}(N_2)$, where
the final graded $\F$-vector space is supported in degrees $[3f,4f]$
by Corollary~\ref{cor:dual-of-N} (noting that $N_2$ is a direct factor of $N$).
In particular, by the semisimplicity of $\F[H]$, we deduce that $\F\otimes_{\gr(\Lambda)} \EE_{\gr(\Lambda)}^{2f}(N_2)$ is a subquotient of $F_{4f}(\wt\pi_2^\vee \otimes \eta)/F_{3f-1}(\wt\pi_2^\vee \otimes \eta)$ as $H$-modules. 
The same corollary applied to $N$ implies that $\gr(\EE_{\Lambda}^{2f}(\pi^\vee))$ is supported in degrees $\le 4f$, so $F_{4f}(\EE_{\Lambda}^{2f}(\pi^\vee)) = \EE_{\Lambda}^{2f}(\pi^\vee)$. 
Hence $F_{4f}(\wt\pi_2^\vee \otimes \eta) = \wt\pi_2^\vee \otimes \eta$ by~\eqref{eq:tilde-pi2}, so $\m^{f+1} \wt\pi_2^\vee \otimes \eta \subset F_{3f-1}(\wt\pi_2^\vee \otimes \eta)$. 
It follows from all this that $\F\otimes_{\gr(\Lambda)}\EE_{\gr(\Lambda)}^{2f}(N_2)$ is a subquotient of $(\wt\pi_2^\vee/\m^{f+1} \wt\pi_2^\vee) \otimes \eta$, or equivalently of $\bigoplus_{i=0}^f \gr_\m(\wt\pi_2^\vee)_i \otimes \eta$, as $H$-modules.

We have $\EE_{\gr(\Lambda)}^{2f}(N_2) \otimes \eta^{-1} \cong N_2'$ as $\gr(\Lambda)$-modules (without grading), where $N_2' \defeq  \bigoplus_{\lambda \in \mathscr P_2^*} \chi_{\lambda}^{-1} \otimes R/\mathfrak a(\lambda)$, by \cite[Prop.\ 3.66]{BHHMS2}.
Corollary~\ref{cor:involution-components} implies that $(N_2')_0$ is dual to $\bigoplus_{i \notin \Sigma} D_0(\rhobar)_{f-i}^{I_1}$.
On the other hand, as at the beginning of the proof, we have $\wt\pi_2^{K_1} = \bigoplus_{i \in \Sigma'} D_0(\rhobar)_{i}$ for some $\Sigma'\subset \{0,1,\dots,f\}$.
Let $\wt N_2$ be the direct summand of $N$ such that its degree 0 part is dual to $\wt\pi_2^{I_1} = \bigoplus_{i \in \Sigma'} D_0(\rhobar)_{i}^{I_1}$.
Then as before we have a surjection $\wt N_2 \onto \gr_\m(\wt\pi_2^\vee)$.
From the previous paragraph, $(N_2')_0\cong \F\otimes_{\gr(\Lambda)}N'_2$ is a subquotient of $\bigoplus_{i=0}^f (\wt N_2)_{-i}$ as $H$-modules.
But $\bigoplus_{i=0}^f N_{-i}$ is multiplicity free as $H$-module by Lemma \ref{lem:N/I-multifree} (with $n = f+1$ and $r=1$, {using that $\brho$ is $(2f+1)$-generic}) 
and $(N_2')_0 \subset N_0$, 
so we deduce that $(N_2')_0 \subset (\wt N_2)_0$ as $H$-modules (do not confuse the graded piece $N_i$ of $N$ for $i=1,2$ with the submodules $N_1$, $N_2$ of $N$ defined just before Step $1$!). Dually, $\bigoplus_{i \in \Sigma'} D_0(\rhobar)_{i}^{I_1}$ surjects onto $\bigoplus_{i \notin \Sigma} D_0(\rhobar)_{f-i}^{I_1}$ as $H$-modules.
In particular, $\Sigma' \supseteq f-\Sigma^c$, i.e.\
\begin{equation}\label{eq:diagram-contains}
  \wt\pi_2^{K_1} = \bigoplus_{i \in \Sigma'} D_0(\rhobar)_{i} \supseteq \bigoplus_{i \notin \Sigma} D_0(\rhobar)_{f-i}.
\end{equation}
Taking $\GL_2(\cO_K)$-socles we get
\begin{equation}\label{eq:soc-length}
  \begin{aligned}
    \ell(\soc_{\GL_2(\cO_K)} (\wt\pi_2)) = \sum_{i \in \Sigma'} \ell(\soc_{\GL_2(\cO_K)} (D_0(\rhobar)_i)) &\ge \sum_{i \notin \Sigma} \ell(\soc_{\GL_2(\cO_K)} (D_0(\rhobar)_{f-i}))\\ 
    &= \sum_{i \notin \Sigma} \ell(\soc_{\GL_2(\cO_K)} (D_0(\rhobar)_i))\\
    &= \ell(\soc_{\GL_2(\cO_K)} (\pi)) - \ell(\soc_{\GL_2(\cO_K)} (\pi_1)).
  \end{aligned}
\end{equation}
By~(\ref{eq:D-xi}) and exactness of the functor $D_{\xi}^{\vee}$ we know that 
\begin{equation*}
\dim_{\F\ppar{X}}D_{\xi}^{\vee}(\widetilde{\pi}_2)
= \dim_{\F\ppar{X}}D_{\xi}^{\vee}(\pi_2) = \dim_{\F\ppar{X}}D_{\xi}^{\vee}(\pi) - \dim_{\F\ppar{X}}D_{\xi}^{\vee}(\pi_1)
\end{equation*}
and hence by \cite[Prop.\ 3.87(ii)]{BHHMS2} that equality has to hold in~(\ref{eq:soc-length}) and hence in~(\ref{eq:diagram-contains}).
By taking $I_1$-invariants in~(\ref{eq:diagram-contains}) we deduce that $N_2' = \wt N_2$.

Consider
\begin{equation}\label{eq:cycles-ineq2-ss}
  \mathcal{Z}(\gr_F(\pi_2^{\vee})) \ge \mathcal{Z}(N_2) = \mathcal{Z}(N_2') = \mathcal{Z}(\wt N_2)\ge \mathcal{Z}(\gr_\m(\wt \pi_2^{\vee})),
\end{equation}
where the first inequality is equation~(\ref{eq:cycles-ineq-ss}), the first equality comes from \cite[Thm.\ 3.83]{BHHMS2}, the second equality holds as $N_2' = \wt N_2$,
and the final inequality comes from $\wt N_2 \onto \gr_\m(\wt\pi_2^\vee)$.
As $\mathcal{Z}(\gr(\pi_2^{\vee})) = \mathcal{Z}(\gr(\wt\pi_2^{\vee}))$,
we deduce that equality holds in~(\ref{eq:cycles-ineq2-ss}), so
$\mathcal{Z}(N_2) = \mathcal{Z}(\gr_F(\pi_2^{\vee}))$ and hence also
$\mathcal{Z}(N_1)= \mathcal{Z}(\gr_{\fm}(\pi_1^{\vee}))$ by the
additivity of $\mathcal{Z}$ in short exact
sequences {(recalling diagram \eqref{eq:cycles-diagram-ss})}. 
Since $N_1$ is pure, any of its nonzero submodules has a nonzero cycle, hence the surjection $N_1\twoheadrightarrow \gr_{\m}(\pi_1^{\vee})$ must be an isomorphism and consequently $\gr_{F}(\pi_2^{\vee})\cong N_2$ by Step~1.
This implies that $\gr_{\m}(\pi_1^{\vee}) \cong N_1$ and $\gr_{F}(\pi_2^{\vee})\cong N_2$ are Cohen--Macaulay, as $N$ is Cohen--Macaulay and  the $N_i$ are direct summands of $N$. Hence $\pi_1^\vee$ and $\pi_2^\vee$ are Cohen--Macaulay, because if a finitely generated $\Lambda$-module $M$ admits a good filtration such that the associated graded module is Cohen--Macaulay, then $M$ itself is Cohen--Macaulay as a consequence of \cite[Prop.~III.2.2.4]{LiOy}.

{The final statements follow: we already noted that $\soc_{\GL_2(\cO_K)}(\pi_1)$ determines $\pi_1^{K_1}$, which in turn determines $\pi_1^{I_1}$ and hence $N_1$, as
  $\mathscr{P}_1 = \{ \lambda \in \mathscr{P} : \chi_\lambda \in \JH(\pi_1^{I_1})\}$.}
\end{proof}

\begin{thm}\label{thm:length:f+1}
{Assume that 
$\rhobar$ is $\max\{9,2f+1\}$-generic.}
\begin{enumerate}
\item Any subrepresentation of $\pi$ is generated by its $\GL_2(\cO_K)$-socle.
\item $\ell_{\GL_2(K)}(\pi) \leq f+1$.
\end{enumerate}
\end{thm}

Note that part (i) for $\pi$ itself was proved in \cite[Thm.\ 3.89]{BHHMS2} {under a slightly weaker genericity assumption}.

\begin{proof}
Let $\pi_1$ be a subrepresentation of $\pi$, and $\pi_1'$ be the
subrepresentation of $\pi_1$ generated by
$\soc_{\GL_2(\cO_K)}(\pi_1)$. In particular,
$\soc_{\GL_2(\cO_K)}(\pi_1)=\soc_{\GL_2(\cO_K)}(\pi_1')$,
so Proposition \ref{prop:split-I1} implies that $\gr_{\m}(\pi_1^{\vee})$ and
$\gr_{\m}(\pi_1'^{\vee})$ are isomorphic and finitely generated as $\gr(\Lambda)$-modules, cf.\ \eqref{eq:split-I1-N1}. 
Therefore, as $\gr(\Lambda)$ is noetherian, the
surjective map
\[ \gr_{\m}(\pi_1^{\vee})\twoheadrightarrow \gr_{\m}(\pi_1'^{\vee}) \]
(induced by the surjection $\pi_1^\vee \onto \pi_1'^\vee$) is an isomorphism. Hence, we deduce $\gr_{\m}(\pi_1^{\vee})=\gr_{\m}(\pi_1'^{\vee})$, from which we deduce $\pi_1^{\vee}/\m^n\congto \pi_1'^{\vee}/\m^n$ for all $n\geq 1$ for dimension reasons and hence $\pi_1=\pi_1'$. This proves (i).

To prove (ii), it suffices to show that any finite ascending chain of $\GL_2(K)$-subrepresentations $0 = \pi_0 \subsetneq \pi_1 \subsetneq \cdots \subsetneq \pi_{\ell} = \pi$ has length $\ell \le f+1$.
 It follows from Proposition \ref{prop:split-I1} again that  
we can write $\pi_j^{K_1} = \bigoplus_{i \in \Sigma_j} D_0(\rhobar)_i$ for unique subsets $\varnothing = \Sigma_0 \subset \cdots \subset \Sigma_{\ell}= \{0,1,\dots,f\}$.
Since $\pi_j^{K_1}$ contains $\soc_{\GL_2(\cO_K)}(\pi_j)$, we deduce from (i) that $\Sigma_j \subsetneq \Sigma_{j+1}$ for all $0 \le j < \ell$, so indeed $\ell \le f+1$.
\end{proof}

We now note further consequences of Proposition \ref{prop:split-I1}.

\begin{cor}\label{cor:split-I1}
  Keep the notation of Proposition~\ref{prop:split-I1} and suppose that $\rhobar$ is $\max\{9,2f+1\}$-generic.
  \begin{enumerate}
  \item The $\m$-adic filtration on $\pi^{\vee}$ induces the $\m$-adic filtration on $\pi_2^{\vee}$.
  \item The induced sequence \[0\ra \gr_\m(\pi_2^\vee)\ra \gr_\m(\pi^\vee)\ra \gr_\m(\pi_1^\vee)\ra0\] 
    of graded $\gr(\Lambda)$-modules with compatible $H$-action is split exact.
    More precisely,
    \begin{equation*}
      \gr_\m(\pi_1^{\vee}) \cong \bigoplus_{\lambda\in\mathscr{P}_1}\chi_{\lambda}^{-1}\otimes \frac{R}{\mathfrak{a}(\lambda)}
    \end{equation*}
    and
    \begin{equation*}
      \gr_\m(\pi_2^{\vee}) \cong \bigoplus_{\lambda\in\mathscr{P}\setminus \P_1}\chi_{\lambda}^{-1}\otimes \frac{R}{\mathfrak{a}(\lambda)},
    \end{equation*}
    where $\P_1 \subset \P$ corresponds to $\pi_1^{I_1} \subset \pi^{I_1}$ {(see \S~\ref{sec:notation})}.
  \end{enumerate}
\end{cor}

\begin{proof}
We keep the notation of the proof of Proposition~\ref{prop:split-I1}.

(i) By the isomorphism $\gr_F(\pi_2^{\vee}) \cong N_2$ proved in Step 2 of the proof of Proposition~\ref{prop:split-I1}, $\gr_{F}(\pi_2^{\vee})$ is generated by its degree 0 part $\gr_{F}(\pi_2^{\vee})_0$ as a $\gr(\Lambda)$-module.
Since $\m^n\pi_2^{\vee}\subset\pi_2^{\vee} \cap \m^n\pi^{\vee}=F_{-n}\pi_2^{\vee}$, we have the natural morphism \[\kappa: \gr_{\m}(\pi_2^{\vee})\ra \gr_{F}(\pi_2^{\vee})\cong N_2,\]
which is surjective in degree $0$ as  
$\m^0\pi_2^{\vee}=F_0\pi_2^{\vee}\ (=\pi_2^{\vee})$. 
Since $N_2$ is generated by its degree $0$ part, $\kappa$ is surjective and it follows from \cite[Thm.\ I.4.2.4(5)]{LiOy} (applied with $L=M=\pi_2^{\vee}$ and $N=0$) that $\m^n\pi_2^{\vee}=F_{-n}\pi_2^{\vee}$ for all $n\geq 0$.

Part (ii) follows, since the sequence $0 \to N_2 \to N \to N_1 \to 0$ is split exact by construction.
\end{proof}

\begin{cor}\label{cor:split-In}
Suppose that $\rhobar$ is $\max\{9,2f+1\}$-generic.
Let $\pi_1\subset\pi_2$ be  subrepresentations of $\pi$. Then for any $n\geq 1$, the sequence of $\Lambda$-modules
\[0\ra \pi_1[\m^n]\ra \pi_2[\m^n]\ra (\pi_2/\pi_1)[\m^n]\ra 0\]
is exact. 
Moreover, the sequence splits as $I$-representations if $\rhobar$ is {also $(2n-1)$-generic}. 
\end{cor}

\begin{proof}
We first treat the special case $\pi_2=\pi$. Then we trivially have $0\ra \pi_1[\m^n]\ra \pi[\m^n]\ra (\pi/\pi_1)[\m^n]$. 
The final map is surjective for dimension reasons because $\gr_{\m}(\pi^{\vee})\cong\gr_{\m}(\pi_1^{\vee})\oplus \gr_{\m}((\pi/\pi_1)^{\vee})$ by Corollary \ref{cor:split-I1}(ii).  In particular, for any subrepresentation $\pi_1$ of $\pi$ we obtain 
\begin{equation}\label{eq:dimequality}\dim_{\F}((\pi/\pi_1)[\m^n])=\dim_{\F}(\pi[\m^n])-\dim_{\F}(\pi_1[\m^n]).\end{equation}

Now we treat the general case. Since $\pi[\m^n]\ra (\pi/\pi_2)[\m^n]$ is surjective by the last paragraph,  the morphism    
\[(\pi/\pi_1)[\m^n]\ra (\pi/\pi_2)[\m^n]\]
is also surjective, and hence the sequence \[0\ra (\pi_2/\pi_1)[\m^n]\ra (\pi/\pi_1)[\m^n]\ra (\pi/\pi_2)[\m^n]\ra 0\] is exact.
Applying \eqref{eq:dimequality} to $\pi_1$ and $\pi_2$, we deduce 
\[\dim_{\F}((\pi_2/\pi_1)[\m^n])=\dim_{\F}(\pi_2[\m^n])-\dim_{\F}(\pi_1[\m^n]),\]
from which the first assertion follows.

For the last assertion, it suffices to show that $\pi_1[\m^n]$ is a direct summand of $\pi[\m^n]$ (hence is also a direct summand of $\pi_2[\m^n]$, see \eqref{eq:directsumm} in the proof of Lemma \ref{lem:G'}). {As $\brho$ is $(2n-1)$-generic} we note that $\pi[\m^n]=\tau^{(n)}[\m^n]$ by Lemma~\ref{lem:isom-modcI}, where $\tau^{(n)}=\bigoplus_{\lambda\in\mathscr{P}}\tau_{\lambda}^{(n)}$ is the subrepresentation of $\pi|_I$ from Lemma \ref{lem:tau-embed}.  Let $\mathscr{P}_1\subset \mathscr{P}$ be the subset as in the proof of Proposition \ref{prop:split-I1} and put 
\[\tau_1^{(n)}\defeq \bigoplus_{\lambda\in\mathscr{P}_1}\tau_{\lambda}^{(n)}, \quad N_1\defeq \bigoplus_{\lambda\in\mathscr{P}_1}\chi_{\lambda}^{-1}\otimes \frac{R}{\mathfrak{a}(\lambda)}.\]
It suffices to show that $\pi_1[\m^n]=\tau_1^{(n)}[\m^n]$, or equivalently (as $\pi[\m^n]$ is multiplicity free) that these $\Lambda$-modules (with compatible $H$-action) have the same graded modules. This follows from the isomorphism
$\gr_{\m}(\pi_1^{\vee})\cong N_1$
established in the proof of Proposition \ref{prop:split-I1}, noting that \[\gr_{\m}((\tau_1^{(n)})^{\vee}/\m^n)=\gr_{\m}((\tau_1^{(n)})^{\vee})/\o\m^n=N_1/\overline{\m}^n\] by the proof of Lemma \ref{lem:isom-modcI}, where $\overline{\m}$ denotes the unique maximal graded ideal of $\gr(\Lambda)$.
\end{proof}

\begin{lem}\label{lem:soc-exact}
Suppose that $\rhobar$ is $\max\{9,2f+1\}$-generic.
Let $\pi_1\subset \pi_2$ be subrepresentations of $\pi$. Then the natural sequence 
\begin{equation}\label{eq:socle}0\ra \soc_{\GL_2(\cO_K)}(\pi_1)\ra \soc_{\GL_2(\cO_K)}(\pi_2)\ra \soc_{\GL_2(\cO_K)}(\pi_2/\pi_1)\ra0\end{equation}
is exact. 
\end{lem}
\begin{proof}

By the second paragraph of the proof of Proposition \ref{prop:split-I1} there exist two subsets $\Sigma_1\subset\Sigma_2$ of $\{0,\dots,f\}$ such that for $j\in \{1,2\}$,
\[\pi_j^{K_1}=\bigoplus_{i\in\Sigma_j}D_{0}(\brho)_i,\ \ \pi_j^{I_1}=\bigoplus_{i\in\Sigma_j}D_0(\brho)_i^{I_1}.\] 
Setting $\pi'\defeq \pi_2/\pi_1$, we deduce that $\pi'^{I_1}\cong \bigoplus_{i\in\Sigma_2\backslash \Sigma_1}D_{0}(\brho)_i^{I_1}$ by Corollary \ref{cor:split-In} (applied with $n=1$), and also that there exists an embedding
$\bigoplus_{i\in\Sigma_2\backslash\Sigma_1}D_0(\brho)_i\hookrightarrow \pi'^{K_1}.$
This in particular implies \[S\defeq \bigoplus_{i\in\Sigma_2\backslash \Sigma_1}\soc_{\GL_2(\cO_K)}(D_0(\brho)_i)\hookrightarrow \soc_{\GL_2(\cO_K)}(\pi').\]
We need to prove that it is an isomorphism. If not, then there exists some Serre weight $\sigma$ such that $\sigma\oplus S\hookrightarrow \pi'|_{\GL_2(\cO_K)}$, hence also $\sigma\oplus \big(\bigoplus_{i\in\Sigma_2\backslash \Sigma_1}D_{0}(\brho)_i\big)\hookrightarrow \pi'|_{\GL_2(\cO_K)}$, which contradicts the structure of $\pi'^{I_1}$. 
\end{proof}

\begin{cor}\label{cor:finite-length}
  Suppose that $\rhobar$ is $\max\{9,2f+1\}$-generic.
  Suppose $\pi'$ is any subquotient of $\pi$.
  \begin{enumerate}
  \item We have $\dim_{\F\ppar{X}} D_\xi^\vee(\pi') = \ell(\soc_{\GL_2(\cO_K)} (\pi'))$.
    In particular, if $\pi' \ne 0$, then $D_\xi^\vee(\pi')$ {is} nonzero.
  \item Let $\mathscr P' \subset \mathscr P$ correspond to $(\pi')^{I_1}$ (such a subset exists by Corollary~\ref{cor:split-In} with $n = 1$). Then the natural map
  \[\bigoplus_{\lambda \in \mathscr P'} \chi_\lambda^{-1} \otimes R/\mathfrak a(\lambda) \onto \gr_\m(\pi'^\vee)\]
  of \ graded \ $\gr(\Lambda)$-modules \ with \ compatible \ $H$-action \ is \ an \ isomorphism.
    \ In \ particular, $\gr_\m(\pi'^\vee)$ (resp.\ $\pi'^\vee$) is Cohen--Macaulay of grade $2f$.
  \item $\pi'$ is generated by its $\GL_2(\cO_K)$-socle.
  \item $\pi$ itself is multiplicity free (of length $\le f+1$).
  \item 
    We have an isomorphism $\EE^{2f}_\Lambda(\pi'^\vee)  \cong \pi''^\vee\otimes (\det(\brho)\omega^{-1})$ as $\Lambda$-modules with compatible actions of $\GL_2(K)$, where $\pi''$ is another subquotient of $\pi$, uniquely determined (by part (iv)) by
    \[\soc_{\GL_2(\cO_K)} (\pi'') \cong \bigoplus_{i \in \Sigma'} \soc_{\GL_2(\cO_K)} (D_0(\rhobar)_{f-i}). \]
  \end{enumerate}
\end{cor}

\begin{proof}
(i)  Choose $\pi_1\subset \pi_2\subset\pi$ such that $\pi'\cong\pi_2/\pi_1$.  By \cite[Prop.~3.87(ii)]{BHHMS2} the assertion holds for $\pi_1$ and $\pi_2$, so we conclude by the exactness of $D_{\xi}^{\vee}(-)$ (\cite[Thm.~3.29]{BHHMS2}) combined with Lemma \ref{lem:soc-exact}.

(ii) Let $\pi_1,\pi_2$ be as in (i). 
Let $\mathscr P_1 \subset \mathscr P_2 \subset \mathscr P$ be the subsets corresponding to $\pi_1 \subset \pi_2$ {(see \S~\ref{sec:notation})}, so $\mathscr P' = \mathscr P_2\setminus \mathscr P_1$ by the proof of Proposition~\ref{prop:split-I1}.
Let $N_1 \subset N_2$ (resp.\ $N'$) be the direct summands of $N$ determined by $\mathscr P_1 \subset \mathscr P_2$ (resp.\ $\mathscr P'$).
As in Step 1 of the proof of Proposition \ref{prop:split-I1} we get a commutative diagram 
\[\xymatrix{ 0 \ar[r] &\gr_{F}(\pi'^{\vee})\ar[r] & \gr_{\m}(\pi_2^{\vee})\ar[r] & \gr_{\m}(\pi_1^{\vee})\ar[r] & 0 \\
  0 \ar[r] &N' \ar[r]\ar[u] & N_2 \ar[r]\ar[u] & N_1 \ar[r]\ar[u] & 0}\]
with exact rows, where $F$ is the filtration on $\pi'^\vee$ induced from the $\m$-adic filtration on $\pi_2^{\vee}$, and 
by Step 2 of the proof of Proposition \ref{prop:split-I1}, the second and third vertical arrows are isomorphisms, hence so is the first.
As $0\ra \pi_1^{I_1}\ra \pi_2^{I_1}\ra \pi'^{I_1}\ra0$ is exact, we conclude that $F$ is the $\m$-adic filtration exactly as at the end of the proof of Corollary \ref{cor:split-I1}(i).

(iii) Let $\pi_1,\pi_2$ be as in (i). The assertion holds for subrepresentations of $\pi$ by Theorem \ref{thm:length:f+1}(i),  so $\pi_2$ is generated by $\soc_{\GL_2(\cO_K)}(\pi_2)$.  Thus $\pi'$  is generated by the image of $\soc_{\GL_2(\cO_K)}(\pi_2)$ in $\pi'$, which is contained in $\soc_{\GL_2(\cO_K)}(\pi')$ (even equal by Lemma~\ref{lem:soc-exact}).

(iv) It is clear by the exact sequence \eqref{eq:socle} in Lemma~\ref{lem:soc-exact}, since $\soc_{\GL_2(\cO_K)}(\pi)$ is multiplicity free.

(v) If $\pi'$ is a quotient of $\pi$, this is established in Step 2 of the proof of Proposition \ref{prop:split-I1}.
In general, if $\pi_1\subset \pi_2\subset\pi$ such that $\pi'\cong\pi_2/\pi_1$, then we get an exact sequence $0 \to \pi' \to \pi/\pi_1 \to \pi/\pi_2 \to 0$
and hence an exact sequence
\[0 \to \EE^{2f}_\Lambda(\pi'^\vee) \otimes \eta \to \EE^{2f}_\Lambda((\pi/\pi_1)^\vee) \otimes \eta \to \EE^{2f}_\Lambda((\pi/\pi_2)^\vee) \otimes \eta \to 0,\]
as $\pi'^\vee$ is Cohen--Macaulay by part (ii) and where $\eta \defeq  \det(\brho)\omega^{-1}$.
Then the claim follows from Lemma~\ref{lem:soc-exact} and the known case for quotient representations (cf.\ Step 2 of the proof of Proposition~\ref{prop:split-I1}).
\end{proof}

\section{Finite length in the nonsplit reducible case}
\label{sec:finite-length-nonss}

We prove that a smooth mod $p$ representation $\pi$ of $\GL_2(K)$ satisfying assumptions \ref{it:assum-i}--\ref{it:assum-iv} of \S~\ref{sec:theorem} with $r=1$ has finite length when the underlying Galois representation $\rhobar$ is \emph{nonsplit} reducible. We also establish several structural results on $\pi$ as an $I$- and $\GL_2(\cO_K)$-representation.

{{Unless} otherwise stated, we} assume that $\rhobar$ is nonsplit reducible and $0$-generic.
{We let $\pi$ be an admissible smooth representation of $\GL_2(K)$ over $\F$ satisfying assumptions \refeq{it:assum-i}--\refeq{it:assum-iv} of \S~\ref{sec:theorem}. 
Recall from \S~\ref{sec:verify-assumpt-iv} that $\pi=\pi(\brho)$ as defined in \S~\ref{sec:verify-assumpt-iv} satisfies assumption \ref{it:assum-iv} for any $r\geq 1$. 
{From \cite[Thm.~6.3(ii)]{YitongWangGKD} together with \cite[Prop.~6.4.6]{BHHMS1} (the hypotheses of the latter being checked in the proof of  \cite[Thm.~5.1]{YitongWangGKD})} one deduces that $\pi(\rhobar)$ satisfies assumptions \ref{it:assum-i} and \ref{it:assum-ii} of \S~\ref{sec:theorem}, for any $r\geq 1$.
It also satisfies assumption \ref{it:assum-iii} (again, for any $r\geq 1$) by \cite[Thm.~8.2]{HuWang2} with  \cite[Thm.~6.3(i)]{YitongWangGKD}}.

{As in \S~\ref{sec:finite-length-ss} we assume that $r=1$ {in assumptions \ref{it:assum-i} and \refeq{it:assum-iv}} throughout.}

\subsection{Preliminaries on Serre weights}
\label{sec:prel-serre-weights}

We collect a number of results on the combinatorics of Serre weights and injective envelopes.

Recall from \S~\ref{sec:notation} that $D_0(\brho)=\bigoplus_{\sigma\in W(\rhobar)}D_{0,\sigma}(\brho)$, and from \cite[\S~13]{BP} that $D_{0,\sigma}(\brho)$ is maximal (for the inclusion) with respect to the two properties $\soc_{\GL_2(\cO_K)}(D_{0,\sigma}(\brho))=\sigma$ and $\JH(D_{0,\sigma}(\brho)/\sigma)\cap W(\brho)=\emptyset$.
In particular, $D_{0,\sigma}(\brho^{\ss})\subset D_{0,\sigma}(\brho)$.

We {first} generalize Lemma \ref{lem:I1-invt-component} to the case where $\brho$ need not be semisimple.

\begin{lem}\label{lem:I1-invt-nonss}
If $\mu\in\P$, then $\chi_{\mu}$ occurs in $D_{0,\sigma}(\rhobar)^{I_1}$, where $\sigma\in W(\brho)$ 
is determined \textnormal{(}via \eqref{eq:J-lambda}\textnormal{)} by $J_{\sigma}=J_{\brho}\cap J_{\mu}$. 
 Moreover, as a Jordan--H\"older factor of $\Ind_I^{\GL_2(\cO_K)}\chi_{\mu}$, $\sigma$ is parametrized {(via Remark \ref{rmk:JH:conventions} and \eqref{eq:S-xi1}) by the following subset of $\{0,\dots,f-1\}$:}
 \begin{equation}\label{eq:X-mu}X(\mu)\defeq \{j : \mu_j(x_j)\in\{x_j,p-2-x_j,p-3-x_j\}\}\cup\{j\in J_{\brho}:\mu_j(x_j)=x_j+1\}.\end{equation}
\end{lem}

\begin{proof}

The proof goes as in \cite[Prop.~2.1]{Hu-SMF} and  we  only briefly recall it.

Let {$\lambda \in \D$ such that} $\sigma \in W(\brho)$ corresponds to $\lambda$.
It is clear that  $\sigma$ 
is a subquotient of $\Ind_I^{\GL_2(\cO_K)}\chi_{\mu}$, so 
{via Remark \ref{rmk:JH:conventions}} there is a unique $f$-tuple $\xi\in \mathcal{P}$ such that 
\begin{equation}\label{eq:xi2}\xi_j^{c}(\mu_j(x_j))=\xi_j(\mu^{[s]}_j(x_j))=\lambda_j(x_j)\end{equation}
for any $j\in\{0,\dots,f-1\}$, where 
\begin{equation}\label{eq:[s]}
\mu^{[s]}\defeq (p-1-\mu_0(x_0),\dots,p-1-\mu_{f-1}(x_{f-1}))\in\mathscr{P}.
\end{equation}
{Here, we used \cite[Lemma~2.1, Lemma~2.7]{HuWang2} to obtain \eqref{eq:xi2} (equality between formal $f$-tuples).}  

Note that our convention for $J \mapsto \xi_J$ is shifted by one compared to \cite[\S~2]{BP} and \cite[\S~2]{breuil-buzzati}.
Using the second equality in \eqref{eq:xi2}, \cite[Prop.~4.3]{breuil-buzzati} (and {the formula for $J^{\max}$ in }eq.~(19) in its proof, {replacing $\lambda$ there by our $\mu$} and noting that $\chi_\mu \ne \chi_\mu^s$) gives the following relation  
\begin{equation}\label{eq:comp2}
\begin{array}{rll}\xi_j(y_j)\in\{y_j-1,p-1-y_j\}&\Longleftrightarrow &\mu_j^{[s]}(x_j)\in\{x_j+1,x_j+2,\un{p-2-x_j},p-1-x_j\}\\
&\Longleftrightarrow&\mu_j(x_j)\in\{x_j,\un{x_j+1},p-2-x_j,p-3-x_j\},\end{array}\end{equation}
making the convention that an underlined entry is only allowed when $j \in J_{\rhobar}$.
We say that a  pair $(\xi,\mu)\in \mathcal{P}\times\mathscr{P}$ is \emph{compatible} if \eqref{eq:comp2} holds.

It is straightforward to list all the possibilities of compatible pairs  $(\xi,\mu)\in\mathcal{P}\times\mathscr{P}$ and verify that
\[\begin{array}{rll} \xi_j(\mu_j^{[s]}(x_j))=\lambda_j(x_j)\in\{x_j+1,p-3-x_j\} &\Longleftrightarrow& \mu_j^{[s]}(x_j)\in \{x_j+2,\un{p-2-x_j},p-3-x_j\}\\
&\Longleftrightarrow& \mu_j(x_j)\in \{\un{x_j+1},x_j+2,p-3-x_j\}.  \end{array}\]
The left-hand side is equivalent to $j \in J_\sigma = J_\lambda$ and the right-hand side is equivalent to $j \in J_\mu \cap J_{\rhobar}$ by \eqref{eq:J-lambda}.
 The second part results from~(\ref{eq:comp2}).
\end{proof} 

Let $\sigma$ be a 1-generic Serre weight.
Recall that the set of Jordan--H\"older factors of $\Inj_{\Gamma}\sigma$ is parametrized by a set of $f$-tuples denoted by $\mathcal{I}\defeq \mathcal{I}(x_0,\dots,x_{f-1})$ in \cite[\S~3]{BP} {(do not confuse this $\mathcal I$ with the ideal $\mathcal I$ before Lemma \ref{lem:min-graded-splitting}!)}. Given $\lambda\in \mathcal{I}$ we write \[\cS(\lambda)\defeq \big\{j\in\{0,\dots,f-1\}: \lambda_j(x_j)\in \{x_j\pm 1, p-2-x_j\pm1\}\big\}\] as in \cite[\S~4]{BP}. (This notation is consistent with~\eqref{eq:S-xi1}, noting that $\mathcal{P} \subset \mathcal{I}$.)
 
The following lemma is true for any $0$-generic $\brho$. 
\begin{lem}\label{lem:Dss-in-D0} 
We have $W(\brho^{\rm ss})\subset \JH(D_0(\brho))$.
\end{lem}
\begin{proof}
By the construction of $D_0(\brho)$ (see \cite[Prop.~13.1]{BP}) and \cite[Prop.~13.4]{BP} we have 
\begin{equation}\label{eq:JH-D0-rho}
  \JH(D_0(\brho)) = \JH\bigg(\bigoplus_{\sigma\in W(\brho)}\Inj_{\Gamma}\sigma\bigg).
\end{equation}
Thus  it suffices to prove that
\[W(\brho^\ss)\subset \JH\bigg(\bigoplus_{\sigma\in W(\brho)}\Inj_{\Gamma}\sigma\bigg).\] 
But it is clear from \cite[Lemma 3.2, Lemma 11.2]{BP} that $W(\brho^\ss)\subset \JH(\Inj_{\Gamma}\sigma_0)$, where $\sigma_0\in W(\brho)$ denotes the unique Serre weight corresponding to $(x_0,\dots,x_{f-1})\in \mathscr{D}$. 
\end{proof}

Recall from \cite[Cor.~3.12]{BP} that given a 0-generic Serre weight $\sigma$ and $\tau\in \JH(\Inj_{\Gamma}\sigma)$, there exists a unique finite dimensional $\Gamma$-representation $I(\sigma,\tau)$ such that $\soc_{\Gamma}I(\sigma,\tau)=\sigma$, $\cosoc_{\Gamma}I(\sigma,\tau)=\tau$ and $[I(\sigma,\tau):\sigma]=1$. 
{If $\sigma$ is $1$-generic, \cite[Cor.~4.11]{BP} implies that $I(\sigma,\tau)$ has length $2^{|\mathcal{S}(\lambda)|}$, where $\lambda\in\mathcal{I}$ corresponds to $\tau$. Recall that any $\sigma\in W(\brho^{\rm ss})$ is $n$-generic if $\brho$ is $n$-generic.}
 
\begin{lem}\label{lem:JHinW}
 Assume that $\rhobar$ is $0$-generic. 
Let $\tau\in W(\brho^\ss)$ and $\sigma\in W(\brho)$ be the unique Serre weight determined by $J_{\sigma}=J_{\brho} \cap J_{\tau}$ {\textnormal{(}via \eqref{eq:J-lambda}\textnormal{)}}. Then $\tau\in \JH(D_{0,\sigma}(\brho))$ and the Jordan--H\"older factors of $I(\sigma,\tau)$ are exactly the {Serre weights} $\tau'\in W(\brho^\ss)$ satisfying $J_{\sigma}\subset J_{\tau'}\subset J_{\tau}$. In particular, $\ell(\tau')\leq \ell(\tau)$ for any $\tau'\in \JH(I(\sigma,\tau))$, with equality if and only if $\tau'=\tau$. 
\end{lem}
\begin{proof}
The assertion $\tau\in \JH(D_{0,\sigma}(\brho))$ follows directly from \cite[Lemma 15.3]{BP} (note that the condition $\ell(\brho,\tau)<+\infty$ in \emph{loc.~cit.}~is satisfied by Lemma \ref{lem:Dss-in-D0}). 
To verify the remaining claim, for any subset $J \subset \{0,1,\dots,f-1\}$ let $\sigma_J \in W(\brho^\ss)$ determined by $J_{\sigma_J} = J$.
From \cite[Cor.~4.11]{BP} we deduce that $I(\sigma_\emptyset,\sigma_{\{0,\dots,f-1\}})$ is of length $2^f$ with constituents all $\sigma_J$ ($J \subset \{0,1,\dots,f-1\}$).
Moreover, the proof of \emph{loc.~cit.} (referring to \cite[Thm.~4.7]{BP}) shows that the lattice of submodules is isomorphic to the lattice of ideals of the partially ordered set $(\{0,\dots,f-1\},\subset)$, by sending a submodule $M$ to the ideal $\{J : \sigma_J \in \JH(M)\}$.
The claim follows, since $\sigma = \sigma_{J_{\brho} \cap J_\tau}$ and $\tau = \sigma_{J_\tau}$.
\end{proof}

\begin{lem}\label{lem:J-sets-add-up}
  Suppose that $\lambda \in \P$ and that $J \defeq  \{ j \in J_{\rhobar}^c : \lambda_j(x_j) \in \{x_j,p-1-x_j\}\}$.
  Then $|J_\lambda|+|J_{\lambda^*}|+|J| = f$, where $\lambda\mapsto \lambda^*$ is the involution of $\P$ defined in \cite[Def.~3.62]{BHHMS2}.
\end{lem}

\begin{proof}
  This follows directly from~(\ref{eq:id:al}) and \cite[Def.~3.62]{BHHMS2}.
\end{proof}

\subsection{Some commutative algebra}
\label{sec:some-algebra}

We prove that certain explicit $\o R$-modules are Cohen--Macaulay.

Recall from \S~\ref{sec:notation} that $R = \F[y_j,z_j : 0 \le j \le f-1]$ and $\o R = \F[y_j,z_j : 0 \le j \le f-1]/(y_j z_j : 0 \le j \le f-1)$.

\begin{lem}\label{lem:CM-graded-R-mod}
  Suppose that $M$ is a nonzero finitely generated graded $R$-module.
  Then $M$ is Cohen--Macaulay (in the sense of commutative algebra) if and only if $\EE^i_R(M) = 0$ for all $i \ne j_R(M)$.
\end{lem}

\begin{proof}
  Let $\m = (y_j,z_j : 0 \le j \le f-1)$ denote the unique maximal graded ideal of $R$.
  Then $M$ is a Cohen--Macaulay $R$-module if and only if $M_\m$ is a Cohen--Macaulay $R_\m$-module (\cite[Cor.\ 2.2.15]{BH93}) if and only if $\EE^i_{R_\m}(M_\m) = 0$ for all but one $i$ (\cite[Cor.\ 3.5.11]{BH93}, as $R_\m$ is regular) if and only if $\EE^i_{R}(M) = 0$ for all but one $i$ (using $\EE^i_{R}(M) \otimes_R R_\m \cong \EE^i_{R_\m}(M_\m)$ and \cite[Prop.\ 1.5.15(c)]{BH93}).
  By definition, $\EE^{j_R(M)}(M) \ne 0$.
\end{proof}

\begin{lem}\label{lem:basic-CM-module}
  Suppose \ that \ $t_j \in \{y_j,z_j,y_jz_j\}$ \ for \ $0 \le j \le f-1$. \ Then \ the \ $R$-module \ $\o R/(t_0,\dots,t_{f-1})$ is Cohen--Macaulay of grade $f$.
\end{lem}

\begin{proof}
  As $\o R/(t_0,\dots,t_{f-1})$ is a Cohen--Macaulay $\gr(\Lambda)$-module of grade $2f$ by the beginning of the proof of Theorem \ref{thm:CMC} in \S~\ref{sec:proof-theorem}, the result follows from \cite[Lemma 3.65]{BHHMS2}.
\end{proof}

\begin{prop}\label{prop:basic-CM-module}
  Suppose that $1 \le d \le f$.
  Let $I_d$ be the homogeneous ideal of $\o R$ generated by all monomials $z_{i_1} \cdots z_{i_d}$ with $0 \le i_1 < \cdots < i_d \le f-1$.
    Then the $R$-module $\o R/I_d$ is Cohen--Macaulay of grade $f$.
\end{prop}

\begin{proof}
  If $d = 1$ this follows from Lemma~\ref{lem:basic-CM-module}, so we suppose $d \ge 2$. Then the ring $\o R/I_d = R/(y_j z_j, z_{i_1} \cdots z_{i_d})$ (all $j$, all $0 \le i_1 < \cdots < i_d \le f-1$) is the Stanley--Reisner ring $\F[\Delta]$ associated to the simplicial complex $\Delta$ whose minimal non-faces $\{y_j, z_j\}$,\ $\{z_{i_1}, \dots, z_{i_d}\}$ correspond to the generators \cite[\S~5.1]{BH93}.
  Thus $\Delta$ is the pure $(f-1)$-dimensional simplicial complex with facets $\un x = \{x_0, \dots, x_{f-1}\}$, where $x_j \in \{ y_j, z_j \}$, $|\{j: x_j = z_j\}| < d$.
  For a facet $\un{x}=\{x_0, \dots, x_{f-1}\}$ let $J(\un x) \defeq  \{j: x_j = z_j\}$.  
  
  We prove that $\Delta$ is shellable \cite[Def.\ 5.1.11]{BH93}, which implies that $\F[\Delta]$ is a Cohen--Macaulay ring by \cite[Thm.\ 5.1.13]{BH93}.
  To see this, we order the facets as $\un x^{(0)}$, $\un x^{(1)}$, \dots\ such that $|J(\un x^{(0)})| \le |J(\un x^{(1)})| \le \cdots$ is non-decreasing.
  Then, using the notation of \cite[\S~5.1]{BH93}, for any $i_0 > 0$ the intersection $\langle \un x^{(0)},\dots,\un x^{(i_0-1)}\rangle \cap \langle\un x^{(i_0)}\rangle$ is generated by the maximal proper faces of $\un x^{(i_0)}$ that are of the form $\un x^{(i_0)} \setminus \{ x_j^{(i_0)} \}$ for some $j \in J(\un x^{(i_0)})$, proving shellability.

  Let $S \defeq  \o R/I_d = \F[\Delta]$, which is graded of dimension $f$ \cite[Thm.\ 5.1.4]{BH93}, and let $\m$ denote the unique maximal graded ideal of $R$ (or its image in $S$).
  As $S$ is a Cohen--Macaulay ring, it is also a Cohen--Macaulay $R$-module \cite[\S~2.1]{BH93}.
  We compute
  \begin{equation*}
    j_R(S) = j_{R_\m}(S_\m) = \dim R_\m - \dim_{R_\m} S_\m = \dim R - \dim_R S = 2f-f = f,
  \end{equation*}
  where we used \cite[Prop.\ 1.5.15(e)]{BH93} for the first equality, \cite[Cor.\ 3.5.11]{BH93} for the second equality, and \cite[Ex.\ 1.5.25]{BH93} for the third equality.
  (Alternatively, it follows from \cite[Thm.\ 5.7.3]{BH93} that $\Ext^f_R(S,R) \ne 0$.)
\end{proof}

\begin{definit}
\label{def:IJideals}
Suppose that $J_1$, $J_2$ are disjoint subsets of $\{0,\dots,f-1\}$ and that $d \in \Z$. 
We define the ideal $I(J_1,J_2,d)$ of $\o R$ as follows:
if $d \ge 1$ let $I(J_1,J_2,d)$ be generated by all $\prod_{j \in J_1'} y_j \prod_{j \in J_2'} z_j$ with $J_1' \subset J_1$, $J_2' \subset J_2$, $|J'_1|+|J'_2| = d$;
if $d\leq 0$, let $I(J_1,J_2,d) \defeq  \o R$.
{Suppose moreover that} $t_j \in \{y_j,z_j,y_jz_j\}$ for all $0 \le j \le f-1$, {we define the ideal} $I(J_1,J_2,d,\un t)$ {of $\o R$ as} $I(J_1,J_2,d) + (t_0,\dots,t_{f-1})$.
\end{definit}

\begin{cor}\label{cor:CM-module-kunneth}
  If $d \ge 1$ and $t_j = y_jz_j$ for all $j \in J_1 \sqcup J_2$, the $R$-module $\o R/I(J_1,J_2,d,\un t)$ is Cohen--Macaulay of grade $f$.
\end{cor}

\begin{proof}
  Relabel indices so that $J_1 \sqcup J_2 = \{0,\dots,k-1\}$ for some $1 \le k \le f$. 
  We define the $\F$-algebras $R^{(1)}$ and $\o R^{(1)}$ (resp.\ $R^{(2)}$ and $\o R^{(2)}$) exactly as we defined $R$ and $\o R$ but using only indices $0 \le j \le k-1$ (resp.\ $k \le j \le f-1$).
  Then $R \cong R^{(1)} \otimes_\F R^{(2)}$ and $\o R/I(J_1,J_2,d,\un t)$ is the tensor product of $M^{(1)} \defeq  \o R^{(1)}/I(J_1,J_2,d)$ and $M^{(2)} \defeq  \o R^{(2)}/(t_j : k \le j \le f-1)$ over $\F$.
  We know that $M^{(1)}$ is a Cohen--Macaulay $R^{(1)}$-module of grade $k$ (by Proposition~\ref{prop:basic-CM-module} if $d \le k$, and by Lemma~\ref{lem:basic-CM-module}, taking $t_j = y_jz_j$ for all $j$, otherwise).
  By Lemma~\ref{lem:basic-CM-module} the $R^{(2)}$-module $M^{(2)}$ is Cohen--Macaulay of grade $f-k$.
  By the K\"unneth formula, we obtain that 
  \begin{equation*}
    \EE^n_R(\o R/I(J_1,J_2,d,\un t)) \cong \bigoplus_{i+j = n} \EE^i_{R^{(1)}}(M^{(1)}) \otimes_\F \EE^j_{R^{(2)}}(M^{(2)}),
  \end{equation*}
  hence $\o R/I(J_1,J_2,d,\un t)$ is a Cohen--Macaulay $R$-module of grade $f$.
\end{proof}

If $N'$ is a finitely generated $\o R$-module, {we let $m_\q(N')\defeq {\rm length}_{\o R_\q}(N'_{\q})$ and define $m(N') \defeq  \sum_{\q} m_\q(N')$, which is the total multiplicity of the cycle $\mathcal{Z}(N')$ in (\ref{eq:cycleintro}) (here $\q$ runs through the minimal prime ideals of $\o R$)}.

\begin{lem}\label{lem:total-mult}
  Suppose that $t_j = y_jz_j$ for all $j \in J \defeq  J_1 \sqcup J_2$.
  Then \[m(\o R/I(J_1,J_2,d,\un t)) = 2^{|\{j \in J^c : t_j = y_jz_j\}|} \left(\sum_{i < d} \binom{|J|}i\right).\]
\end{lem}

\begin{proof}
  If $d \le 0$, the formula is trivially true, so we suppose $d \ge 1$.
  Without loss of generality we assume that $J = \{0,\dots,k-1\}$ for some $1 \le k \le f$ and that $J_1 = \emptyset$.
  Consider the minimal prime $\q = (v_0,\dots,v_{f-1})$ of $\o R$ given by $v_j \in \{y_j,z_j\}$.
  Write \[M \defeq  \o R/I(J_1,J_2,d,\un t) = \F[y_j,z_j : 0 \le j \le f-1]/(y_jz_j,z_{i_1}\cdots z_{i_d},t_{j'}),\] where $0 \le j < k \le j' < f$, and $0 \le i_1 < \cdots < i_d \le k-1$.
  If $v_j = y_j$, then in $M_\q$ the variable $z_j$ is inverted and $y_j$ becomes zero, and vice versa when $v_j = z_j$.
  It follows that $m_\q(M) = 1$ if $|\{0 \le j \le k-1 : v_j = y_j\}| < d$ and $v_{j'}$ divides $t_{j'}$ for all $k \le j' < f$, whereas $m_\q(M) = 0$ otherwise.
  The lemma follows by summing over all $\q$.
\end{proof}

\subsection{On the structure of subrepresentations of \texorpdfstring{$\pi$}{pi}}
\label{sec:structure-subrep}

The main result of this section is the description of the $K_1$-invariants of subrepresentations of $\pi$ (Theorem \ref{thm:conj2}). We need several technical results on $\GL_2(\cO_K)$-representations induced from certain multiplicity-free $I$-representations.
 
\subsubsection{Some induced representations of \texorpdfstring{$\GL_2(\cO_K)$}{GL\_2(O\_K)}}
\label{sec:some-repr-gl_2}

We study $\GL_2(\cO_K)$-representations induced from certain multiplicity-free $I$-representations.

Given a character $\chi:I\ra \F^{\times}$ and two subsets $J_1,J_2\subset \{0,\dots,f-1\}$ such that $J_1\cap J_2=\emptyset$, set 
\begin{equation}\label{eq:chi-J}\chi^{J_1,J_2}\defeq \chi\prod_{j\in J_1}\alpha_j^{-1}\prod_{j\in J_2}\alpha_j.\end{equation}

\begin{lem}\label{lem:Wchi}
There exists a unique $I$-representation of dimension $2^{|J_1|+|J_2|}$ with socle $\chi$ and cosocle $\chi^{J_1,J_2}$ such that the $d$-th socle layer is given by \[\bigoplus_{J_1'\subset J_1, J_2'\subset J_2, |J_1'|+|J_2'|=d}\chi\prod_{j\in J_1'}\alpha_j^{-1}\prod_{j\in J_2'}\alpha_j.\]   We denote it by $W(\chi,\chi^{J_1,J_2})$. Moreover, 
\begin{enumerate}
\item $W(\chi,\chi^{J_1,J_2})$ is multiplicity free;  
\item $W(\chi,\chi^{J_1,J_2})$ is fixed by $K_1$ if and only if $J_2=\emptyset$.
\end{enumerate}
\end{lem}
 
\begin{proof}
We first prove uniqueness. By \cite[(42),\;(43)]{BHHMS1} and the Poincar\'e--Birkhoff--Witt theorem, any Jordan--H\"older factor $\chi'$ of $(\Inj_{I/Z_1}\chi)[\m^{n+1}]$ has the form 
$\chi\alpha_{i_1}^{t_{1}}\cdots\alpha_{i_m}^{t_m}$, 
where $m\leq n$,  $i_k\in\{0,\dots,f-1\}$ and $t_k\in\{\pm1\}$, which is equal to  
\[\chi\prod_j\alpha_j^{b_j}\prod_{j}\alpha_j^{-b_j'}\]
with   $b_j\defeq |\{k: i_k=j, t_k=1 \}|$ and $b_j'\defeq |\{k:i_k=j, t_k=-1\}|$. In particular, $\chi^{J_1,J_2}$ occurs in $(\Inj_{I/Z_1}\chi)[\m^{|J_1|+|J_2|+1}]$. We claim that it occurs with multiplicity one. Indeed, if $\chi'=\chi^{J_1,J_2}$,  Lemma \ref{lem:unique} shows that \[b_j-b_j'=\left\{\begin{array}{rl}-1 &\textrm{if}\ j\in J_1,\\
1&\textrm{if}\ j\in J_2,\\
0&\textrm{otherwise.}\end{array}\right.\] 
Using the condition $\sum_{j=0}^{f-1}(b_j+b_j')\leq |J_1|+|J_2|$, we deduce that 
\[(b_j,b_j')=\left\{\begin{array}{rl}(0,1) &\textrm{if}\ j\in J_1,\\
(1,0)&\textrm{if}\ j\in J_2,\\
(0,0)&\textrm{otherwise.}\end{array}\right.\] 
This \ implies \ the \ claim. \ As \ a \ consequence, \ if \ $W(\chi,\chi^{J_1,J_2})$ \ exists, \ it \ embeds \ into $(\Inj_{I/Z_1}\chi)[\m^{|J_1|+|J_2|+1}]$ and is hence the unique subrepresentation of $(\Inj_{I/Z_1}\chi)[\m^{|J_1|+|J_2|+1}]$ with cosocle $\chi^{J_1,J_2}$. 

For the existence, we may assume $\chi=\mathbf{1}$, and let $E_j^{\pm}$ denote the $I$-representation $E_j^{\pm}(1)$ constructed in the proof of Lemma \ref{lem:tau-embed} (with $s=1$). We take
\[W(\mathbf{1}, \mathbf{1}^{J_1,J_2})\defeq \big(\bigotimes_{j\in J_1}E_j^-\big)  \otimes_{\F} \big(\bigotimes_{j\in J_2}E_j^+\big).\]   
It is multiplicity free by Lemma~\ref{lem:unique}. The assertion on the $d$-th socle layer of $W(\mathbf{1}, \mathbf{1}^{J_1,J_2})$  can be proved as in  the proof of Lemma \ref{lem:tau-embed}, which shows that the socle filtration of $W(\mathbf{1}, \mathbf{1}^{J_1,J_2})$ corresponds to a suitable tensor product filtration on $W(\mathbf{1}, \mathbf{1}^{J_1,J_2})^{\vee}$ under duality.

Finally, assertion (i) was established above and (ii) follows from the fact that $E_j^-$ is fixed by $K_1$ and $E_j^+$ is not.
\end{proof}

\begin{lem}\label{lem:unique}
Suppose $p > 3$.
Let $\underline{a}=(a_0,\dots,a_{f-1})$, $\underline{b}=(b_0,\dots,b_{f-1})\in \Z^f$. Assume that $a_j\in\{-1,0,1\}$ for all $j$, $\sum_{j=0}^{f-1}|a_j|\geq \sum_{j=0}^{f-1}|b_j|$ and
  \[\sum_{j=0}^{f-1}a_jp^j\equiv \sum_{j=0}^{f-1}b_jp^j\ (\mathrm{mod}\  p^f-1).\]
Then  $\un{a}=\un{b}$.
\end{lem}
\begin{proof}
Let $|\un{a}|\defeq \sum_{j=0}^{f-1}|a_j|$. We induct on the pair $(|\un{a}|,|\un{b}|)$ with lexicographic order. 
We fix a pair $(\un{a},\un{b})$ and suppose  that the result holds for all $(\un{a}',\un{b}')<(\un{a},\un{b})$. 

We claim that $|b_j|\leq p-2$ for all $j$. If $b_i\geq p-1$ for some $i$, we define $\un{b}'\in\Z^f$ by 
\[ b_j'\defeq 
\begin{cases}
  b_j-p & \text{if $j=i$},\\
  b_{j}+1& \text{if $j= i+1$},\\
  b_j&\text{otherwise}.
\end{cases}
\]
Then 
  $\sum_{j=0}^{f-1}b_jp^j\equiv \sum_{j=0}^{f-1}b_j'p^j\ (\mathrm{mod}\ p^{f}-1)$, and $|\un{b}'|<|\un{b}|$ (as $p > 3$).
By induction we have $\un{a}=\un{b}'$, which implies $|\un{a}|=|\un{b}'|<|\un{b}|$, contradiction. 
Thus $b_j\leq p-2$ for all $j$. In a similar way we get $b_j\geq -(p-2)$ for all $j$.

 The assumption implies 
\[\sum_j(b_j-a_j)p^j\equiv 0\ (\mathrm{mod}\  p^f-1)\]
with $|b_j-a_j|\leq p-1$ for all $j$, as $|b_j|\leq p-2$. Then this can happen only when $b_j=a_j$ for all $j$, or $|b_j-a_j|=p-1$ for all $j$. The second possibility cannot happen, because it forces $|b_j|=p-2$ for all $j$, which contradicts $|\un{a}|\ge|\un{b}|$, as $p > 3$.
\end{proof}

Note that if $\chi$ is $n$-generic (see \S~\ref{sec:notation}) with $n\geq 2$, then every character occurring in  $W(\chi,\chi^{J_1,J_2})$ is $(n-2)$-generic by Lemma \ref{lem:Wchi}.

\begin{lem}\label{lem:multfree}
Assume \ \ that \ \ $\chi$ \ \ is \ \ $2$-generic.
\ \ Then \ \ the \ \ $\GL_2(\cO_K)$-representation $\Ind_{I}^{\GL_2(\cO_K)}W(\chi,\chi^{J_1,J_2})$ is multiplicity free.
\end{lem}

\begin{proof}
It follows from our genericity assumption and \cite[Lemma~2.2]{BP}.
\end{proof}

We recall from \S~\ref{sec:preliminaries-fin-len} that the Jordan--H\"older factors $\sigma$ of a principal series representation $\Ind_I^{\GL_2(\cO_K)}\chi'$ {for a $1$-generic character $\chi' : I \to \F\s$} are parametrized by the subsets of $\{0,\dots,f-1\}$, sending $\sigma$ to $\mathcal{S}(\sigma)$, such that the socle of $\Ind_I^{\GL_2(\cO_K)}\chi'$ corresponds to the empty set {(see Remark \ref{rmk:JH:conventions})}. 

{We also recall from \cite[Def.~2.9]{HuWang2} some notation and a lemma. 
Assume first $f>1$.
Given $j\in\{0,\dots,f-1\}$ and $*\in \{+,-\}$ we define the elements $\mu_j^*\in \mathcal{I}$ 
as follows: $(\mu_j^*)_{j-1}(x_{j-1})=p-2-x_{j-1}$, $(\mu_j^*)_{j}(x_{j})=x_{j}\ast1$ and $(\mu_j^*)_{i}(x_i)=x_i$ for $i\notin\{j-1,j\}$.
If $f = 1$ we define $\mu_0^* \in \Z[x_0]$ by $\mu_0^*(x_{0})=p-2-(\ast 1)-x_0$.
For any $f\geq 1$, if $\sigma$ is a $0$-generic Serre weight corresponding to a tuple $(s_0,\dots,s_{f-1})\in\{0,\dots,p-1\}^f$ we write $\mu^*_j(\sigma)$ for the Serre weight $\mu_j^*\big((s_0,\dots,s_{f-1})\big)\otimes\det^{e(\mu_j^*)(s_0,\dots,s_{f-1})}$, where $e(\mu_j^*)\in \Z\oplus \bigoplus_{i=0}^{f-1}\Z x_i$ is defined in \cite[\S~3]{BP}.
 (Note that $\mu_j^-(\sigma)$ is undefined if $f \ge 2$ and $s_j = 0$ and $\mu_j^+(\sigma)$ is undefined if $f = 1$ and $s_j = p-2$.)
} 

\begin{lem}\label{lem:ext1}
 Let $\sigma$ and $\sigma'$ be two $0$-generic Serre weights. If $f=1$, suppose that  $\sigma$, $\sigma'$ are not both isomorphic to $\mathrm{Sym}^{p-2}\F^2 \otimes \eta$ for some $\eta$. 
 Then 
 \begin{equation*}
   \Ext^1_{\GL_2(\cO_K)/Z_1}(\sigma',\sigma)\neq0 \iff \Ext^1_{\Gamma}(\sigma',\sigma)\neq0 \iff \sigma' \in \big\{\mu_j^*(\sigma) : 0\leq j\leq f-1, *\in \{+,-\} \big\}.
  \end{equation*}
\end{lem}
 
Lemma \ref{lem:ext1} follows from  \cite[Lemma 2.10]{HuWang2}  and \cite[Prop.~2.21]{yongquan-algebra}, except when $f=1$ the proof is incomplete in \emph{loc.~cit.}

\begin{proof}
If $\sigma'=\mu_j^*(\sigma)$ for some $0\leq j\leq f-1$ and $*\in\{+,-\}$, it follows from \cite[Cor.~5.6]{BP} that $\Ext^1_{\Gamma}(\sigma',\sigma)=\Ext^1_{K/Z_1}(\sigma',\sigma)\neq0$.  
Conversely,  suppose $\Ext^1_{K/Z_1}(\sigma',\sigma)\neq0$ and  we need to prove that $\sigma'=\mu_j^*(\sigma)$ for some $0\leq j\leq f-1$ and $*\in\{+,-\}$. Using \cite[Cor.~5.6]{BP} and \cite[Lem.~2.10(i)]{HuWang2}, it suffices to exclude cases (a) and (c) of \cite[Cor.~5.6(ii)]{BP}. The argument below is taken from the proof  of
\cite[Lemma~2.10(i)]{HuWang2}.

First assume that we are in case (c). Thus, as $\sigma,\sigma'$ are $0$-generic, we may write \[\sigma'=(s_0,\dots,s_{f-1})\otimes\eta,\ \ \sigma=(s_0,\dots,s_{j}-2,\dots,s_{f-1})\otimes\eta{\det}^{p^j},\] with $2\leq s_j\leq p-2$ and $0\leq s_i\leq p-2$ for $i\neq j$.  Let $0\ra \sigma\ra V\ra \sigma'\ra0$ be a nonsplit $K/Z_1$-extension. Let $w\in V$ be an $H$-eigenvector of character $\chi_{\sigma'}$ such that its image in $\sigma'$  spans $\sigma'^{I_1}$. We will prove that $w$ is fixed by $I_1/Z_1$, thus by Frobenius reciprocity we obtain a $\GL_2(\cO_K)$-equivariant morphism $\Ind_I^{\GL_2(\cO_K)}\chi_{\sigma'}\ra V$ which must be surjective (as it surjects onto $\cosoc_K V$). But this is impossible by the structure of $\Ind_I^{\GL_2(\cO_K)}\chi_{\sigma'}$ by \cite[Lemma~2.2]{BP}. 

For $0\leq i\leq f-1$, consider the operators
\[X_i\defeq\sum_{\lambda\in\F_q}\kappa_0(\lambda)^{-p^i}\smatr{1}0{[\lambda]}1, \ \ Y_i\defeq\sum_{\lambda\in\F_q}\kappa_0(\lambda)^{-p^i}\smatr{1}{[\lambda]}01,\]
which are viewed as elements of $\F[\![K/Z_1]\!]$.  
By \cite[Lemma~3.37]{BHHMS2}, we have $\F[\![\smatr1{\cO_K}01]\!]= \F[\![Y_0,\dots,Y_{f-1}]\!]$ and similarly $\F[\![\smatr{1}0{p\cO_K}1]\!]= \F[\![X_0^{p},\dots,X_{f-1}^{p}]\!]$. 
Thus, by (the proof of) \cite[Prop.\ 5.3.3]{BHHMS1} the elements $\{X_i^p, Y_i : 0\leq i\leq f-1\}$ topologically generate the maximal ideal of $\F[\![I_1/Z_1]\!]$ and we are left to prove that $X_i^pw=Y_iw=0$ for all $0\leq i\leq f-1$. 

It is direct to check that $X_iw$ (resp.~$Y_iw$) has $H$-eigencharacter $\chi_{\sigma'}\alpha_i^{-1}$ (resp.~$\chi_{\sigma'}\alpha_i$). On the other hand, as $\chi_{\sigma}=\chi_{\sigma'}\alpha_j^{-1}$, we have
\begin{equation*}
 \JH(V|_H) = \JH(\sigma'|_H) = \bigg\{\chi_{\sigma'}\prod_{i=0}^{f-1}\alpha_i^{-k_i} : 0\leq k_i\leq s_i\bigg\}.
\end{equation*}
Moreover, $Y_iw \in \sigma$, by our choice of $w$.
However,  it is direct to check that $\chi_{\sigma'}\alpha_i^{-(p-1)}$ does not occur in $\JH(V|_H)$ and $\chi_{\sigma'}\alpha_i$ does not occur in $\JH(\sigma|_H)$. 
Therefore $X_i^{p-1}w=Y_iw=0$ for all $0\leq i\leq f-1$, hence also $X_i^{p}w=0$, as desired. 

Case (a) can be treated by passing to the dual  (the dual extension of $\sigma^\vee$ by $\sigma'^\vee$ is as in case (c)). 
\end{proof}

\begin{lem}\label{lem:sigma'-chi'}
 Assume that $\chi$ is $3$-generic. 
Let $\sigma\in\JH(\Ind_I^{\GL_2(\cO_K)}\chi)$ and let $\sigma'=\mu_j^*(\sigma)$ for some $0\leq j\leq f-1$ and $*\in\{+,-\}$. Let $J_1,J_2\subset\{0,\dots,f-1\}$ such that $J_1\cap J_2=\emptyset$. Assume $\sigma'\in \JH\big(\Ind_I^{\GL_2(\cO_K)}W(\chi,\chi^{J_1,J_2})\big)$. Then $\sigma'\in \JH\big(\Ind_I^{\GL_2(\cO_K)}(\chi\oplus \chi\alpha_j^{-1}\oplus\chi\alpha_j)\big)$ and exactly one of the following cases happens:
\begin{enumerate}
\item $\sigma'\in \JH(\Ind_I^{\GL_2(\cO_K)}\chi)$, in which case either 
$\mathcal{S}(\sigma)\sqcup\{j\}=\mathcal{S}(\sigma')$ or 
$\mathcal{S}(\sigma')\sqcup\{j\}=\mathcal{S}(\sigma)$;
\item $\sigma'\in \JH(\Ind_I^{\GL_2(\cO_K)}\chi\alpha_j^{-1})$, in which case $j \in J_1$ 
and $\mathcal{S}(\sigma)\sqcup\{j\}=\mathcal{S}(\sigma')$; 
\item $\sigma'\in \JH(\Ind_I^{\GL_2(\cO_K)}\chi\alpha_j)$, in which case $j \in J_2$ 
and $\mathcal{S}(\sigma')\sqcup\{j\}=\mathcal{S}(\sigma)$. 
\end{enumerate}
\end{lem}
\begin{proof}
First, it is direct to check that $\sigma'$ occurs in $\Ind_I^{\GL_2(\cO_K)}(\chi\oplus \chi\alpha_j^{-1}\oplus\chi\alpha_j)$. 
The claim on the relation between $\mathcal{S}(\sigma)$ and $\mathcal{S}(\sigma')$ follows directly from the definition of $\mu_j^*(\sigma)$ and (\ref{eq:S-xi1}) in case (i), and 
from \cite[Lemmas 3.8, 3.7]{HuWang2} in cases (ii) and (iii) respectively.
\end{proof}

\begin{prop}\label{prop:HW}
{Assume that $\chi$ is $5$-generic.} 
  \begin{enumerate}
  \item 
  \label{it:prop:HW:1}
    The cosocle of $\Ind_I^{\GL_2(\cO_K)}W(\chi,\chi^{J_1,J_2})$ is $\bigoplus_{J_1'\subset J_1}\sigma^{J_1',J_2}$, where $\sigma^{J_1',J_2} $ denotes the cosocle of $\Ind_I^{\GL_2(\cO_K)}\chi^{J_1',J_2}$.  
   \item \label{it:prop:HW:2}
    Let $\sigma\in \JH(\Ind_I^{\GL_2(\cO_K)}\chi)$ be parametrized by $\mathcal{S}(\sigma)$ and $\tau\in \JH(\Ind_I^{\GL_2(\cO_K)}\chi^{J_1,J_2})$ be parametrized by $\mathcal{S}(\tau)$. Let $Q_{\sigma}$ be the unique quotient of $\Ind_I^{\GL_2(\cO_K)}W(\chi,\chi^{J_1,J_2})$ with socle $\sigma$ (by Lemma~\ref{lem:multfree}). Then $\tau\in \JH(Q_{\sigma})$ if and only if 
    \begin{equation}\label{eq:HW}
      \mathcal{S}(\sigma)\cap J_1=\emptyset \text{\ \ and\ \ } \mathcal{S}(\sigma)\sqcup J_1\subset \mathcal{S}(\tau)\cup J_2.\end{equation}
  \end{enumerate}
\end{prop}
 \begin{rem}\label{rem:HW}
 In Proposition \ref{prop:HW}\ref{it:prop:HW:2}, let $V_{\tau}\subset \Ind_{I}^{\GL_2(\cO_K)}W(\chi,\chi^{J_1,J_2})$ be the unique subrepresentation with cosocle $\tau$ (again by Lemma~\ref{lem:multfree}). 
 Then $\tau\in \JH(Q_{\sigma})$ if and only if $\sigma\in \JH(V_{\tau})$.
 \end{rem}
 
\begin{proof}
{Note that the genericity assumption implies that any $\chi'\in \JH(W(\chi,\chi^{J_1,J_2}))$ is $3$-generic.}

(i) By Lemma~\ref{lem:Wchi}(ii), any $I$-equivariant morphism $W(\chi,\chi^{J_1,J_2})\ra \sigma'|_I$ (where $\sigma'$ is a Serre weight) factors through the quotient of $K_1$-coinvariants \[W(\chi,\chi^{J_1,J_2})\twoheadrightarrow W(\chi^{\emptyset,J_2},\chi^{J_1,J_2}). \]
By Frobenius reciprocity, this implies that the cosocle of $\Ind_I^{\GL_2(\cO_K)}W(\chi,\chi^{J_1,J_2})$ is equal to that of $\Ind_I^{\GL_2(\cO_K)}W(\chi^{\emptyset,J_2},\chi^{J_1,J_2})$, 
 so any of its irreducible constituents is of the form $\sigma^{J_1',J_2}$ for some $J_1'\subset J_1$.
Conversely, $\Ind_I^{\GL_2(\cO_K)}W(\chi,\chi^{J_1,J_2})$ surjects onto $\Ind_I^{\GL_2(\cO_K)}W(\chi^{J_1',J_2},\chi^{J_1,J_2})$, which surjects onto $\sigma^{J_1',J_2}$. 
(Write $\sigma^{J_1',J_2}=(s_0,\dots,s_{f-1})\otimes\eta$ with $0\leq s_j\leq p-1$.
As $s_j \ge 1$ for all $j$, $W(\chi^{J_1',J_2},\chi^{J_1,J_2})$ embeds in $\sigma^{J_1',J_2}|_I$ and is identified with the subspace spanned by $x^{s-i}y^i\in \sigma^{J_1',J_2}$ for $i\in\{\sum_{j\in J_1''}p^j: J_1''\subset J_1'\}$, where $s\defeq \sum_{j=0}^{f-1}p^js_j$; see the discussion at the beginning of  \cite[\S~17]{BP}. 
As a consequence, $\sigma^{J_1',J_2}$ occurs in the cosocle of $\Ind_I^{\GL_2(\cO_K)}W(\chi^{J_1',J_2},\chi^{J_1,J_2})$ by Frobenius reciprocity.)

(ii) 
Assume $\tau$  satisfies condition \eqref{eq:HW}. Let
\[\chi''\defeq \chi^{J_1,J_2},  \ \ \chi'\defeq \chi^{J_1,\emptyset}=\chi\prod_{j\in J_1}\alpha_j^{-1}.\]
Then \ $W(\chi,\chi')\hookrightarrow W(\chi,\chi'')$ \ and \ $W(\chi,\chi'')\twoheadrightarrow W(\chi',\chi'')$. \ Let \ $Q_1$ \ be \ the \ image \ of \ $\Ind_I^{\GL_2(\cO_K)}W(\chi,\chi')$ \ in \ $Q_\sigma$ \ and \ $Q_2$ \ be \ the \ pushout \ of \ $Q_\sigma$ \ and \ $\Ind_I^{\GL_2(\cO_K)}W(\chi',\chi'')$ \ along $\Ind_I^{\GL_2(\cO_K)}W(\chi,\chi'')$.  

 If $\tau'\in\JH(\Ind_I^{\GL_2(\cO_K)}\chi')$ denotes the Jordan--H\"older factor parametrized by $\mathcal{S}(\sigma)\sqcup J_1$, then $\tau'\in\JH(Q_1)$  by repeated use of \cite[Lemma 3.8]{HuWang2}, and consequently $\tau'\in \JH(Q_{\sigma})$. Since $\Ind_I^{\GL_2(\cO_K)}W(\chi,\chi'')$ is multiplicity free by Lemma \ref{lem:multfree} and $\tau'$ occurs in $\Ind_I^{\GL_2(\cO_K)}\chi'$, we have  $\tau'\in \JH(Q_2)$ by construction of $Q_2$. 
Thus $\tau'' \in \JH(Q_2)$, where $\tau'' \in \JH(\Ind_I^{\GL_2(\cO_K)}\chi')$ denotes the Jordan--H\"older factor parametrized by $(\mathcal{S}(\sigma) \cup J_2) \sqcup J_1$.
By repeated use of \cite[Lemma 3.7]{HuWang2} we deduce that $\tau''' \in \JH(Q_2)$, where $\tau''' \in \JH(\Ind_I^{\GL_2(\cO_K)}\chi'')$ denotes the Jordan--H\"older factor parametrized by $(\mathcal{S}(\sigma)\setminus J_2)\sqcup J_1$.
As $(\mathcal{S}(\sigma)\setminus J_2)\sqcup J_1 \subset \mathcal{S}(\tau)$ by assumption \eqref{eq:HW}, we conclude that $\tau \in \JH(Q_2) \subset \JH(Q_\sigma)$. 

 Conversely, assume $\tau\in \JH(Q_{\sigma})$. Let $Q_{\sigma}^{\tau}\subset Q_{\sigma}$ be the unique subrepresentation with cosocle $\tau$. We induct on the length $\ell\defeq \ell(Q_{\sigma}^{\tau})$. 
If $\ell = 1$, then $\tau = \sigma$ and $J_1 = J_2 = \emptyset$, so~(\ref{eq:HW}) follows.
If $\ell\geq 2$, let $\mathcal{E}\subset Q_{\sigma}^{\tau}$ be a subrepresentation of length $2$, namely $\mathcal{E}$ has the form
\[0\ra \sigma\ra \mathcal{E}\ra\sigma'\ra0\]
for some $\sigma'$ satisfying $\Ext^1_{K/Z_1}(\sigma',\sigma)\neq0$.  
By Lemma~\ref{lem:ext1}, $\sigma' \cong \mu_j^*(\sigma)$ for some $0\leq j\leq f-1$ and $*\in\{+,-\}$. 
Define again $\chi''\defeq \chi^{J_1,J_2}$.
Let $\chi'=\chi^{J_1',J_2'}$ be the unique character occurring in $W(\chi,\chi'')$ such that $\sigma'\in \JH(\Ind_I^{\GL_2(\cO_K)}\chi')$. 
Let $Q_{\sigma'}$ denote the unique quotient of $Q_\sigma$ with socle $\sigma'$. Since $\Ind_I^{\GL_2(\cO_K)}W(\chi,\chi'')$ is multiplicity free, it is easy to see that the quotient map \[\Ind_I^{\GL_2(\cO_K)}W(\chi,\chi'')\twoheadrightarrow Q_{\sigma'}\] factors through $\Ind_I^{\GL_2(\cO_K)}W(\chi',\chi'')$, namely $Q_{\sigma'}$ is a quotient of $\Ind_I^{\GL_2(\cO_K)}W(\chi',\chi'')$.

By Lemma \ref{lem:sigma'-chi'}, we have either $J_1'=J_2'=\emptyset$ or $J_1'\sqcup J_2'=\{j\}$. In the first case, since $\mathcal{E}$ is a subquotient of $\Ind_I^{\GL_2(\cO_K)}\chi$, we must have $\mathcal{S}(\sigma)\sqcup\{j\}=\mathcal{S}(\sigma')$ by Lemma \ref{lem:sigma'-chi'} and \cite[Thm.~2.4]{BP}. 
On the other hand, the inductive hypothesis (applied to $Q_{\sigma'}$) implies  $\mathcal{S}(\sigma') \cap J_1 = \emptyset$ and $\mathcal{S}(\sigma')\sqcup J_1\subset \mathcal{S}(\tau)\cup J_2$, from which we conclude. 
In the second case, we have the following two subcases:
\begin{itemize}
\item  $J_1'=\{j\}$ and $J_2'=\emptyset$, in which case 
$\mathcal{S}(\sigma)\sqcup \{j\}=\mathcal{S}(\sigma')$. By the inductive hypothesis, we also have
 \[\mathcal{S}(\sigma')\sqcup (J_1\setminus\{j\})\subset \mathcal{S}(\tau)\cup J_2\]
and hence \eqref{eq:HW} holds.
\item $J_1' = \emptyset$ and $J_2'=\{j\}$, in which case 
$\mathcal{S}(\sigma)=\mathcal{S}(\sigma')\sqcup\{j\}$.  
By the inductive hypothesis, we also have
\[\mathcal{S}(\sigma')\sqcup J_1\subset \mathcal{S}(\tau)\cup(J_2\setminus\{j\})\]
and hence \eqref{eq:HW} holds.\qedhere
\end{itemize}
\end{proof}

\begin{prop}\label{prop:cond-K1inv}
Let $\sigma,\tau$ be as in Proposition \ref{prop:HW}(ii). 
Assume $\tau\in \JH(Q_{\sigma})$. 
Then the following are equivalent: 
\begin{enumerate}
\item $\tau\in \JH(Q_{\sigma}^{K_1})$;
\item $\tau\in \JH(\Inj_{\Gamma}\sigma)$;
\item $J_2\subset \mathcal{S}(\sigma)  \text{\ \ and\ \ } \mathcal{S}(\tau)\cap J_2=\emptyset$.
\end{enumerate}
\end{prop}

\begin{proof}
Clearly (i) implies (ii). 
We now show that (ii) implies (iii).
Let $\tau'$ be the constituent of $\Ind_I^{\GL_2(\cO_K)}\chi$ that is parametrized by $\mathcal{S}(\tau') = \mathcal{S}(\tau)$.
Let $\sigma_\emptyset$ denote the socle of $\Ind_I^{\GL_2(\cO_K)}\chi$.
Then by the definitions, $\sigma = \lambda(\sigma_\emptyset)$ and $\tau' = \mu'(\sigma_\emptyset)$ for unique $\lambda,\mu' \in \mathcal{P}$ (using again the notation of \cite[(2.2)]{HuWang2}).
Moreover, as $\mathcal{S}(\tau)=\mathcal{S}(\tau')$ we can write $\tau = \mu(\sigma_\emptyset)$, where 
\begin{equation}\label{eq:mu-mu'}
  \mu_j(x_j) = \begin{cases}
    \mu_j'(x_j+2) & \text{if $j \in J_1$},\\
    \mu_j'(x_j-2) & \text{if $j \in J_2$},\\
    \mu_j'(x_j) & \text{otherwise}.
  \end{cases}
\end{equation}
As $\tau \in \JH(\Inj_{\Gamma}\sigma)$, it corresponds to some $\nu \in \mathcal{I}$ such that $\tau = \nu(\sigma)$.
It is direct to check that $\mu$ defined in \eqref{eq:mu-mu'} satisfies the condition in \cite[Lemma 2.1]{HuWang2}, so by \cite[Lemmas 2.1, 2.7]{HuWang2}, we deduce $\mu = \nu \circ \lambda$.

Suppose by contradiction that $J_2 \setminus \mathcal{S}(\sigma)\neq \emptyset$ and we choose $j\in J_2 \setminus \mathcal{S}(\sigma)$.
Then $\lambda_j(x_j) \in \{x_j,p-2-x_j\}$, and by (\ref{eq:mu-mu'}) and the definition of $\mathcal P$ we have
\begin{equation*}
  \mu_j(x_j) \in \{x_j-2,x_j-3,p+1-x_j,p-x_j\},
\end{equation*}
contradicting that $\mu = \nu \circ \lambda$ with $\nu \in \mathcal{I}$ (see also the table in the proof of \cite[Lemma 2.6]{HuWang2}).
Similarly, suppose that $J_2 \cap \mathcal{S}(\tau)\neq\emptyset$ and we choose $j \in J_2 \cap \mathcal{S}(\tau)$.
Then $\mu'_j(x_j) \in \{x_j-1,p-1-x_j\}$, so by (\ref{eq:mu-mu'}) and the definition of $\mathcal P$ we have
\begin{equation*}
  \mu_j(x_j) \in \{x_j-3,p+1-x_j\},\quad
  \lambda_j(x_j) \in \{x_j,x_j-1,p-1-x_j,p-2-x_j\}.
\end{equation*}
This yields a contradiction as before.

We finally show that (iii) implies (i).
Let $Q_{\sigma}^{\tau}\subset Q_{\sigma}$ be the unique subrepresentation with cosocle $\tau$ (which exists by assumption).
It will suffice to show that $Q_\sigma^\tau$ is $K_1$-invariant, and we will do that by verifying the assumption of \cite[Cor.~5.7]{BP}.
Note first that $Q_\sigma^\tau$ is multiplicity free.
By the genericity condition (which in particular implies $p>3$) it will suffice to rule out conditions (a) and (c) of \cite[Cor.~5.6]{BP} for any pair of distinct constituents of $Q_\sigma^\tau$ and for any $0 \le j \le f-1$.
Observe that if $\tau'$ is a (sufficiently generic) constituent of $\Ind_I^{\GL_2(\cO_K)}\chi'$, then the constituent $\sigma'$ described in condition (a) or (c) of \cite[Cor.~5.6]{BP} for some $0 \le j \le f-1$ occurs in $\Ind_I^{\GL_2(\cO_K)}\chi'\alpha_j^{\pm 1}$ for some choice of sign, and moreover $\mathcal{S}(\sigma') = \mathcal{S}(\tau')$ for the parametrizing sets.
It therefore suffices to show that any two distinct constituents of $Q_\sigma^\tau$ have distinct parametrizing sets.

Suppose that $\tau' \in \JH(Q_\sigma^\tau)$ occurs in $\JH(\Ind_I^{\GL_2(\cO_K)}\chi^{J_1',J_2'})$ for some $J'_1\subset J_1$ and $J_2'\subset J_2$.
By the previous paragraph it is enough to show that $\mathcal S(\tau') \cap J_1 = J_1'$ and $\mathcal S(\tau') \cap J_2 = J_2''$, where we write $J_i'' \defeq  J_i \setminus J_i'$.
From $\tau' \in \JH(Q_\sigma)$ and $\tau \in \JH(Q_{\tau'})$ we obtain from Proposition \ref{prop:HW}(ii) that \[\mathcal{S}(\sigma) \sqcup J_1' \subset \mathcal{S}(\tau') \cup J_2', \ \  
\mathcal{S}(\tau') \cap J_1'' = \emptyset,\ \  \mathcal{S}(\tau') \sqcup J_1'' \subset \mathcal{S}(\tau) \cup J_2''.\]
The first and second statements together show that $\mathcal S(\tau') \cap J_1 = J_1'$.
The first statement plus $J_2 \subset \mathcal{S}(\sigma)$ and the third statement plus $\mathcal{S}(\tau)\cap J_2=\emptyset$ give $\mathcal S(\tau') \cap J_2 = J_2''$, as desired.
\end{proof}

\subsubsection{Some \texorpdfstring{$\GL_2(\cO_K)$}{GL\_2(O\_K)}-subrepresentations of \texorpdfstring{$\pi$}{pi}}

{We apply \S~\ref{sec:some-repr-gl_2} to construct some $\GL_2(\cO_K)$-subrepresentations of $\pi$ that will be important in the proof of Theorem~\ref{thm:conj2}.}

For $\mu\in\P$ define 
\begin{align}
\label{eq:Y-mu}Y(\mu)&\defeq \{j:\mu_j(x_j)\in\{x_j,x_j+1,p-2-x_j,p-3-x_j\}\}\cup J_{\brho}^c, \\ 
\label{eq:Z-mu}Z(\mu)&\defeq \{j : \mu_j(x_j)\in \{x_j+1,x_j+2, p-1-x_j,p-2-x_j\}\}\cup J_{\brho}^c. 
\end{align}
Note that $Y(\mu)$ (resp.\ $Z(\mu)$) is exactly the set of $j$ such that $t_j\neq y_j$ (resp.\ $t_j\neq z_j$) in \eqref{eq:id:al}. Here, we recall that $\mu_j(x_j)\in \{x_j+2, p-3-x_j\}$ implies $j\in J_{\brho}$ by (\ref{eq:P}).

\begin{lem}\label{lem:W-embeds}
 Suppose that $\brho$ is 2-generic. 
Let $\mu\in \P$ and $\chi\defeq \chi_{\mu}$. Let  $J_1\subset Y(\mu)$ and $J_2\subset Z(\mu)$ be subsets satisfying  $J_1\cap J_2=\emptyset$. Then $\JH(W(\chi,\chi^{J_1,J_2}))\cap \JH(\pi^{I_1})=\{\chi\}$ and there exists a unique (up to scalar) $I$-equivariant embedding $W(\chi,\chi^{J_1,J_2})\hookrightarrow \pi|_I$. Moreover, 
\begin{enumerate}
\item the image of the induced morphism 
\[\Ind_I^{\GL_2(\cO_K)}W(\chi,\chi^{J_1,J_2})\ra \pi|_{\GL_2(\cO_K)}\]
has socle $\sigma\in W(\brho)$, where $\sigma$ is the Serre weight determined by $J_{\sigma}=J_{\brho}\cap J_{\mu}$ {\textnormal{(}via \eqref{eq:J-lambda}\textnormal{)}};
\item 
\label{it:W-embeds:2}
$\sigma\in\JH(\Ind_I^{\GL_2(\cO_K)}\chi)$ and it is parametrized 
by $X(\mu)$, where $X(\mu)$ is defined in \eqref{eq:X-mu}.
\end{enumerate}
\end{lem}
\begin{proof}
The first claim follows from Lemma~\ref{lem:compare-I1-invts}(ii) with $m = 1$.
The second follows from the fact that $\Ext^i_{I/Z_1}(\chi',\pi)=0$ for $\chi'\in \JH(W(\chi,\chi^{J_1,J_2})/\chi)$ and $i=0,1$ using the first claim together with assumption~\ref{it:assum-iv} imposed on $\pi$; see the proof of Lemma \ref{lem:isom-modcI}.  
By Lemma \ref{lem:I1-invt-nonss}, the image of $\Ind_I^{\GL_2(\cO_K)}\chi\ra \pi$ has socle $\sigma$ and $\sigma$ is parametrized by $X(\mu)$. To deduce (i) and (ii), it suffices to prove  
 \[\JH(\Ind_I^{\GL_2(\cO_K)}W(\chi,\chi^{J_1,J_2})/\chi)\cap W(\brho)=\emptyset.\]
This follows from the first claim and the fact that $\JH(\Ind_I^{\GL_2(\cO_K)}\chi')\cap W(\brho)=\emptyset$ for any $\chi'\notin \pi^{I_1}$ by \cite[Prop.\ 4.2]{breuil-buzzati}.   
\end{proof}

\begin{lem}\label{lem:W-embeds-3}
Let $\mu\in \P$ and $\sigma\in W(\brho)$ be the Serre weight determined by $J_{\sigma}=J_{\brho}\cap J_{\mu}$.
If $\sigma'\in \JH(\Ind_I^{\GL_2(\cO_K)}\chi_{\mu})\cap W(\brho)$ then $\mathcal{S}(\sigma')\subset \mathcal{S}(\sigma)=X(\mu)$.
As a consequence, $\sigma\in \JH(Q_{\sigma'})$, where $Q_{\sigma'}$ denotes the unique quotient of $\Ind_I^{\GL_2(\cO_K)}\chi_{\mu}$ with socle $\sigma'$.
\end{lem}
\begin{proof}
 Lemma \ref{lem:I1-invt-nonss} implies that the image $V$ of the natural map $\Ind_I^{\GL_2(\cO_K)}\chi_{\mu} \to D_0(\brho)$ has socle $\sigma$ and $\JH(V/\sigma) \cap W(\brho) = \emptyset$.
From \cite[Prop.\ 4.3]{breuil-buzzati} (applied with $\chi=\chi_\mu^s$, noting $\chi \ne \chi^s$) we deduce that $\mathcal{S}(\sigma')\subset \mathcal{S}(\sigma)=X(\mu)$.
\end{proof}

Now we consider a special situation.  Suppose that $\brho$ is 3-generic, so that $\chi_\mu$ is 2-generic for any $\mu \in \P$. Let $\lambda\in \P^\ss$ and denote
  \begin{equation}
  \label{eq:J1:J2}
    J_1 \defeq  \{j\in J_{\rhobar}^c:\lambda_j(x_j) = p-3-x_j\},\quad J_2 \defeq  \{j\in J_{\rhobar}^c:\lambda_j(x_j) = x_j+2\}.
  \end{equation}
We define an $f$-tuple  $\mu=(\mu_j(x_j))$ by 
\begin{equation}\label{eq:mu_j}
\mu_j(x_j)\defeq \begin{cases}p-1-x_j&\text{if}\ j\in J_1,\\
x_j&\text{if}\ j\in J_2,\\
\lambda_j(x_j)&\text{otherwise}.\end{cases}
\end{equation}
It is direct to check that $\mu\in \P$, $\chi_\lambda = \chi_\mu \prod_{j \in J_1} \alpha_j^{-1}\prod_{j \in J_2} \alpha_j$ and $|J_\mu| = |J_\lambda|-|J_1|-|J_2|$.
It is clear that $J_1\subset Y(\mu)$, $J_2\subset Z(\mu)$ and $J_1\cap J_2=\emptyset$. 
By Lemma \ref{lem:W-embeds} there is a unique embedding $\iota: W(\chi_\mu,\chi_\lambda)\hookrightarrow \pi|_I$. Consider the induced morphism 
\[\wt{\iota}: \Ind_I^{\GL_2(\cO_K)}W(\chi_{\mu},\chi_\lambda)\ra \pi|_{\GL_2(\cO_K)}\]
and let $V$ be its image. By Lemma \ref{lem:W-embeds}, $\soc_{\GL_2(\cO_K)}(V)=\sigma$, where $\sigma\in W(\brho)$ is the Serre weight determined by $J_{\sigma}=J_{\brho}\cap J_{\mu}$, so that $\chi_\mu$ contributes to $D_{0,\sigma}(\brho)^{I_1}$ by Lemma~\ref{lem:I1-invt-nonss}. Also, let $\tau\in W(\brho^\ss)$ be the Serre weight determined by $J_{\tau}=J_{\lambda}$, so that $\chi_\lambda$ contributes to $D_{0,\tau}(\brho^\ss)^{I_1}$ by Lemma \ref{lem:I1-invt-component}. Then $\tau$ occurs as a subquotient of $\Ind_I^{\GL_2(\cO_K)}\chi_\lambda$, hence also of  $\Ind_I^{\GL_2(\cO_K)}W(\chi_\mu,\chi_\lambda)$  (with multiplicity one by Lemma \ref{lem:multfree}). 

\begin{lem}\label{lem:JH-V}
Keep the above notation {and assume that $\rhobar$ is $6$-generic}. 
We have $I(\sigma,\tau)\subset V$ and 
\[\JH(V/I(\sigma,\tau))\cap W(\brho^\ss)=\emptyset.\]
\end{lem}
\begin{proof}
By Lemma \ref{lem:W-embeds}\ref{it:W-embeds:2}, the Jordan--H\"older factor $\sigma$ of $\Ind_I^{\GL_2(\cO_K)} \chi_\mu$ is parametrized by the subset 
\begin{equation}\label{eq:J2-X}
\begin{array}{rrl}X(\mu)&=&\{j: \mu_j(x_j)\in\{x_j,\underline{x_j+1},p-2-x_j,p-3-x_j\}\}
\\
&=&J_2\sqcup\{j:\lambda_j(x_j)\in \{x_j,\underline{x_j+1},p-2-x_j,\underline{p-3-x_j}\}\}
\end{array}\end{equation}
and the Jordan--H\"older factor $\tau$ of $\Ind_I^{\GL_2(\cO_K)} \chi_\lambda$ is  parametrized by
\[X^\ss(\lambda)\defeq \{j:\lambda_j(x_j)\in \{x_j, x_j+1,p-2-x_j,p-3-x_j\}\}.\]
({As in the proof of Lemma \ref{lem:I1-invt-nonss}} we use the convention that an underlined entry is only allowed when $j \in J_{\rhobar}$.) 
We check by the definition of $\mu$ and \eqref{eq:J2-X} that $X(\mu)\cap J_1=\emptyset$ and 
\[X(\mu)\sqcup J_1\subset X^\ss(\lambda)\cup J_2.\] 
We then conclude by Proposition \ref{prop:HW}(ii)  (note that $\chi_\mu$ is $5$-generic)  
that $\tau$ contributes to the image $V$ of $\wt\iota$. Moreover, since $J_2\subset X(\mu)$ and $X^{\rm ss}(\lambda)\cap J_2=\emptyset$, Proposition \ref{prop:cond-K1inv} implies that $\tau\in V^{K_1}$ and $I(\sigma,\tau)$ is identified with the unique subrepresentation of $V$ with cosocle $\tau$. This proves the first assertion.

As $V$ is multiplicity free, it remains to prove \[\JH(V)\cap W(\brho^\ss) \subset \JH(I(\sigma,\tau)).\] Let $\chi'\in \JH(W(\chi_\mu,\chi_\lambda))$ and write $\chi'=\chi_\mu^{J_1',J_2'}$ for $J'_1
\subset J_1$ and $J_2'\subset J_2$. By the definition of $\P^\ss$, it is clear that $\chi' = \chi_{\mu'}$ 
for some $\mu'\in \P^\ss$ 
with 
\begin{equation}\label{eq:Xss-mu'}X^{\ss}(\mu')=(X^{\ss}(\lambda)\setminus J_1'')\cup J_2'',\end{equation}
where $J_1''\defeq J_1\setminus J_1'$ and $J_2''\defeq J_2\setminus J_2'$.   
In particular, $\chi'$ is also $5$-generic. 
Let $\tau'\in \JH(\Ind_I^{\GL_2(\cO_K)}\chi')$ be the constituent parametrized by $\mathcal{S}(\tau')$.  
If $\tau'\in W(\brho^\ss)$, then $\mathcal{S}(\tau')\subset X^{\ss}(\mu')$  by Lemma \ref{lem:W-embeds-3} (applied to $\brho^\ss$),
and so  $\mathcal{S}(\tau')\sqcup J_1''\subset X^\ss(\lambda)\cup J_2''$ by \eqref{eq:Xss-mu'}. By Proposition \ref{prop:HW}(ii) applied to $\Ind_I^{\GL_2(\cO_K)}W(\chi',\chi_{\lambda})$, this implies that $\tau'\in \JH(V'_{\tau})$, where $V_{\tau}'\subset\Ind_I^{\GL_2(\cO_K)}W(\chi',\chi_{\lambda})$  is the unique subrepresentation with cosocle $\tau$ (cf.~Remark \ref{rem:HW}). 
Hence $\tau'$ occurs in the unique subrepresentation $\wt V_\tau'$ of $\Ind_I^{\GL_2(\cO_K)}W(\chi_\mu,\chi_{\lambda})$ with cosocle $\tau$.
If moreover $\tau'\in \JH(V)$, then $\tau'$ has to occur in the image $V_\tau \subset V$ of $\wt V_\tau'$, which is just $I(\sigma,\tau)$ by the previous paragraph.
Thus, we obtain $\JH(V)\cap W(\brho^\ss)\subset \JH(V_\tau) = \JH(I(\sigma,\tau))$.
\end{proof}

\subsubsection{Generalization of \texorpdfstring{\cite[\S~19]{BP}}{[BP12,S19]}}
\label{sec:generalization-bp19}

We generalize \cite[Lemma 19.7]{BP} (Lemma \ref{lem:BP-gene}).

Assume that $\brho$ is $3$-generic so that  $\chi_{\lambda}$ is $2$-generic for any $\lambda\in\P^\ss$. Let $\lambda\in \D^\ss$ which corresponds to $\sigma\in W(\brho^\ss)$  and let $\chi\defeq \chi_\sigma=\chi_\lambda$ {(see \S~\ref{sec:notation})}. We let
\begin{equation}\label{eq:def-Rchi}\wt{R}(\chi)\defeq \Ind_I^{\GL_2(\cO_K)}W\Big(\chi^s,\chi^s\prod_{j=0}^{f-1}\alpha_j\Big).\end{equation}
In particular, $\wt{R}(\chi)$ is multiplicity free by Lemma~\ref{lem:multfree}. 
It is isomorphic to the $\GL_2(\cO_K)$-representation denoted by $\wt{R}(\sigma)$ in \cite[\S~17]{BP}. (Indeed, with the notation in \emph{loc.~cit.}, $J_{\sigma}=\{0,\dots,f-1\}$ by our genericity assumption and $W(\chi,\chi\prod_{j=0}^{f-1}\alpha_j^{-1})$ embeds in $\sigma|_{I}$ and is identified with the subspace spanned by $x^{s-i}y^i\in \sigma$ for $i\in\{\sum_{j\in J}p^j: J\subset \{0,\dots,f-1\}\}$, where we have written $\sigma=(s_0,\dots,s_{f-1})\otimes\theta$ and $s\defeq \sum_{j=0}^{f-1}p^js_j$. The representation $\wt{R}(\sigma)$ is then the $\GL_2(\cO_K)$-subrepresentation of $\cInd_{\GL_2(\cO_K)Z}^{\GL_2(K)}\sigma$ generated by $[\smatr01p0,W(\chi,\chi\prod_{j=0}^{f-1}\alpha_j^{-1})]$, hence is isomorphic to our $\wt{R}(\chi)$ in \eqref{eq:def-Rchi}.)  Furthermore, in \emph{loc.~cit.} is defined a subrepresentation of $\wt{R}(\chi)$, denoted by $R(\chi)$, whose Jordan--H\"older factors consist precisely of all the \emph{special} ones, cf.\ \cite[Def.~17.2]{BP}.  
We remark that $R(\chi)$ is the unique subrepresentation of $\wt{R}(\chi)$ whose cosocle is the socle of $\Ind_I^{\GL_2(\cO_K)}\chi^s\prod_{j=0}^{f-1}\alpha_j$.

We recall the following result from   \cite[Lemmas 19.5, 19.7]{BP}.
Let $\delta(\sigma)\in W(\brho^{\ss})$ be the Serre weight corresponding to $\delta(\lambda)\in \D^\ss$ (see \S~\ref{sec:notation} for the definition of $\delta(\lambda)$) {and recall that $\sigma^{[s]}$ is defined in \S~\ref{sec:notation}.}

\begin{prop}\label{prop:BP-Rchi}
There is a unique quotient $Q(\brho^{\ss},\sigma^{[s]})$ of $R(\chi)$ such that:
\begin{itemize}
\item $\soc_{\GL_2(\cO_K)} Q(\brho^\ss,\sigma^{[s]})\subset \bigoplus_{\sigma'\in W(\brho^\ss)}\sigma'$;
\item $Q(\brho^\ss,\sigma^{[s]})$ contains  $I(\delta(\sigma),\sigma^{[s]})$, the unique subrepresentation of $\Inj_{\Gamma}\delta(\sigma)$ with cosocle $\sigma^{[s]}$ in which $\delta(\sigma)$ occurs with multiplicity one.
\end{itemize} 
Moreover, we have $\soc_{\GL_2(\cO_K)} Q(\brho^\ss,\sigma^{[s]})=\delta(\sigma)$ and $Q(\brho^\ss,\sigma^{[s]})$ contains   $D_{0,\delta(\sigma)}(\brho^\ss)$.
\end{prop}

For $J\subset\{0,\dots,f-1\}$, we define
\[\wt{R}_{J}(\chi)\defeq \Ind_I^{\GL_2(\cO_K)}W\big(\chi^s,\chi^s\prod_{j\in J}\alpha_j\big)\hookrightarrow \wt{R}(\chi)\]
and $R_J(\chi)\defeq \wt{R}_J(\chi)\cap R(\chi)$. The following result slightly strengthens Proposition \ref{prop:BP-Rchi}.

\begin{lem}\label{lem:BP-gene}
Let $J\subset \{0,\dots,f-1\}$ be  a subset satisfying 
\[J\supseteq \big\{j: \lambda_j(x_j)\in\{x_j+1,p-2-x_j\}\big\}.\]
Then  $D_{0,\delta(\sigma)}(\brho^\ss)$ is contained in the image of $R_J(\chi)\hookrightarrow R(\chi)\twoheadrightarrow Q(\brho^\ss,\sigma^{[s]})$. 
 In particular, the unique quotient of $R_J(\chi)$ (or of $\wt R_J(\chi)$) with socle $\delta(\sigma)$ contains $D_{0,\delta(\sigma)}(\brho^\ss)$.
\end{lem}
\begin{proof}
Clearly we may assume $J=\{j:\lambda_j(x_j)\in\{x_j+1,p-2-x_j\}\}$. 
Applying Lemma \ref{lem:I1-invt-component} to $\lambda^{[s]}$ {(same notation as (\ref{eq:[s]}))}, and remembering that $\lambda\in\D^{\ss}$, we conclude that $X^{\rm ss}(\lambda^{[s]})=J$ parametrizes $\delta(\sigma)$ (as a constituent of $\Ind_I^{\GL_2(\cO_K)}\chi_\lambda^s$).
Let $\xi\in\mathcal{P}$ correspond to $\delta(\sigma)$, so that $\mathcal{S}(\xi) = J$.
Define $\mu_{\xi}\in \mathcal{I}$ as follows (cf.~\cite[\S~19]{BP}):
\begin{equation*}
\mu_{\xi,j}(y_j)\defeq 
\begin{cases}
  p-1-y_j & \text{if $\xi_j(x_j)\in\{x_j-1,x_j\}$},\\
  p-3-y_j & \text{if $\xi_j(x_j)\in\{p-2-x_j,p-1-x_j\}$}.
\end{cases}
\end{equation*}
Write $\sigma=(s_0,\dots,s_{f-1})\otimes\theta$, so $\delta(\sigma)=(\xi_0(s_0),\dots,\xi_{f-1}(s_{f-1}))\otimes {\det}^{e(\xi)(s_0,\dots,s_{f-1})}\theta$.
By \cite[Lemma 19.2]{BP}, $D_{0,\delta(\sigma)}(\brho^\ss)$ is equal to $I(\delta(\sigma),\tau)$, where 
\[\tau=\mu_{\xi}(\delta(\sigma))\defeq \big(\mu_{\xi,0}(\xi_0(s_0)),\dots,\mu_{\xi,f-1}(\xi_{f-1}(s_{f-1}))\big)\otimes {\det}^{e(\mu_{\xi}\circ\xi)(s_0,\dots,s_{f-1})}\theta.\] 
 By Proposition~\ref{prop:BP-Rchi}, $\tau$ occurs in $Q(\brho^\ss,\sigma^{[s]})$, hence also in $R(\chi)$.

To conclude, it suffices to prove that $\tau$ occurs in $\Ind_I^{\GL_2(\cO_K)}\chi^s\prod_{j\in J}\alpha_j$. By \cite[Lemma 17.12(i)]{BP}, it is equivalent to proving that $J$ equals 
\[J(\mu_{\xi}\circ\xi)\defeq \big\{j\in\{0,\dots,f-1\}: (\mu_{\xi}\circ\xi)(x_j)\in \{x_j-2,p-x_j\}\big\}.\]
By the definition of $\mu_{\xi}$, we see that $J(\mu_{\xi}\circ\xi)$ is exactly 
$\{j:\xi_j(x_j)\in \{x_j-1,p-1-x_j\}\} = \mathcal{S}(\xi)$,  which equals $J$ by above.
\end{proof}

\subsubsection{\texorpdfstring{$K_1$-}{K\_1-} and \texorpdfstring{$I_1$}{I\_1}-invariants of subrepresentations of \texorpdfstring{$\pi$}{pi}}
\label{sec:k_1-invar-subr}

We describe the $K_1$-invariants, $I_1$-invariants and the $\GL_2(\cO_K)$-socles of subrepresentations of $\pi$.

By \cite[Prop.\ 5.2]{Hu-SMF}, {the $\Gamma$-representation} $D_0(\brho)$ admits a unique filtration 
\begin{equation}\label{eq:fil-D0}
0=D_0(\brho)_{\leq -1}\subsetneq D_0(\brho)_{\leq 0}\subsetneq \cdots \subsetneq D_0(\brho)_{\leq i}\subsetneq \cdots \subsetneq D_0(\brho)_{\leq f}=D_0(\brho)
\end{equation}
such that for any $0\leq i\leq f$, 
\begin{equation*}D_{0}(\brho)_i\defeq D_0(\brho)_{\leq i}/D_0(\brho)_{\leq i-1}\end{equation*}
is a subrepresentation of $D_0(\brho^\ss)_i\defeq \bigoplus_{\tau\in W(\brho^\ss),\ell(\tau)=i}D_{0,\tau}(\brho^{\ss})$ and 
\begin{equation}\label{eq:soc-i}
\soc_{\GL_2(\cO_K)}D_0(\brho)_i=\bigoplus_{\tau\in W(\brho^\ss),~\ell(\tau)=i}\tau.\end{equation}
By construction, $D_0(\brho)_{\leq i}$ is the largest $\Gamma$-subrepresentation of $D_0(\brho)$ not containing any $\tau\in W(\brho^\ss)$, $\ell(\tau)>i$ as subquotient.
Set $D_0(\brho^\ss)_{\leq i}\defeq \bigoplus_{j\leq i}D_{0}(\brho^\ss)_j$. We obtain  
\begin{align}
\label{eq:JH-D0-i}\JH(D_0(\brho)_{i})&=\JH(D_0(\brho))\cap  \JH(D_0(\brho^\ss)_{i}), \text{\ \ and}\\
\label{eq:JH-D0-leqi}\JH(D_0(\brho)_{\leq i})&=\JH(D_0(\brho))\cap  \JH(D_0(\brho^\ss)_{\leq i}).
\end{align}
Indeed, \eqref{eq:JH-D0-i} implies \eqref{eq:JH-D0-leqi}. 
For \eqref{eq:JH-D0-i}, the inclusion $\subseteq$ is obvious, but both sides form a partition of $\JH(D_0(\brho))$ as $i$ varies, so equality holds.

Since $D_0(\brho)$ is multiplicity free and decomposes as $\bigoplus_{\sigma\in W(\brho)}D_{0,\sigma}(\brho)$, we see that $D_{0}(\brho)_{\leq i}$ also decomposes as a direct sum
\begin{equation}\label{eq:decom-leqi}D_{0}(\brho)_{\leq i}=\bigoplus_{\sigma\in W(\brho)}D_{0,\sigma}(\brho)_{\leq i},\end{equation}
where $D_{0,\sigma}(\brho)_{\leq i}\defeq D_{0,\sigma}(\brho)\cap D_{0}(\brho)_{\leq i}$. 
(Note that by \eqref{eq:JH-D0-leqi} we have $D_{0,\sigma}(\brho)_{\leq i} \ne 0$ if and only if $\ell(\sigma)\leq i$.)
Similarly, 
$D_0(\brho)_i$ also decomposes as a direct sum  $\bigoplus_{\tau\in W(\brho^\ss),\ell(\tau)=i}D_{0,\tau}(\brho)_i$, where 
$D_{0,\tau}(\brho)_{i} \defeq  D_0(\brho)_i \cap D_{0,\tau}(\brho^\ss)$.

We remark that by \eqref{eq:JH-D0-leqi} and Lemma~\ref{lem:inter} we have for any $\sigma \in W(\brho)$:
\begin{equation}\label{eq:D0-sigma-i-decomp}
  \frac{D_{0,\sigma}(\brho)_{\le i}}{D_{0,\sigma}(\brho)_{\le i-1}} = \bigoplus_{\tau \in W(\brho^\ss), \ell(\tau)=i, J_\sigma = J_{\brho} \cap J_\tau} D_{0,\tau}(\brho)_{i}.
\end{equation}

\begin{lem}\label{lem:inter}
Let $\tau\in W(\brho^{\rm ss})$ and $\sigma\in W(\brho)$ be the element such that $J_{\sigma}=J_{\brho}\cap J_{\tau}$. Then
\[\JH(D_{0,\tau}(\brho^{\rm ss}))\cap \JH(D_0(\brho))\subset \JH(D_{0,\sigma}(\brho)).\]
\end{lem}
\begin{proof}
This is a consequence of \cite[Lemma 15.3]{BP}.
\end{proof}

\begin{thm}\label{thm:conj2}
 Assume that $\rhobar$ is $6$-generic.  
Let $\pi_1$ be a  subrepresentation of $\pi$. Then there exists a unique integer $i_0 = i_0(\pi_1)$ with $-1\leq i_0\leq f$ such that
\[\pi_1^{K_1}=D_0(\brho)_{\leq i_0}. \]
\end{thm}
\begin{proof}
If $\pi_1^{K_1}=0$ (resp.\ $\pi_1^{K_1}=D_0(\brho)$) we are done, with $i_0 = -1$ (resp.\ $i_0 = f$).
Otherwise, by \eqref{eq:fil-D0} there exists a unique integer $-1 < i_0< f$ such that $D_0(\brho)_{\leq i_0}\subset \pi_1^{K_1}$ and $D_0(\brho)_{\leq i_0+1}\nsubseteq \pi_1^{K_1}$. We need to prove that the (first) inclusion is an equality. 
Suppose this is not the case. Then we may find a Serre weight $\tau$ which embeds in $\pi_1^{K_1}/D_0(\brho)_{\leq i_0}$, hence also embeds in $D_0(\brho)/D_{0}(\brho)_{\leq i_0}$. This implies that $\tau\in W(\brho^\ss)$ with $\ell(\tau)>i_0$ by \eqref{eq:fil-D0} and \eqref{eq:soc-i}. Thus, there exists a Serre weight $\tau$ satisfying the condition 
\begin{equation}\label{eq:cond-tau}
\tau\in W(\brho^{\rm ss})\cap \JH(\pi_1^{K_1}), \ \ \ell(\tau)>i_0.
\end{equation} We choose $\tau$ satisfying \eqref{eq:cond-tau} such that $\ell(\tau)$ is minimal.

\textbf{Step 1.} We prove that $\ell(\tau)=i_0+1$. First assume $\tau\in W(\brho^\ss)\setminus W(\brho)$ and  let $\sigma\in W(\brho)$ be the Serre weight with $J_{\sigma}=J_{\brho}\cap J_{\tau}$. Note that $\tau \in \JH(D_{0}(\brho))$ by Lemma~\ref{lem:Dss-in-D0}, so $I(\sigma,\tau)\hookrightarrow D_{0}(\brho)$ by Lemma~\ref{lem:JHinW}, and thus $I(\sigma,\tau)\subset \pi_1^{K_1}$. Since $\sigma\neq \tau$, we have $\rad_\Gamma(I(\sigma,\tau))\neq0$, and by using again Lemma~\ref{lem:JHinW} we have
\[\JH(\rad_\Gamma(I(\sigma,\tau)))\subset W(\brho^{\rm ss})\cap \JH(\pi_1^{K_1}).\] 
By the choice of $\tau$, we must have $\ell(\tau')\leq i_0$ for any $\tau'\in \JH(\rad_\Gamma(I(\sigma,\tau)))$. Then by the second sentence of Lemma \ref{lem:JHinW} and remembering that $\ell(\tau') = |J_{\tau'}|$ for $\tau' \in W(\brho^\ss)$, this forces $\ell(\tau)\leq i_0+1$, hence $\ell(\tau)=i_0+1$ (as $\ell(\tau)>i_0$ by construction).   

Next, we assume that $\tau\in W(\brho)$, i.e.~$\tau$ occurs in the $\GL_2(\cO_K)$-socle of $\pi$. This is equivalent to  $J_{\tau}\subset J_{\brho}$. 
Note that in this case we have $\tau \into \pi_1^{K_1}$.
By assumption, $\ell(\tau)=i_0+1>0$.
By Lemma \ref{lem:I1-invt-component} (resp.\ Lemma~\ref{lem:I1-invt-nonss}), using the observation $J_{\mu^{[s]}} = \delta(J_\mu)$ for $\mu \in \D^\ss$, the Serre weight ${\tau}^{[s]}$ occurs in $D_{0,\tau_1}(\brho^\ss)$ (resp.\ $D_{0,\sigma_1}(\brho)$), where $\tau_1\in W(\brho^\ss)$ and $\sigma_1\in W(\brho)$ are uniquely determined by \[J_{\tau_1}=\delta(J_{\tau}),\ \ J_{\sigma_1}=J_{\brho}\cap \delta(J_{\tau}).\] 
Moreover, the image of $\Ind_I^{\GL_2(\cO_K)}\chi_{\tau}^s\ra \pi_1$ is equal to $I(\sigma_1,\tau^{[s]})$, which contains $\tau_1$ as a subquotient (by Lemma \ref{lem:W-embeds-3} applied to $\brho^{\rm ss}$ and $\chi_\mu=\chi_{\tau^{[s]}}$), so we have $\tau_1\in \JH(\pi_1^{K_1})$.   We note that  $\ell(\tau_1)=\ell(\tau)>i_0$, thus $\tau_1$ also satisfies \eqref{eq:cond-tau} and $\ell(\tau_1)$ is minimal subject to \eqref{eq:cond-tau}.
If again $\tau_1\in  W(\brho)$, i.e.~$J_{\tau_1}\subset J_{\brho}$, we may continue this procedure to obtain $\tau_2$ and $\sigma_2$. Since $J_\tau \ne \emptyset$, $J_{\brho}\neq \{0,\dots,f-1\}$, we finally arrive  at some $\tau_n$ with $J_{\tau_n}=\delta^n(J_{\tau})\nsubseteq J_{\brho}$, equivalently $\tau_n\in W(\brho^\ss)\setminus W(\brho)$,  and we are reduced to the case in the previous paragraph. Thus $\ell(\tau)=\ell(\tau_n)=i_0+1$ as desired.

\textbf{Step 2.} Let $\lambda\in \D^\ss$ be the element corresponding to $\tau$, and define the $f$-tuple $\mu=(\mu_j(x_j))$ as in (\ref{eq:mu_j}), i.e.\ $\mu_j(x_j)=p-1-x_j$ if $j\in J_1$ and $\mu_j(x_j)=\lambda_j(x_j)$ otherwise, where 
\[J_1\defeq \{j\in J_{\brho}^c:\lambda_j(x_j)=p-3-x_j\} \quad (\text{and}\ J_2=\emptyset).\]
It is direct to check that $\mu\in \P$ and $\chi_{\lambda}=\chi_{\mu}\prod_{j\in J_1}\alpha_j^{-1}$.  
We also note that $J_{\brho} \cap J_\lambda \subset J_\mu \subset J_\lambda$, i.e.\ $J_{\brho} \cap J_\lambda = J_\mu$. 
Let
\[ \wt{J}_1\defeq \{j:\lambda_j(x_j)\in\{x_j+1,p-2-x_j\}\} = \{j:\mu_j(x_j)\in\{x_j+1,p-2-x_j\}\}.\]
Then $J_1\cap \wt{J}_1=\emptyset$ and $J\defeq J_1\sqcup \wt{J}_1\subset Y(\mu)$, where $Y(\mu)$ is defined in \eqref{eq:Y-mu}. By Lemma \ref{lem:W-embeds}, there is a unique (up to scalar) $I$-equivariant embedding   $\iota: W(\chi_\mu,\chi'')\hookrightarrow \pi|_I$, where
\[\chi''\defeq \chi_{\mu}^{J,\emptyset}=\chi_{\mu}\prod_{j\in J}\alpha_j^{-1}.\] 
Note that $W(\chi_\mu,\chi_\lambda)\into W(\chi_\mu,\chi'')$ by construction (Lemma \ref{lem:Wchi}). 

\textbf{Step 3.}  We prove that $\im(\iota)$ is contained in $\pi_1$. 
It is equivalent to prove that  $V$ is contained in $\pi_1$, where $V$ denotes the image of the  $\GL_2(\cO_K)$-equivariant morphism (induced by Frobenius reciprocity)
\[\wt{\iota}:\Ind_I^{\GL_2(\cO_K)}W(\chi_{\mu},\chi'')\ra \pi|_{\GL_2(\cO_K)}. \]
Note that $V$ is contained in $\pi^{K_1} \cong \bigoplus_{\sigma' \in W(\brho)} D_{0,\sigma'}(\brho)$, since $W(\chi_\mu,\chi'')$ is fixed by $K_1$ by Lemma \ref{lem:Wchi}. By Lemma \ref{lem:W-embeds}, $V$ is contained in $D_{0,\sigma}(\brho)$, where $\sigma \in W(\brho)$ is as in Step 1.  
For $J'\subset J$, let $\tau^{J'}$ be the cosocle of $\Ind_I^{\GL_2(\cO_K)}\chi_\mu\prod_{j\in J'}\alpha_j^{-1}$, i.e.~the unique Serre weight with $(\tau^{J'})^{I_1}=\chi_{\mu}\prod_{j\in J'}\alpha_j^{-1}$. Note that $\tau^{J_1}=\tau$. 
It follows from Proposition \ref{prop:HW}(i) that 
\begin{equation}
\cosoc_\Gamma(V)\cong\bigoplus_{J'\subset J,~\tau^{J'}\in \JH(V)}\tau^{J'}.\label{eq:cosoc-V}
\end{equation}
By multiplicity freeness of $\pi^{K_1}$ it suffices to show that $\tau^{J'}$ occurs in $\pi_1^{K_1}$ for each $J'\subset J$ satisfying $\tau^{J'}\in \JH(V)$.  
If $J' = J_1$, this is true by assumption, so we may assume $J' \ne J_1$ in the following. 

We have $I(\sigma,\tau)\hookrightarrow  V$ by Lemma \ref{lem:JH-V}, and    $\JH(I(\sigma,\tau))\subset W(\brho^\ss)$ by Lemma \ref{lem:JHinW}. 
Moreover, we have 
\begin{equation}\label{eq:V/I}
\JH(V/I(\sigma,\tau))\cap W(\brho^\ss)=\emptyset.\end{equation}
This follows from Lemma \ref{lem:JH-V} by noting that if $\chi'\in \JH(W(\chi_{\mu},\chi''))\setminus \JH(W(\chi_{\mu}, \chi_{\lambda}))$, then  $\chi'\notin \JH(D_0(\brho^\ss)^{I_1})$   by the explicit description of $\mathscr{P}^\ss$ and so $\JH(\Ind_I^{\GL_2(\cO_K)}\chi')\cap W(\brho^\ss)=\emptyset$
by \cite[Prop.~4.2]{breuil-buzzati}. 

Now fix $J'\subset J$ satisfying $J' \ne J_1$ and $\tau^{J'}\in \JH(V)$. In particular $\tau^{J'}\neq \tau$. 
As $V$ is $K_1$-invariant, $I(\sigma,\tau^{J'}) \subset V$.
If $I(\sigma,\tau^{J'})\nsubseteq D_0(\brho)_{\leq i_0}$, equivalently the morphism $I(\sigma,\tau^{J'})\ra D_0(\brho)/D_{0}(\brho)_{\leq i_0}$ is nonzero, then $\JH(I(\sigma,\tau^{J'}))$ would contain some element $\tau'\in W(\brho^\ss)$ with $\ell(\tau')\geq i_0+1$, by \eqref{eq:soc-i}. 
As $\tau'$ must contribute to $I(\sigma,\tau)$ by \eqref{eq:V/I},  by Lemma~\ref{lem:JHinW} we deduce $\tau' = \tau$ (as otherwise $\ell(\tau') < \ell(\tau) = i_0+1$) and hence $\tau \in \JH(I(\sigma,\tau^{J'}))$.
But $\tau$ is a quotient of $V$ by~(\ref{eq:cosoc-V}) and hence of $I(\sigma,\tau^{J'})$, so $\tau^{J'}=\tau$, contradiction.
Hence $\tau^{J'}$ occurs in $D_0(\brho)_{\leq i_0} \subset \pi_1^{K_1}$, as desired.

\textbf{Step 4.} 
Our goal is to prove that $D_{0}(\brho)_{\leq i_0+1}\subset \pi_1^{K_1}$, which will contradict our choice of $i_0$. By the multiplicity freeness of $\pi^{K_1}$, it suffices to prove  
\[
\JH(D_{0}(\brho)_{i_0+1})\subset \JH(\pi_1^{K_1}),\]
or equivalently (by \eqref{eq:JH-D0-i}),
\begin{equation}\label{eq:check-tau'}
\JH(D_{0,\tau'}(\brho^\ss)) \cap \JH(D_{0}(\brho)) \subset \JH(\pi_1^{K_1})\end{equation}  for any $\tau' \in W(\brho^\ss)$ satisfying $\ell(\tau')=i_0+1$.  
In this step we prove that~(\ref{eq:check-tau'}) holds under the additional hypothesis that $\tau' \in \JH(\pi_1^{K_1})$.

We may assume that $\tau' = \tau$ and let {again} $\lambda\in\D^\ss$ correspond to $\tau$. 
Since $\pi_1$ carries an action of $\smatr{0}1p0$, we deduce an injective morphism $\kappa:W(\chi_{\mu}^s,\chi''^{s})\hookrightarrow \pi_1|_I$ from Step~3, hence a $\GL_2(\cO_K)$-equivariant morphism (induced by Frobenius reciprocity)
\[\wt{\kappa}:\Ind_I^{\GL_2(\cO_K)}W(\chi_{\mu}^s,\chi''^s)\ra \pi_1|_{\GL_2(\cO_K)}.\]
Let $\sigma_1\in W(\brho)$ be the Serre weight such that $\chi_{\mu}^s$ contributes to $D_{0,\sigma_{1}}(\brho)^{I_1}$. Then $\sigma_1$ occurs in $\Ind_I^{\GL_2(\cO_K)}\chi_\mu^s$ and is parametrized by $X(\mu^{[s]})$ {(recall from (\ref{eq:[s]}) that $\mu^{[s]}\in \P$ is the $f$-tuple corresponding to $\chi_\mu^s$)}.
Similarly, let $\tau_1=\delta(\tau)\in W(\brho^\ss)$ be such that $\chi_\lambda^s$  contributes to $D_{0,\tau_1}(\brho^{\ss})^{I_1}$, then 
$\tau_1$ occurs in $\Ind_I^{\GL_2(\cO_K)}\chi_\lambda^s$ and is parametrized by $X^{\ss}(\lambda^{[s]})$.
By Lemma~\ref{lem:W-embeds} (applied with $(\mu^{[s]},\emptyset,J)$ for $(\mu,J_1,J_2)$ there) we see that $\soc_{\GL_2(\cO_K)}(\im(\wt{\kappa})) = \sigma_1$.
By Lemma~\ref{lem:JH-V} (applied with $\lambda^{[s]}$, resp.\ $\mu^{[s]}$, instead of $\lambda$, resp.\ $\mu$, and noting that $W(\chi_\mu^s,\chi_\lambda^s)\into W(\chi_\mu^s,\chi''^s)$) we deduce that $I(\sigma_1,\tau_1) \subset \im(\wt{\kappa})^{K_1}\subset \pi_1^{K_1}$.
In particular, $\tau_1 \in \JH(\pi_1^{K_1})$.
We also note that $J_{\sigma_1} = J_{\brho} \cap J_{\tau_1}$.
(Using Lemmas \ref{lem:I1-invt-nonss} and \ref{lem:I1-invt-component} we have $J_{\sigma_1} = J_{\brho} \cap J_{\mu^{[s]}}$, $J_{\tau_1} = J_{\lambda^{[s]}}$, and recall from Step 2 that $\lambda_j = \mu_j$, hence $\lambda_j^{[s]} = \mu_j^{[s]}$, for all $j \in J_1^c \supset J_{\brho}$.)

By Lemma \ref{lem:BP-gene} applied to $\wt{R}_{\wt{J}_1}(\chi_\lambda)$, and noting that we have a surjection 
\[
\Ind_I^{\GL_2(\cO_K)}W(\chi_{\mu}^s,\chi''^s)\onto \Ind_I^{\GL_2(\cO_K)}W(\chi_{\lambda}^s,\chi''^s) = \wt{R}_{\wt{J}_1}(\chi_\lambda),
\] 
we see that the unique quotient $Q_{\tau_1}$ of $\Ind_I^{\GL_2(\cO_K)}W(\chi_{\mu}^s,\chi''^s)$ with socle $\delta(\tau)=\tau_1$ contains $D_{0,\tau_1}(\brho^\ss)$.
As $\tau_1$ occurs in $Q_{\sigma_1} = \im(\wt{\kappa})$  (the unique quotient with socle $\sigma_1$), we see that 
$Q_{\sigma_1}$ surjects onto $Q_{\tau_1}$ and hence contains $D_{0,\tau_1}(\brho^\ss)$ as subquotient. 
By Lemma \ref{lem:inter}, we have \[\JH(D_{0,\tau_1}(\brho^{\ss}))\cap \JH(D_0(\brho))\subset \JH(D_{0,\tau_1}(\brho^{\ss}))\cap \JH(D_{0,\sigma_1}(\brho))\subset \JH(Q_{\sigma_1})\cap \JH(D_{0,\sigma_1}(\brho))\subset \JH(Q_{\sigma_1}^{K_1}),\]
where the last inclusion results from Proposition \ref{prop:cond-K1inv} (applied with $\sigma = \sigma_1$ and varying $\tau$). 
As $Q_{\sigma_1} = \im(\wt{\kappa}) \subset \pi_1$, 
\eqref{eq:check-tau'} holds for $\tau'=\tau_1$.  
Repeating the same argument with $\tau'=\tau_1 \in W(\brho^{\rm ss}) \cap \JH(\pi_1^{K_1})$, which still has length $i_0+1$,
we see that \eqref{eq:check-tau'} holds for all $\delta^n(\tau)$, in particular for $\tau$ itself as $\delta(\cdot)$ is periodic.  
Thus, we deduce that \eqref{eq:check-tau'} holds for all $\tau' \in W(\brho^\ss)$ such that $\ell(\tau')=i_0+1$ and $\tau' \in \JH(\pi_1^{K_1})$.
 
\textbf{Step 5.} We modify the proof of \cite[Thm.~15.4]{BP} to show that \eqref{eq:check-tau'} holds for any  $\tau'\in W(\brho^{\ss})$ with $\ell(\tau')=i_0+1$. 
We may assume that $i_0+1<f$.
As in the previous step 
we start with $\tau\in W(\brho^\ss) \cap \JH(\pi_1^{K_1})$ and recall that $\ell(\tau)=i_0+1$.
Write $J_\tau=\mathcal{S}_{1}\sqcup \cdots\sqcup \mathcal{S}_r$ with $\mathcal{S}_{i}=\{a_i,a_i+1,\dots,b_i=a_i+\ell_i-1\}$ (thought of inside $\Z/f\Z$), $0 \le a_1 < a_2 < \dots < a_r < f$, and $b_{i}+1 \notin J_\tau$ for each $1\leq i\leq r$.
In particular, $\ell(\tau) = \sum_{i=1}^r \ell_i$.
Fix $1\le i\le r$. Define an $f$-tuple $\lambda$ as follows (note that $\lambda$ has a different meaning than in the previous steps): 
\[\lambda_j(x_j)=\begin{cases}p-3-x_j&\text{if}\ j\in J_\tau\setminus\{b_i\},\\
x_j+1&\text{if}\ j=b_i,\\
p-2-x_j&\text{if}\ j=b_i+1,\\
p-1-x_j &\text{otherwise}.\end{cases}\] 
Then it is direct to check that $\lambda\in \mathscr{P}^{\rm ss}$ and $|J_{\lambda}|=i_0+1$.  Moreover, letting $\tau'\in W(\brho^{\ss})$ be the element such that $\chi_{\lambda}^s$ contributes to $D_{0,\tau'}(\brho^\ss)$,  by Lemma \ref{lem:I1-invt-component}  applied to $\chi_{\lambda}^s$ we have
\begin{equation}
J_{\tau'}=(J_\tau\setminus\{b_i\})\sqcup \{b_i+1\},\label{eq:J-tau'}
\end{equation} 
so in particular, $\ell(\tau') = \ell(\tau) = i_0+1$.
Below we will prove that $\tau'\in \JH(\pi_1^{K_1})$, so that \eqref{eq:check-tau'} holds for $\tau'$ by Step 4. By repeating this procedure, it is easy to see using~(\ref{eq:J-tau'}) that \eqref{eq:check-tau'} holds for any $\tau'\in W(\brho^{\ss})$ of length $i_0+1$.

We 
define $\mu\in\mathscr{P}$ as in Step 2, with $J_1\defeq \{j\in J_{\brho}^c: \lambda_j(x_j)=p-3-x_j\}$ (and $J_2=\emptyset$).
Then $W(\chi_{\mu},\chi_{\lambda})\hookrightarrow \pi|_{I_1}$ by Lemma \ref{lem:W-embeds}.  
We claim that the image is contained in $\pi_1$, or equivalently that the image $V$ of the induced map $\Ind_I^{\GL_2(\cO_K)}W(\chi_{\mu},\chi_{\lambda}) \to \pi^{K_1}$ is contained in $\pi_1^{K_1}$.
As in Step 3, letting $\tau^{J'}$ denote the cosocle of $\Ind_I^{\GL_2(\cO_K)}\chi_\mu\prod_{j\in J'}\alpha_j^{-1}$, where $J'\subset J_1$, it suffices to show that $\tau^{J'} \in \JH(V)$ implies $\tau^{J'} \in \JH(\pi_1^{K_1})$ for any $J' \subset J_1$.
Assume $\tau^{J'} \in \JH(V)$ for some $J' \subset J_1$ and define an $f$-tuple $\mu'$ by $\mu'_j(x_j) = \mu_j(x_j)-2=\lambda_j(x_j)$ if $j \in J'$, $\mu'_j(x_j) = \mu_j(x_j)$ otherwise, so that $\chi_{\mu'}=\chi_{\mu}\prod_{j\in J'}\alpha_j^{-1}$.
Then $\mu' \in \P^\ss$ and $|J_{\mu'}| \le |J_\lambda| = i_0+1$, with equality holding if and only if $\mu'=\lambda$ (i.e.\ $J'=J_1$).
If $J' = J_1$, then $\tau^{J'} \in \JH(D_{0,\tau}(\brho^\ss)) \cap \JH(D_0(\brho)) $ (by Lemma \ref{lem:I1-invt-component}) and so $\tau^{J'} \in \JH(\pi_1^{K_1})$ by \eqref{eq:check-tau'} for $\tau' = \tau$ in Step 4.
If $J' \subsetneq J_1$, then $\tau^{J'} \in \JH(D_{0}(\brho^\ss)_{\le i_0}) \cap \JH(D_0(\brho)) = \JH(D_0(\brho)_{\leq i_0})\subset \JH(\pi_1^{K_1})$, by assumption.
This proves the claim. By \ Lemma~\ref{lem:JH-V}, \ $\tau'$ \ is \ contained \ in \ the \ $K_1$-invariants \ of \ the \ image \ of $\Ind_I^{\GL_2(\cO_K)}W(\chi_{\mu}^s,\chi_{\lambda}^s)$ in $\pi_1$, hence $\tau'\in \JH(\pi_1^{K_1})$ as desired.
\end{proof}

\begin{cor}\label{cor:conj1}
  Let $i_0 = i_0(\pi_1)$ with $-1 \le i_0 \le f$ be as in Theorem \ref{thm:conj2}.
  Then
  \begin{equation}\label{eq:2}
    \JH(\pi_1^{I_1}) = \{ \chi_\lambda : \text{$\lambda \in \P$ such that $|J_\lambda| \le i_0$} \}.
  \end{equation}
\end{cor}
\begin{proof}
By \ Lemma \ \ref{lem:I1-invt-component}, \ $\chi$ \ is \ contained \ in \ the \ right-hand \ side \ of~(\ref{eq:2}) \ if \ and \ only \ if \ $\chi\in \JH(D_0(\brho)^{I_1})\cap \JH(D_0(\brho^\ss)_{\leq i_0}^{I_1})$.
As $\JH(\pi_1^{I_1}) \subset \JH(D_0(\brho)^{I_1})\cap \JH(D_0(\brho^\ss)_{\leq i_0}^{I_1})$ by \eqref{eq:JH-D0-leqi} and Theorem \ref{thm:conj2}, we deduce that ``$\subset$'' holds in~(\ref{eq:2}).
Conversely, if $\chi \in \JH(D_0(\brho)^{I_1})\cap \JH(D_0(\brho^\ss)_{\leq i_0}^{I_1})$, then $\chi$ contributes to $D_0(\brho)_i^{I_1}$ for some $i$, hence to $D_0(\brho^\ss)_i^{I_1}$, which implies $i \le i_0$.
In particular, $\chi$ does not contribute to $D_0(\brho)_i^{I_1}$ for any $i > i_0$, so $\chi$ contributes to $D_0(\brho)_{\le i_0}^{I_1} = \pi_1^{I_1}$ by Theorem \ref{thm:conj2}.
\end{proof}

\begin{cor}\label{cor:pi1-K-soc}
  Let $i_0 = i_0(\pi_1)$ with $-1 \le i_0 \le f$ be as in Theorem \ref{thm:conj2}.
  Then
  \begin{equation*}
    \soc_{\GL_2(\cO_K)}(\pi_1) \cong \bigoplus_{\sigma \in W(\brho),\ell(\sigma) \le i_0} \sigma.
  \end{equation*}
\end{cor}

\begin{proof}
  Note that $\soc_{\GL_2(\cO_K)}(\pi_1)$ is multiplicity free by Corollary~\ref{cor:conj1}.
  If $\sigma \subset \pi_1|_{\GL_2(\cO_K)}$ is an irreducible subrepresentation, then $\sigma \in W(\brho)$, $\ell(\sigma) \le i_0$ by Theorem~\ref{thm:conj2} and \eqref{eq:soc-i}.
  Conversely, suppose that $\sigma \in W(\brho)$ with $\ell(\sigma) \le i_0$.
 Then $\sigma\hookrightarrow D_{0}(\brho)_{\leq i_0}$ by the sentence after \eqref{eq:soc-i}, hence $\sigma \subset \pi_1|_{\GL_2(\cO_K)}$ by Theorem~\ref{thm:conj2}.
\end{proof}

\begin{rem}\label{rem:pi1-K-soc}
  In particular, a subrepresentation $\pi_1$ is not determined by $\pi_1^{I_1}$. 
  For example, if $J_{\brho} = \emptyset$, then it follows from the definitions that $|J_{\lambda}| \le f/2$ for all $\lambda \in \P$.
  Likewise, $\pi_1$ is not determined by $\soc_{\GL_2(\cO_K)}(\pi_1)$.
  (On the other hand, $\pi_1$ is determined by $\pi_1^{K_1}$ by Theorem \ref{thm:fin-length-nonsplit}.)
\end{rem}

We conclude with a result on higher Iwahori invariants.

\begin{prop}\label{prop:conj5}
 Assume that $\rhobar$ is $\max\{6,2f+1\}$-generic.  
Let $i_0 = i_0(\pi_1)$ with $-1 \le i_0 \le f$ be as in Theorem \ref{thm:conj2}.
Then for any $\lambda \in \P^\ss \setminus \P$ such that $|J_\lambda| = i_0+1$, the character $\chi_\lambda$ does not occur in $\pi_1[\m^{f+1}]$.
\end{prop}

\begin{proof}

Define disjoint subsets $J_1$, $J_2$ of $\{0,1,\dots,f-1\}$ and $\mu\in \mathscr P$ as in \eqref{eq:J1:J2} and \eqref{eq:mu_j}.

We let again $\sigma\in W(\brho)$ be the Serre weight determined by $J_\sigma=J_{\brho}\cap J_{\mu}$ and $\tau\in W(\brho^\ss)$ be the Serre weight determined by $J_{\tau}=J_{\lambda}$.
We also recall that there is a unique embedding $\iota:W(\chi_\mu,\chi_\lambda)\hookrightarrow \pi|_{I}$ and let $V$ be the image of the induced morphism $\wt\iota:\Ind_I^{\GL_2(\cO_K)} W(\chi_{\mu},\chi_{\lambda}) \to \pi$.
By Lemma \ref{lem:JH-V}, $I(\sigma,\tau)\subset V^{K_1}$.

Note that $\chi_\lambda$ contributes to $D_{0,\tau}(\brho^\ss)^{I_1}$ by Lemma \ref{lem:I1-invt-component}, so $\ell(\tau)=|J_{\lambda}|=i_0+1$.

Suppose by contradiction that $\chi_\lambda \in \JH(\pi_1[\m^{f+1}])$.
As $|J_1|+|J_2| \le f$ we see by Lemma~\ref{lem:Wchi} and multiplicity freeness of $\pi[\m^{f+1}]$ (which holds by Corollary~\ref{cor:tau-multfree}(ii), applied with $n = f+1$ and $r = 1$) that $\im(\iota) \subset \pi_1$ and hence $V \subset \pi_1$.
Since $I(\sigma,\tau)\subset V^{K_1}$, we deduce that $\tau \in \JH(\pi_1^{K_1})$. 
By Theorem \ref{thm:conj2}, $\JH(\pi_1^{K_1}) \subset \JH(D_0(\brho^\ss)_{\le i_0})$, contradicting $\ell(\tau)=i_0+1$.
\end{proof}

\subsection{Finite length}
\label{sec:finite-length-nonss-sub}

We prove that (the duals of) subrepresentations and quotients of $\pi$ are Cohen--Macaulay $\Lambda$-modules of grade $2f$. We deduce many results on the structure of $\pi$ as a $\GL_2(K)$-representation, including that it is of finite length. 

{In the proofs we will use the functor $D_{\xi}^{\vee}$ (see the paragraph preceding Proposition \ref{prop:split-I1})}. We first state a theorem of Yitong Wang \cite[Thm.\ 1.2]{YitongWang} that will be essential for our proof. 

\begin{thm}[Y.\ Wang]\label{thm:yitong}
 Assume that $2f<r_j<p-2-2f$ for all $0\leq j\leq f-1$. 
  Let $\pi_1$ be a subrepresentation of $\pi$. Then we have
  \begin{equation*}
    \dim_{\F\ppar{X}} D_{\xi}^{\vee}(\pi_1)=|\JH(\pi_1^{K_1}) \cap W(\rhobar^\ss)|.
  \end{equation*}
\end{thm}

By equation \eqref{eq:soc-i} we deduce the following corollary.

\begin{cor}\label{cor:yitong} 
 Assume that $\brho$ is $\max\{6,2f+1\}$-generic. 
  Let $i_0 = i_0(\pi_1)$ with $-1 \le i_0 \le f$ be as in Theorem \ref{thm:conj2}.
  Then \[\dim_{\F\ppar{X}} D_{\xi}^{\vee}(\pi_1)=\sum_{i \le i_0} \binom fi.\]
\end{cor}

We denote by $N$ the graded module defined in \S~\ref{sec:some-homol-argum}, namely
\[N\defeq \bigoplus_{\lambda\in\mathscr{P}}\chi_{\lambda}^{-1}\otimes \frac{R}{\mathfrak{a}(\lambda)}.\]
If $\rhobar$ is moreover 9-generic we have $\gr_{\m}(\pi^{\vee})\cong N$ by Theorem \ref{thm:CMC}.
From now on, we thus assume that $\rhobar$ is $\max\{9,2f+1\}$-generic (in addition to assumptions \ref{it:assum-i}--\ref{it:assum-iv}).

\begin{prop}\label{prop:nonsplit-I1}
Assume that $\brho$ is $\max\{9,2f+1\}$-generic.
Let $0 \subsetneq \pi_1 \subsetneq \pi$ be a subrepresentation of $\pi$ and let $\pi_2\defeq \pi/\pi_1$. 
Then both $\gr_\m(\pi_1^{\vee})$ and $\gr_F(\pi_2^{\vee})$ are Cohen--Macaulay $\gr(\Lambda)$-modules of grade $2f$, where $F$ denotes the filtration induced from $\pi^\vee$.
In particular, $\pi_1^\vee$ and $\pi_2^\vee$ are Cohen--Macaulay $\Lambda$-modules of grade $2f$.
\end{prop}

\begin{rem}\label{rem:nonsplit-I1}
  It is easy to see that $F$ does not coincide with the $\m$-adic filtration in general (when $\rhobar$ is nonsplit), for example because $\pi_2^{I_1}$ is bigger than $\pi^{I_1}/\pi_1^{I_1}$ already when $f = 1$ and $\pi_1$ is a principal series representation.
 {We will determine $\gr_\m(\pi_2^\vee)$ in \cite{BHHMS5}.} 
\end{rem}

 Recall the ideals $I(J_1,J_2,d)$ and $I(J_1,J_2,d,\un t) = I(J_1,J_2,d)+(\un t)$ of $\o R$ from Definition \ref{def:IJideals}, where $J_1$, $J_2$ are disjoint subsets of $\{0,\dots,f-1\}$, $d \in \Z$, and $t_j \in \{y_j,z_j,y_jz_j\}$ for all $0 \le j \le f-1$. 
If $d \ge 1$, the ideal $I(J_1,J_2,d)$ is generated by all $\prod_{j \in J_1'} y_j \prod_{j \in J_2'} z_j$ with $J_1' \subset J_1$, $J_2' \subset J_2$, $|J'_1|+|J'_2| = d$ (plus all $t_j$ {for $I(J_1,J_2,d,\un t)$}). If $d \le 0$ {these ideals equal} $\o R$. 

For $\lambda \in \P$ define the ideal of $\o R$,
\begin{equation}\label{eq:fa-1}\mathfrak a_1^{i_0}(\lambda) \defeq  I(J_1,J_2,i_0+1-|J_\lambda|) + \mathfrak a(\lambda),\end{equation}
where $J_1 \defeq  \{ j \in J_{\rhobar}^c : \lambda_j(x_j) = p-1-x_j \}$ and $J_2 \defeq  \{ j \in J_{\rhobar}^c : \lambda_j(x_j) = x_j \}$.
In other words, $\mathfrak a_1^{i_0}(\lambda) = I(J_1,J_2,i_0+1-|J_\lambda|,\un t)$, where $t_j \in \{y_j,z_j,y_jz_j\}$ is defined in~(\ref{eq:id:al}) in terms of $\lambda$.
(Note that $t_j = y_jz_j$ for all $j \in J_1 \sqcup J_2$.)
By definition, $\mathfrak a_1^{i_0}(\lambda) = \o R$ if $i_0 < |J_\lambda|$ and $\mathfrak a_1^{i_0}(\lambda) = \mathfrak a(\lambda)$ if $|J_1|+|J_2| < i_0+1-|J_\lambda|$.

\begin{proof}[Proof of Proposition \ref{prop:nonsplit-I1}]
For most of the proof we allow the extreme cases $\pi_1 = 0$ and $\pi_1 = \pi$.

\textbf{Step 1.} We show that for $\lambda \in \P$ the ideal $\mathfrak a_1^{i_0}(\lambda)$ kills the $\chi_\lambda^{-1}$-eigenspace of $\gr_\m(\pi_1^\vee)_0 \cong (\pi_1^{I_1})^\vee$ inside $\gr_\m(\pi_1^\vee)$.

By Corollary \ref{cor:conj1} we may assume that $|J_\lambda| \le i_0$.
We already know that the $\chi_\lambda^{-1}$-eigenspace is killed \ by \ $\mathfrak a(\lambda)$ \ (\cite[Thm.\ 3.67]{BHHMS2}, \cite[Cor.\ 8.12]{HuWang2}), \ so \ let \ us \ take \ a \ monomial $\prod_{j \in J_1'} y_j \prod_{j \in J_2'} z_j$ with $J_1' \subset J_1$, $J_2' \subset J_2$, $|J'_1|+|J'_2| = i_0+1-|J_\lambda|$ (in particular of degree $> 0$).

Define $\lambda' \in \P^\ss$ by letting $\lambda'_j(x_j) \defeq  \lambda_j(x_j)-2$ if $j \in J_1'$, $\lambda'_j(x_j) \defeq  \lambda_j(x_j)+2$ if $j \in J_2'$, and $\lambda'_j(x_j) \defeq  \lambda_j(x_j)$ otherwise.
Then $\lambda' \in \P^\ss\setminus\P$ using the definition of $\P^\ss$ and (\ref{eq:P}). 
Moreover, $|J_{\lambda'}| = |J_\lambda|+(i_0+1-|J_\lambda|) = i_0+1$.
By Proposition \ref{prop:conj5} we deduce (on the dual side) that the monomial $\prod_{j \in J_1'} y_j \prod_{j \in J_2'} z_j$ kills the $\chi_\lambda^{-1}$-eigenspace of $\gr_\m(\pi_1^\vee)_0$ inside $\gr_\m(\pi_1^\vee)$.

\textbf{Step 2.} 
Define $N_1^{i_0} \defeq  \bigoplus_{\lambda\in\mathscr{P}}\chi_{\lambda}^{-1}\otimes \o R/\mathfrak{a}_1^{i_0}(\lambda)$ and let $N_2^{i_0}$ be the kernel of the natural map $N \onto N_1^{i_0}$.
Consider the induced short exact sequence
\[0\ra \gr_{F}(\pi_2^{\vee})\ra \gr_{\m}(\pi^{\vee})\ra \gr_{\m}(\pi_1^{\vee})\ra0,\]
where $F$ is the filtration on $\pi_2^{\vee}$ induced from the $\m$-adic filtration on $\pi^{\vee}$. 
By Step~1 the morphism $N \congto \gr_{\m}(\pi^{\vee}) \onto \gr_{\m}(\pi_1^{\vee})$ factors through $N_1^{i_0}$,  hence we get an induced commutative diagram
\begin{equation*}
  \xymatrix{0\ar[r]& \gr_{F}(\pi_2^{\vee})\ar[r]& \gr_{\m}(\pi^{\vee})\ar[r]& \gr_{\m}(\pi_1^{\vee})\ar[r]&0 \\
  0\ar[r]& N_2^{i_0}\ar[r]\ar@{^{(}->}[u]& N\ar[r]\ar[u]^\cong & N_1^{i_0}\ar[r]\ar@{->>}[u]& 0}
\end{equation*}
with injective (resp.\ surjective) vertical map on the left (resp.\ right).
Thus
\begin{equation}\label{eq:cycles-ineq}\mathcal{Z}(N_1^{i_0})\geq \mathcal{Z}(\gr_{\fm}(\pi_1^{\vee})),\ \ \mathcal{Z}(N_2^{i_0})\leq \mathcal{Z}(\gr_F(\pi_2^{\vee})).\end{equation}

\textbf{Step 3.} We show that $N_1^{i_0}$ and $N_2^{i_0}$ are Cohen--Macaulay of grade $2f$, or zero.

First note that $j_{\gr(\Lambda)}(N_1^{i_0}) \ge j_{\gr(\Lambda)}(N) = 2f$.
By Corollary~\ref{cor:CM-module-kunneth} and \cite[Lemma 3.65]{BHHMS2}, $N_1^{i_0}$ is a Cohen--Macaulay $\gr(\Lambda)$-module of grade $2f$, or zero.
(We may omit the terms in the direct sum with $|J_\lambda| > i_0$, as they vanish.)
As $N_2^{i_0} = \ker(N \onto N_1^{i_0})$ and both $N$ and $N_1^{i_0}$ are Cohen--Macaulay of grade $2f$, or zero, so is $N_2^{i_0}$.

\textbf{Step 4.} We show that $\gr_\m(\pi_1^{\vee})$ and $\gr_F(\pi_2^{\vee})$ are Cohen--Macaulay of grade $2f$.

{By assumption \ref{it:assum-iii}}  we have $\EE_{\Lambda}^{2f}(\pi^\vee) \cong \pi^\vee \otimes \det(\rhobar)\omega^{-1}$ as $\GL_2(K)$-representations. 
As in the proof of \cite[Prop.\ 3.87(iii)]{BHHMS2} we may construct a subrepresentation $\widetilde{\pi}_2\subset \pi$ such that $\mathcal{Z}(\gr(\wt\pi_2^{\vee}))=\mathcal{Z}(\gr(\pi_2^{\vee}))$ (with respect to any good filtrations).
By \cite[Prop.\ 3.87(i)]{BHHMS2} and the exactness of $D_{\xi}^{\vee}$ we have
\begin{equation*}
  \dim_{\F\ppar{X}}D_{\xi}^{\vee}(\wt\pi_2)=\dim_{\F\ppar{X}}D_{\xi}^{\vee}(\pi_2)=\dim_{\F\ppar{X}}D_{\xi}^{\vee}(\pi)-\dim_{\F\ppar{X}}D_{\xi}^{\vee}(\pi_1).
\end{equation*}
By Corollary~\ref{cor:yitong} we deduce that $i_0(\wt\pi_2) = f-1-i_0(\pi_1)$.

In particular, noting that $N_1^{i_0}$ only depends on $i_0 = i_0(\pi_1)$, we deduce by~(\ref{eq:cycles-ineq}) applied to the subrepresentation $\wt\pi_2$ that $\mathcal{Z}(N_1^{f-1-i_0})\geq \mathcal{Z}(\gr_{\fm}(\wt\pi_2^{\vee}))$.
Hence
\begin{equation}\label{eq:cycles-ineq2}
  \mathcal{Z}(N_1^{f-1-i_0})\geq \mathcal{Z}(\gr(\wt\pi_2^{\vee})) = \mathcal{Z}(\gr(\pi_2^{\vee})) \ge \mathcal{Z}(N_2^{i_0}) = \mathcal{Z}(N)-\mathcal{Z}(N_1^{i_0}).
\end{equation}
We claim that equality holds, and it suffices to show that $m(N_1^{i_0})+m(N_1^{f-1-i_0}) = m(N)$.

As $N_1^{i_0} = \bigoplus_{\lambda\in\mathscr{P}}\chi_{\lambda}^{-1}\otimes \o R/\mathfrak a_1^{i_0}(\lambda)$ and the involution $\lambda \mapsto \lambda^*$ preserves (i.e.\ induces a bijection on) $\P$ by \cite[Lemma 3.63(i)]{BHHMS2}, it suffices to show that
\begin{equation}
  m(\o R/\mathfrak a_1^{i_0}(\lambda))+m(\o R/\mathfrak a_1^{f-1-i_0}(\lambda^*)) = m(\o R/\mathfrak a(\lambda))\ \ \text{for each $\lambda \in \P$}.\label{eq:mult-add}
\end{equation}
Fix now $\lambda \in \P$.
Recall that $\mathfrak a_1^{i_0}(\lambda) = I(J_1,J_2,i_0+1-|J_\lambda|,\un t)$, where $J_1 = \{ j \in J_{\rhobar}^c : \lambda_j(x_j) = p-1-x_j \}$, $J_2 = \{ j \in J_{\rhobar}^c : \lambda_j(x_j) = x_j \}$, and $t_j \in \{y_j,z_j,y_jz_j\}$ is defined in~(\ref{eq:id:al}).
Let $J \defeq  J_1 \sqcup J_2$.
By Lemma~\ref{lem:J-sets-add-up} we have $|J_\lambda|+|J_{\lambda^*}|+|J| = f$ (and $J$ is unchanged when $\lambda$ is replaced by $\lambda^*$).
By Lemma~\ref{lem:total-mult} we have
\begin{equation}\label{eq:mult-a1}
  m(\o R/\mathfrak a_1^{i_0}(\lambda)) = 2^{|\{ j : \lambda_j(x_j) \in \{x_j+1,p-2-x_j\}\}|} \left(\sum_{i < i_0+1-|J_\lambda|} \binom{|J|}i\right),
\end{equation}
noting that $\{ j \in J^c : t_j = y_jz_j \} = \{ j : \lambda_j(x_j) \in \{x_j+1,p-2-x_j\}\}$.
In particular, taking $i_0 = f$ and noting that $|J_\lambda|+|J| \le f$ by Lemma~\ref{lem:J-sets-add-up} (or by arguing directly) we have,
\begin{equation}\label{eq:mult-a}
  m(\o R/\mathfrak a(\lambda)) = 2^{|\{ j : \lambda_j(x_j) \in \{x_j+1,p-2-x_j\}\}|} \cdot 2^{|J|}.
\end{equation}
From~(\ref{eq:mult-a1}) and the definition of $\lambda \mapsto \lambda^*$ in \cite[Def.\ 3.62]{BHHMS2} we obtain
\begin{equation*}
  m(\o R/\mathfrak a_1^{f-1-i_0}(\lambda^*)) = 2^{|\{ j : \lambda_j(x_j) \in \{x_j+1,p-2-x_j\}\}|} \left(\sum_{i < f-i_0-|J_{\lambda^*}|} \binom{|J|}i\right),
\end{equation*}
By Lemma~\ref{lem:J-sets-add-up},
\begin{equation*}
  \sum_{i < f-i_0-|J_{\lambda^*}|} \binom{|J|}i = \sum_{i < |J|+|J_\lambda|-i_0} \binom{|J|}i = \sum_{i > i_0-|J_\lambda|} \binom{|J|}i,
\end{equation*}
and we deduce~(\ref{eq:mult-add}) and hence equality in~(\ref{eq:cycles-ineq2}) and~(\ref{eq:cycles-ineq}).

Since $N_1^{i_0}$ is Cohen--Macaulay, hence pure (by combining Prop.~3.5(v)(a) and Prop.~3.9(i) in \cite{Venjakob}), or since $N_1^{i_0} = 0$, any nonzero submodule of $N_1^{i_0}$ has a nonzero cycle.
Hence the surjection $N_1^{i_0}\twoheadrightarrow \gr_{\m}(\pi_1^{\vee})$ must be an isomorphism and consequently $\gr_{F}(\pi_2^{\vee})\cong N_2^{i_0}$ by Step~2.
We finally assume that $\pi_1 \ne 0$ and $\pi_1 \ne \pi$.
Then the isomorphisms we just established show that $N_1^{i_0} \ne 0$ and $N_2^{i_0} \ne 0$, so both $N_1^{i_0}$ and $N_2^{i_0}$ are Cohen--Macaulay by Step 3.
Hence $\pi_1^\vee$ and $\pi_2^\vee$ are Cohen--Macaulay, because if a finitely generated $\Lambda$-module $M$ admits a good filtration such that the associated graded module is Cohen--Macaulay, then $M$ itself is Cohen--Macaulay by \cite[Prop.~III.2.2.4]{LiOy}.
\end{proof}

\begin{cor}\label{cor:gr-pi1}
Assume that $\brho$ is  $\max\{9,2f+1\}$-generic. 
  Let $i_0 = i_0(\pi_1)$ with $-1 \le i_0 \le f$ be as in Theorem \ref{thm:conj2}.
  Then
  \begin{equation*}
    \gr_\m(\pi_1^{\vee}) \cong \bigoplus_{\lambda\in\mathscr{P}}\chi_{\lambda}^{-1}\otimes \frac{\o R}{\mathfrak{a}_1^{i_0}(\lambda)}
  \end{equation*}
  and
  \begin{equation*}
    \gr_F(\pi_2^{\vee}) \cong \bigoplus_{\lambda\in\mathscr{P}}\chi_{\lambda}^{-1}\otimes \frac{\mathfrak{a}_1^{i_0}(\lambda)}{\mathfrak{a}(\lambda)},
  \end{equation*}
  where $F$ denotes the filtration induced from $\pi^\vee$.
\end{cor}

\begin{cor}\label{cor:subquot}
Assume that $\brho$ is $\max\{9,2f+1\}$-generic. 
  Suppose that $\pi' = \pi_1'/\pi_1$ is any nonzero subquotient of $\pi$, where $\pi_1 \subsetneq \pi_1' \subset \pi$.
  Let $i_0 \defeq  i_0(\pi_1)$, $i_0' \defeq  i_0(\pi_1')$, so $-1 \le i_0 < i_0' \le f$.
  Let $F$ denote the subquotient filtration on $\pi'^\vee$ induced from the $\m$-adic filtration on $\pi^\vee$.
  Then
  \begin{equation}\label{eq:gr-subquot}
    \gr_F(\pi'^\vee) \cong \bigoplus_{\lambda\in\mathscr{P}}\chi_{\lambda}^{-1}\otimes \frac{\mathfrak{a}_1^{i_0}(\lambda)}{\mathfrak{a}_1^{i_0'}(\lambda)}.
  \end{equation}
  Moreover, $\gr_F(\pi'^\vee)$ (resp.\ $\pi'^\vee$) is Cohen--Macaulay of grade $2f$.
\end{cor}

\begin{proof}
  The exact sequence $0 \to \pi'^\vee \to \pi_1'^\vee \to \pi_1^\vee \to 0$ of $\Lambda$-modules gives rise to an exact sequence
  \begin{equation*}
    0 \to \gr_F(\pi'^\vee) \to \gr(\pi_1'^\vee) \to \gr(\pi_1^\vee) \to 0.
  \end{equation*}
  The second map is identified with the natural map $\bigoplus_{\lambda\in\mathscr{P}}\chi_{\lambda}^{-1}\otimes \o R/\mathfrak{a}_1^{i_0'}(\lambda) \to \bigoplus_{\lambda\in\mathscr{P}}\chi_{\lambda}^{-1}\otimes \o R/\mathfrak{a}_1^{i_0}(\lambda)$
  by Corollary~\ref{cor:gr-pi1} (cf.\ Step 2 of the proof of Proposition~\ref{prop:nonsplit-I1}).
  Formula~(\ref{eq:gr-subquot}) follows.
  As $\gr(\pi_1'^\vee)$, $\gr(\pi_1^\vee)$ (resp.\ $\pi_1'^\vee$, $\pi_1^\vee$) are Cohen--Macaulay of grade $2f$ by Proposition~\ref{prop:nonsplit-I1}, so is $\gr_F(\pi'^\vee)$ (resp.\ $\pi'^\vee$).
  (If $0 \to M' \to M \to M'' \to 0$ is an exact sequence of $\Lambda$-modules (resp.\ $\gr(\Lambda)$-modules) and $M$ and $M''$ are Cohen--Macaulay of the same grade $j$, then $M'$ is zero or Cohen--Macaulay of grade $j$ by \cite[Cor.~III.2.1.6]{LiOy}.) 
\end{proof}

For $0\leq j\leq f$, let $\mathscr{P}^{\rm ss}_{j}$ (resp.\ $\mathscr{P}_{j}$) denote the subset of $\lambda\in\mathscr{P}^{\rm ss}$ (resp.\ $\lambda\in\mathscr{P})$ with $|J_{\lambda}|=j$. 
\begin{cor}\label{cor:subquot-F}
Keep the assumptions and notation in Corollary \ref{cor:subquot}.  There is an $H$-equiva\-riant isomorphism 
 \[\F\otimes_{\gr(\Lambda)}\gr_F(\pi'^{\vee})\cong \bigoplus_\lambda\chi_{\lambda}^{-1},\]
where $\lambda$ runs through all $\lambda\in \mathscr{P}^{\rm ss}_{i_0+1}\cup \big(\bigcup_{i_0+2\leq j\leq i_0'}\mathscr{P}_j\big)$.
\end{cor}
\begin{proof}
We first look at $X_{i_0,i_0'}(\lambda)\defeq \F\otimes_{\overline{R}}\fa_1^{i_0}(\lambda)/\fa_1^{i_0'}(\lambda)$ for $\lambda\in \mathscr{P}$.
If $|J_{\lambda}|>i_0'$, then $\fa_{1}^{i_0}(\lambda)=\fa_{1}^{i_0'}(\lambda)=\overline{R}$, so $X_{i_0,i_0'}(\lambda)=0$; if $i_0<|J_{\lambda}|\leq i_0'$, then $\fa_1^{i_0}(\lambda)=\overline{R}$ while $\fa_1^{i_0'}(\lambda)\subset \fm_{\overline{R}}$ (the unique maximal graded ideal in $\overline{R}$), so $X_{i_0,i_0'}(\lambda)\cong \F$. 
Finally suppose $|J_{\lambda}|\leq i_0$, and recall $\fa_1^{i}(\lambda) = I(J_1,J_2,i+1-|J_\lambda|)+\fa(\lambda)$, where $J_1, J_2$ are as in \eqref{eq:fa-1}.
Hence $I(J_1,J_2,i_0'+1-|J_\lambda|) \subset \fm_{\overline{R}} I(J_1,J_2,i_0+1-|J_\lambda|)$ and so
\[X_{i_0,i_0'}(\lambda)\cong\F\otimes_{\overline{R}} I(J_1,J_2,i_0+1-|J_\lambda|) \cong \bigoplus_{(J_1',J_2')} \F\big( \overline{\prod_{j\in J_1'}y_j\prod_{j\in J_2'}z_j}\big), \]
where $(J_1',J_2')$ runs through all pairs with $J_1'\subset J_1$, $J_2'\subset J_2$, $|J_1'|+|J_2'|=i_0+1-|J_{\lambda}|$. 
Step 1 of the proof of Proposition \ref{prop:nonsplit-I1} shows that to each pair $(J_1',J_2')$ as above, one can associate an element $\lambda'\in \mathscr{P}^{\rm ss}\setminus\P$  with $|J_{\lambda'}|=i_0+1$, such that $\chi_{\lambda}^{-1}\prod_{j\in J_1'}\alpha_j\prod_{j\in J_2'}\alpha_j^{-1}=\chi_{\lambda'}^{-1}$.  Conversely, by the construction in \eqref{eq:mu_j}, any element $\lambda'\in \mathscr{P}^{\ss}\setminus\P$ with $|J_{\lambda'}|=i_0+1$ arises in this way and $\lambda'$ uniquely determines $\lambda\in\mathscr{P}$ and $J_1'$, $J_2'$. 
The result follows from this combined with Corollary \ref{cor:subquot}.
\end{proof}

\begin{thm}\label{thm:fin-length-nonsplit}
  Assume that $\brho$ is $\max\{9,2f+1\}$-generic. 
  \begin{enumerate}
  \item 
    Any subquotient of $\pi$ is generated by its $K_1$-invariants.
  \item 
    The representation $\pi$ is uniserial of length at most $f+1$.
    More precisely, suppose that $\pi_1$, $\pi_1'$ are any subrepresentations of $\pi$.  Then the following are equivalent:
    \begin{enumerate}
    \item $\pi_1 \subset \pi_1'$;
    \item $\pi_1^{K_1} \subset \pi_1'^{K_1}$;
    \item $i_0(\pi_1) \le i_0(\pi_1')$;
    \item $\dim_{\F\ppar{X}} D_{\xi}^{\vee}(\pi_1)\le\dim_{\F\ppar{X}} D_{\xi}^{\vee}(\pi_1')$. 
  \end{enumerate}
  \item 
    If $\pi'$ is any nonzero subquotient of $\pi$, then $D_{\xi}^{\vee}(\pi')$ is nonzero.
  \end{enumerate}
\end{thm}

\begin{proof}
  (i) 
  The quotient of any $\GL_2(K)$-representation generated by its $K_1$-invariants is generated by its $K_1$-invariants, hence it suffices to consider the case of a subrepresentation $\pi_1 \subset \pi$.
  Let $\pi_1' \defeq  \ang{\GL_2(K)\cdot \pi_1^{K_1}}$ be the subrepresentation of $\pi_1$ generated by $\pi_1^{K_1}$, so $\pi_1'^{K_1} = \pi_1^{K_1}$.
  By Theorem \ref{thm:conj2} we have $i_0(\pi_1') = i_0(\pi_1)$.
  By the proof of Proposition~\ref{prop:nonsplit-I1} the natural map $\gr_\m(\pi_1^\vee) \onto \gr_\m(\pi_1'^\vee)$ is an isomorphism (consider the diagram in Step 2), so $\pi_1' = \pi_1$.
  
  (ii) To show the equivalence, we note that (a)$\Rightarrow$(b) and the converse holds by part (i), (b)$\Leftrightarrow$(c) by Theorem \ref{thm:conj2}, and (c)$\Leftrightarrow$(d) by Corollary~\ref{cor:yitong}.
  Finally, condition (c) implies that $\pi$ is uniserial of length at most $f+1$.

  (iii) Write $\pi' = \pi_1'/\pi_1$ for some subrepresentations $\pi_1 \subsetneq \pi_1' \subset \pi$.
  By part (ii) we deduce that $\dim_{\F\ppar{X}} D_{\xi}^{\vee}(\pi_1) < \dim_{\F\ppar{X}} D_{\xi}^{\vee}(\pi_1')$. 
  We conclude by the exactness of $D_{\xi}^{\vee}$.
\end{proof}
\begin{rem}\label{rem:fin-length-nonsplit}
The statement of Theorem \ref{thm:fin-length-nonsplit}(i) fails if we replace $K_1$ by $I_1$, already when $K=\Qp$, by \cite[Thm.\ 20.3(i)]{BP} (see also \cite[Thm.\ 1.1]{morra-AA} for a different proof).
\end{rem}

\begin{cor}\label{cor:pi-mult-free1}
Assume that $\brho$ is $\max\{9,2f+1\}$-generic. 
  The $\GL_2(K)$-representation $\pi$ is multiplicity free (of length $\le f+1$).
\end{cor}

\begin{proof}
Let $\pi'$ be any nonzero subquotient of $\pi$ and $F$ be the subquotient filtration on $\pi'^\vee$ induced from the $\m$-adic filtration on $\pi^\vee$. As in the proof of Proposition \ref{prop:Tor-inj}, by replacing $\gr_{\m}(\pi^{\vee})$ by $\gr_{F}(\pi'^{\vee})$ we obtain a spectral sequence $E_{i}^r\Longrightarrow \Tor_i^{\Lambda}(\F,\pi'^{\vee})$ with $E_i^1=\Tor_i^{\gr(\Lambda)}(\F,\gr_{F}(\pi'^{\vee}))$ for $i\geq 0$. In particular, we get a surjective graded morphism compatible with $H$-action
\[E^1_0 = \F\otimes_{\gr(\Lambda)}\gr_F(\pi'^{\vee})\twoheadrightarrow \gr(\F\otimes_{\Lambda}\pi'^{\vee}) = E^\infty_0,\]
hence an inclusion 
\begin{equation}\label{eq:multifree-JH}\JH(\F\otimes_{\Lambda}\pi'^{\vee})=\JH(\gr(\F\otimes_{\Lambda}\pi'^{\vee}))\subset \JH(\F\otimes_{\gr(\Lambda)}\gr_F(\pi'^{\vee})).\end{equation}

By Theorem~\ref{thm:fin-length-nonsplit}(ii) there exists a unique composition series $0=\pi_0\subsetneq \pi_1\subsetneq\cdots\subsetneq\pi_\ell=\pi$ of the $\GL_2(K)$-representation $\pi$, and moreover $-1=i_0(\pi_0)<i_0(\pi_1)<\cdots<i_0(\pi_\ell)=f$.
Corollary \ref{cor:subquot} implies that 
\[\gr_F((\pi_j/\pi_{j-1})^{\vee})\cong \bigoplus_{\lambda\in\mathscr{P}}\chi_{\lambda}^{-1}\otimes \frac{\fa_1^{i_0(\pi_{j-1})}(\lambda)}{\fa_1^{i_0(\pi_{j})}(\lambda)}.\]
As $\F\otimes_{\Lambda}(\pi_j/\pi_{j-1})^{\vee}$ is dual to $(\pi_j/\pi_{j-1})^{I_1}$, we deduce from Corollary \ref{cor:subquot-F} and \eqref{eq:multifree-JH} that the sets $\JH((\pi_j/\pi_{j-1})^{I_1})$ (of $H$-representations)  are disjoint for $1\leq j\leq \ell$, which proves the multiplicity freeness of $\pi$.
\end{proof}

\appendix
\section{Appendix: canonical filtrations on Tor and Ext groups}
\label{sec:append-canon-filtr}

We prove useful lemmas on the canonical filtration on Tor and Ext groups of filtered
modules.

Let $R$ be a filtered ring {(not necessarily the ring $R$ of \S~\ref{sec:notation}!)}, and let $\wt R$ be its Rees ring (see \cite[Def.~I.4.3.5]{LiOy} or \cite[\S~4.1]{BE90}). 
Then $\wt R$ is a graded ring, and we have a functor $N\mapsto \wt N$ from the category of filtered $R$-modules to the category of graded $\wt R$-modules (see~\cite[\S~I.4.3]{LiOy}).

Letting $X\defeq 1\in \wt{R}_1$ be the canonical homogeneous element of degree $1$ we have $\wt R=\bigoplus_{n\in\Z}(F_nR)X^n$ (\cite[Def.~I.4.3.6(b)]{LiOy}). 
We thus define the \emph{dehomogenization functor} $\cE$ from the category of graded $\wt R$-modules to the category of filtered $R$-modules as follows: for a graded $\wt R$-module $W=\bigoplus_{n\in\Z}W_n$ we set $\cE(W)\defeq W/(1-X)W$, with filtration defined by
\[
F_n(\cE(W))\defeq (W_n+(1-X)W)/(1-X)W
\]
for \ any \ $n\in\Z$. \ By \ \cite[Prop.~I.4.3.7(5)]{LiOy} \ the \ functor \ $\cE$ \ is \ exact, \ and \ by \cite[Prop.~I.4.3.7(2),\,(3)]{LiOy} it induces an equivalence when restricted to the full subcategory of $X$-torsion-free graded $\wt R$-modules, with quasi-inverse $N\mapsto \wt N$.
In particular, $\cE(\wt N)\cong N$ for any filtered $R$-module $N$.

\begin{lem1}\label{lem:strict}
  Suppose that $R$ is a filtered ring and that $N_1 \to N_2 \to N_3$ is an exact sequence of graded $\wt{R}$-modules.
  If $N_3$ is $X$-torsion-free, then the sequence $\mathcal{E}(N_1) \to \mathcal{E}(N_2) \to \mathcal{E}(N_3)$ of filtered $R$-modules is exact and the first morphism is strict.
  In particular, taking $N_3 = 0$: if $N_1 \to N_2$ is surjective, then $\mathcal{E}(N_1) \to \mathcal{E}(N_2)$ is a strict surjection.
\end{lem1}

\begin{proof}
  As recorded above (cf.\ \cite[Prop.\ 5.3]{BE90}), the Rees module $\wt{\mathcal{E}(N)}$ is identified with the largest $X$-torsion-free quotient of $N$.
  As $N_3$ is $X$-torsion-free, a diagram chase shows that the sequence $\wt{\mathcal{E}(N_1)} \to \wt{\mathcal{E}(N_2)} \to \wt{\mathcal{E}(N_3)}$ is exact.
  The result follows from \cite[Prop.\ I.4.3.8(2)]{LiOy}.
\end{proof}

Suppose now that $R$, $S$ are filtered rings such that the Rees ring $\wt S$ is noetherian, and let $N$ be any filtered $(S,R)$-bimodule, i.e.\ equipped with a filtration $F_nN$ ($n \in \Z$) such that with this filtration $N$ is both a filtered left $S$-module and a filtered right $R$-module (cf.\ \cite[Def.\ I.2.2]{LiOy}).
Then the notions in the previous paragraphs extend to filtered and graded bimodules, and we have a dehomogenization functor $\cE$ from graded $(\wt S,\wt R)$-bimodules to filtered $(S,R)$-bimodules (in particular, $\cE(\wt N) \cong N$ as filtered $(S,R)$-bimodules).

Following \cite[\S~5]{BE90} in the case of $\Ext_R^i(-,R)$, we now explain that $\Tor_i^R(N,R)$ is canonically and functorially a filtered $S$-module.
We also establish some basic properties of this canonical filtration.

If $W$ is any graded $\wt R$-module, then
\begin{equation}\label{eq:dehomog-tensor}
  \cE(\wt N \otimes_{\wt R} W) \cong S \otimes_{\wt S} \wt N \otimes_{\wt R} W \cong N \otimes_{\wt R} W\cong N \otimes_R \cE(W),
\end{equation}
where we used that $X = 1 \in \wt S_1$ (resp.\ $\wt R_1$) acts the same on the left and right of $\wt N$.
Here, $\wt N \otimes_{\wt R} W$ is a graded $\wt S$-module (cf.\ the discussion at the end of \S~\ref{sec:prelim}), $N \otimes_R \cE(W)$ is a filtered $S$-module (cf.\ \cite[\S~I.6]{LiOy}) and~\eqref{eq:dehomog-tensor} is easily checked to be an isomorphism of filtered $S$-modules.

Forgetting filtrations for a moment, as $\cE$ is exact, we have a natural isomorphism
\begin{equation}\label{eq:dehomog-tor}
  \cE(\Tor_i^{\wt R} (\wt N, W)) \cong \Tor_i^R (N, \cE(W))
\end{equation}
as $S$-modules for all $i \ge 0$.
As $\Tor_i^{\wt R} (\wt N, W)$ is a graded $\wt S$-module, the isomorphism induces a canonical and functorial filtration on $\Tor_i^R (N, \cE(W))$.
In particular, if $W = \wt M$ for a filtered $R$-module $M$ we obtain a canonical and functorial filtration on the $S$-module $\Tor_i^R (N, M)$.

\begin{lem1}\label{lem:fil-tor-seq}
  If $0 \to M_1 \to M_2 \to M_3 \to 0$ is a strict exact sequence of filtered $R$-modules, then the long exact sequence
  \[ \cdots \to \Tor_1^R(N,M_2) \to \Tor_1^R(N,M_3) \to N \otimes_R M_1 \to N \otimes_R M_2 \to N \otimes_R M_3 \to 0\]
  of $S$-modules respects filtrations.
\end{lem1}

The reason is that by strictness the induced sequence $0 \to \wt M_1 \to \wt M_2 \to \wt M_3 \to 0$ is still exact.

\begin{lem1}\label{lem:good-fil-tor}
  Suppose that $\wt R$ is noetherian, and suppose that $N$ has the property that as a filtered $S$-module its filtration is good.
  Then for any filtered $R$-module $M$ equipped with a good filtration, the canonical filtration on each $\Tor_i^R (N, M)$ is good.
\end{lem1}

Note that the condition on $N$ is equivalent to $\wt N$ being a finitely generated $\wt S$-module \cite[Prop.\ I.5.4(1)]{LiOy}.

\begin{proof}
  From the isomorphism~\eqref{eq:dehomog-tor} with $W = \wt M$ and \cite[Prop.\ I.4.3.7(2),\,(3)]{LiOy} it follows that the Rees module of $\Tor_i^R (N, M)$ is the largest $X$-torsion-free quotient of $\Tor_i^{\wt R} (\wt N, \wt M)$.
  Hence by \cite[Prop.\ I.5.4(1)]{LiOy} it suffices to show that $\Tor_i^{\wt R} (\wt N, \wt M)$ is a finitely generated $\wt S$-module for all $i$.
  By picking a gr-free resolution of $\wt M$ whose terms are moreover finitely generated (using $\wt R$ noetherian) and since $\wt S$ is noetherian, we reduce to the case $i = 0$ and $\wt M$ gr-free, in which case the claim follows from the assumption on $N$.
\end{proof}

\begin{lem1}\label{lem:fil-tor-by-resolution}
  Suppose that $\cdots \to F_1 \to F_0 \to M \to 0$ is a strict exact sequence with $F_i$ filt-free for all $i$ {(see the beginning of \S~\ref{sec:prelim} for filt-free)}.
  Then the canonical filtration on $\Tor_i^R (N, M)$ coincides with the subquotient filtration on the $i$-th homology of the complex of filtered $S$-modules $N \otimes_R F_\bullet$ 
  (each carrying the tensor product filtration).
\end{lem1}

\begin{proof}
  By strictness, the sequence $\cdots \to \wt F_1 \to \wt F_0 \to \wt M \to 0$ of graded $\wt R$-modules is exact.
  Hence $\Tor_i^{\wt R} (\wt N, \wt M)$ is isomorphic to the $i$-th homology of the complex $\wt N \otimes_{\wt R} \wt F_\bullet$ of graded $\wt S$-modules.
  Let $C_i \defeq  \wt N \otimes_{\wt R} \wt F_i$, so $\cE(C_i) \cong N \otimes_R F_i$ with the tensor product filtration.
  Let $Z_i$ (resp.\ $B_{i-1}$) denote the kernel (resp.\ the image) of $C_i \to C_{i-1}$, and let $H_i \defeq  Z_i/B_i$.
  By exactness of $\cE$ we have $\cE(H_i) \cong \cE(Z_i)/\cE(B_i)$ as $S$-modules and we need to show that it carries the subquotient topology inside $\cE(C_i)$, i.e.\ that the maps $\cE(Z_i) \into \cE(C_i)$ and $\cE(Z_i) \onto \cE(H_i)$ are both strict.
  As $\wt F_i$ is gr-free, it follows that $C_i$ is $X$-torsion-free, and hence so are $B_i$ and $Z_i$.
  From Lemma~\ref{lem:strict} we deduce that the sequences $0 \to \cE(Z_i) \to \cE(C_i) \to \cE(B_{i-1}) \to 0$ and $\cE(Z_i) \to \cE(H_i) \to 0$ are strict exact.
\end{proof}

Similarly, if $N$ is a filtered $(R,S)$-bimodule, and $M$ is an $R$-module with a \emph{good} filtration, then the right $S$-module $\Ext^i_R(M,N)$ is canonically and functorially a filtered $S$-module.
The reason is that for any finitely generated graded $\wt R$-module $W$ we have a natural isomorphism of filtered right $S$-modules
\begin{equation*}
  \cE(\Ext^i_{\wt R}(W,\wt N)) \cong \Ext^i_R(\cE(W),N)
\end{equation*}
and that $\Hom_R(W,-)$ is naturally graded \cite[Lemma I.4.1.1]{LiOy} and $\Hom_R(\cE(W),N)$ is naturally filtered \cite[Prop.\ I.6.6]{LiOy}, as $W$ is finitely generated.
The analogues of Lemmas~\ref{lem:fil-tor-seq}, \ref{lem:good-fil-tor}, and ~\ref{lem:fil-tor-by-resolution} hold, with the analogous proofs, provided in the first lemma all $M_i$ carry good filtrations and in the last lemma all $F_i$ are filt-free of finite rank.

We finally specialize to the case where $R=S=\Lambda$ and $M$ is a finitely generated (left) $\Lambda$-module equipped with a good filtration.
In particular the right $\Lambda$-module $\EE^i_\Lambda(M)= \Ext^i_\Lambda(M,\Lambda)$ carries a canonical and functorial filtration.
\begin{lem1}\label{lem:E-strict}
  Suppose that $0 \to M_1 \to M_2 \to M_3 \to 0$ is a strict exact sequence of finitely generated filtered $\Lambda$-modules.
  Suppose that the filtration on $M_2$ (and hence on $M_1$, $M_3$) is good and that $j \defeq  j_\Lambda(M_2)$.
  Then the induced morphism $0 \to \EE^j_\Lambda(M_3) \to \EE^j_\Lambda(M_2)$ is strict.
\end{lem1}

\begin{proof}
  By strictness we get $0 \to \wt M_1 \to \wt M_2 \to \wt M_3 \to 0$ of graded $\wt \Lambda$-modules, with $j_{\wt\Lambda}(\wt M_2) = j$ by \cite[\S~III.2.5]{LiOy}.
  Hence we obtain the exact sequence
  \[0 \to \EE^j_{\wt \Lambda}(\wt M_3) \to \EE^j_{\wt \Lambda}(\wt M_2) \to \EE^j_{\wt \Lambda}(\wt M_1)\] of graded right $\wt \Lambda$-modules.
  Each $\EE^j_{\wt \Lambda}(\wt M_i)$ is $X$-torsion-free by \cite[Lemma 5.11]{BE90}, 
    The result follows from Lemma~\ref{lem:strict}.
\end{proof}

\bibliography{Biblio}
\bibliographystyle{amsalpha}

\end{document}